\documentclass[11pt]{amsart}
\usepackage{fullpage, amsmath, amsthm,amsfonts,url,amssymb,stmaryrd, mathrsfs}
\usepackage{mathrsfs}
\usepackage{amssymb,amscd}
\usepackage{color}
\usepackage{chemarrow}
\usepackage{pictexwd,dcpic}
\usepackage{extarrows}
\usepackage[english,french]{babel}
\usepackage[colorlinks,linkcolor=red,anchorcolor=green,citecolor=blue]{hyperref}
\usepackage{enumitem}

\newtheorem{theorem}{Theorem}[section]
\newtheorem{corollary}[theorem]{Corollary}
\newtheorem{lemma}[theorem]{Lemma}
\newtheorem{proposition}[theorem]{Proposition}

\newtheorem{definition}[theorem]{Definition}
\newtheorem{hypothesis}[theorem]{Hypothesis}
\newtheorem{remark}[theorem]{Remark}

\newtheorem{example}[theorem]{Example}
\newtheorem{conjecture}[theorem]{Conjecture}

\newcommand{\hooklongrightarrow}{\lhook\joinrel\longrightarrow}
\newcommand{\twoheadlongrightarrow}{\relbar\joinrel\twoheadrightarrow}

\newcommand{\ra}{\rightarrow}
\newcommand{\lra}{\longrightarrow}

\newcommand{\ul}{\underline}

\newcommand{\bA}{\mathbb A}
\newcommand{\bC}{\mathbb C}

\newcommand{\bG}{\mathbb G}

\newcommand{\Q}{\mathbb Q}
\newcommand{\bR}{\mathbb R}

\newcommand{\bT}{\mathbb T}
\newcommand{\Z}{\mathbb Z}

\newcommand{\bU}{\mathbb U}

\newcommand{\cN}{\mathcal N}

\newcommand{\co}{\mathcal O}

\newcommand{\cR}{\mathcal R}
\newcommand{\cH}{\mathcal H}
\newcommand{\cC}{\mathcal C}
\newcommand{\cS}{\mathcal S}
\newcommand{\cD}{\mathcal D}

\newcommand{\cW}{\mathcal W}
\newcommand{\cT}{\mathcal T}
\newcommand{\cM}{\mathcal M}
\newcommand{\cF}{\mathcal F}

\newcommand{\cE}{\mathcal E}

\newcommand{\cU}{\mathcal U}
\newcommand{\cZ}{\mathcal Z}

\newcommand{\fn}{\mathfrak n}
\newcommand{\fh}{\mathfrak h}
\newcommand{\fm}{\mathfrak{m}}
\newcommand{\ub}{\mathfrak b}

\newcommand{\fp}{\mathfrak p}

\newcommand{\ug}{\mathfrak g}

\newcommand{\fX}{\mathfrak X}

\newcommand{\ft}{\mathfrak t}
\newcommand{\fa}{\mathfrak a}

\newcommand{\fI}{\mathfrak I}
\newcommand{\fu}{\mathfrak u}
\newcommand{\fw}{\mathfrak w}

\DeclareMathOperator{\gl}{\mathfrak gl}

\DeclareMathOperator{\tr}{\mathrm tr}
\DeclareMathOperator{\GL}{\mathrm GL}

\DeclareMathOperator{\Fil}{\mathrm Fil}
\DeclareMathOperator{\Res}{\mathrm Res}

\DeclareMathOperator{\Gal}{\mathrm Gal}
\DeclareMathOperator{\Hom}{\mathrm Hom}

\DeclareMathOperator{\End}{\mathrm End}
\DeclareMathOperator{\cris}{\mathrm cris}
\DeclareMathOperator{\rig}{\mathrm rig}
\DeclareMathOperator{\an}{\mathrm an}
\DeclareMathOperator{\Spec}{\mathrm Spec}

\DeclareMathOperator{\dR}{\mathrm dR}
\DeclareMathOperator{\Frob}{\mathrm Frob}
\DeclareMathOperator{\Ind}{\mathrm Ind}
\DeclareMathOperator{\unr}{\mathrm unr}
\DeclareMathOperator{\lc}{\mathrm lc}

\DeclareMathOperator{\Ker}{\mathrm Ker}

\DeclareMathOperator{\Ext}{\mathrm Ext}

\DeclareMathOperator{\Spm}{\mathrm Spm}
\DeclareMathOperator{\Spf}{\mathrm Spf}
\DeclareMathOperator{\Ima}{\mathrm Im}

\DeclareMathOperator{\lalg}{\mathrm lalg}

\DeclareMathOperator{\id}{\mathrm id}

\DeclareMathOperator{\dett}{\mathrm det}
\DeclareMathOperator{\alg}{\mathrm alg}
\DeclareMathOperator{\cyc}{\mathrm cyc}
\DeclareMathOperator{\soc}{\mathrm soc}

\DeclareMathOperator{\Ad}{\mathrm Ad}

\DeclareMathOperator{\red}{\mathrm red}

\DeclareMathOperator{\WD}{\mathrm WD}
\DeclareMathOperator{\HT}{\mathrm HT}

\DeclareMathOperator{\rk}{\mathrm rk}

\DeclareMathOperator{\wt}{\mathrm wt}

\DeclareMathOperator{\tri}{\mathrm tri}
\DeclareMathOperator{\reg}{\mathrm reg}
\DeclareMathOperator{\Norm}{\mathrm Norm}
\DeclareMathOperator{\diag}{\mathrm diag}
\DeclareMathOperator{\sm}{\mathrm sm}

\DeclareMathOperator{\pdR}{\mathrm pdR}
\DeclareMathOperator{\val}{\mathrm val}
\DeclareMathOperator{\crr}{\mathrm cr}

\begin{document}
\title{Companion points and locally analytic socle for $\GL_2(L)$}
\author{Yiwen Ding}
\thanks{This work was supported by EPSRC grant EP/L025485/1.}
\maketitle
\selectlanguage{english}
\begin{abstract}Let  $L$ be a finite extension of $\Q_p$, we prove under mild hypothesis Breuil's locally analytic socle conjecture for $\GL_2(L)$, showing the existence of all the companion points on the definite (patched) eigenvariety. This work relies on infinitesimal ``R=T" results for the patched eigenvariety and the comparison of (partially) de Rham families and (partially) Hodge-Tate families. This method allows in particular to find companion points of non-classical points.
\end{abstract}
\tableofcontents
\section{Introduction}\addtocontents{toc}{\protect\setcounter{tocdepth}{1}}

In this paper, we prove (under mild technical hypotheses) Breuil's locally analytic socle conjecture for $\GL_2(L)$ where $L$ is a finite extension of $\Q_p$.

\subsection*{Breuil's locally analytic socle conjecture for $\GL_2(L)$}Let $L_0$ be the maximal unramified extension over $\Q_p$ in $L$, $d:=[L:\Q_p]$, $d_0:=[L_0:\Q_p]$, $E$ be a finite extension of $\Q_p$ big enough to contain all the $\Q_p$-embeddings of $L$ in $\overline{\Q_p}$, and $\Sigma_L:=\Hom_{\Q_p}(L,\overline{\Q_p})=\Hom_{\Q_p}(L,E)$.

Let $\rho_L$ be a two dimensional crystalline representation of $\Gal_L$ over $E$ with distinct Hodge-Tate weights $(-k_{1,\sigma},-k_{2,\sigma})_{\sigma\in \Sigma_L}$ ($k_{1,\sigma}>k_{2,\sigma}$) (where we use the convention that the Hodge-Tate weight of the cyclotomic character is $-1$), let $\alpha$, $\widetilde{\alpha}$ be the eigenvalues of crystalline Frobenius $\varphi^{d_0}$ on $D_{\cris}(\rho_L)$, and suppose $\alpha \widetilde{\alpha}^{-1}\neq 1, p^{\pm d_0}$. Put $\delta:=\unr(\alpha)\prod_{\sigma\in \Sigma_L} \sigma^{k_{1,\sigma}} \otimes \unr(\widetilde{\alpha}) \prod_{\sigma\in \Sigma_L} \sigma^{k_{2,\sigma}}$ (as a character of $T(L)\cong L^{\times}\times L^{\times}$, the subgroup of $\GL_2(L)$ of diagonal matrices), and for $J\subseteq \Sigma_L$, put $\delta_J^c:=\delta \big(\prod_{\sigma\in J} \sigma^{k_{2,\sigma}-k_{1,\sigma}} \otimes \prod_{\sigma\in J} \sigma^{k_{1,\sigma}-k_{2,\sigma}}\big)$ where $\unr(z)$ denotes the unramified character of $L^{\times}$ sending uniformizers to $z$; we define $\widetilde{\delta}$, $\widetilde{\delta}_J^c$ the same way as $\delta$, $\delta_J^c$ by exchanging $\alpha$ and $\widetilde{\alpha}$. Recall $\rho_L$ is trianguline, and  there exists $\Sigma\subseteq \Sigma_L$ (resp. $\widetilde{\Sigma}\subseteq \Sigma_L$) such that $\delta_{\Sigma}^c$ is a trianguline parameter of $\rho_L$, also called a \emph{refinement} of $\rho_L$.
The refinement $\delta_{\Sigma}^c$ \big(resp. $\widetilde{\delta}_{\widetilde{\Sigma}}^c$\big) is called \emph{non-critical} if $\Sigma=\emptyset$ (resp. $\widetilde{\Sigma}=\emptyset$). Note the information of $\Sigma$ and $\widetilde{\Sigma}$ is lost when passing to the Weil-Deligne representation associated to $\rho_L$, thus is invisible in classical local Langlands correspondence. In fact, in terms of filtered $\varphi$-modules, we have
\begin{eqnarray*}
  \Sigma &=& \big\{\sigma\in \Sigma_L \ |\ \text{$\Fil^i D_{\dR}(\rho_L)_{\sigma}$ is an eigenspace of $\varphi^{d_0}$ of eigenvalue $\alpha$,}\ -k_{1,\sigma}<i\leq -k_{2,\sigma}\big\}, \\
   \widetilde{\Sigma} &=& \big\{\sigma\in \Sigma_L \ |\ \text{$\Fil^i D_{\dR}(\rho_L)_{\sigma}$ is an eigenspace of $\varphi^{d_0}$ of eigenvalue $\widetilde{\alpha}$,}\ -k_{1,\sigma}<i\leq -k_{2,\sigma}\big\},
\end{eqnarray*}
where $D_{\dR}(\rho_L)\cong D_{\cris}(\rho_L)\otimes_{L_0} L\cong \otimes_{\sigma\in \Sigma_L} D_{\dR}(\rho_L)_{\sigma\in \Sigma_L}$ is naturally equipped with an $E$-linear action of $\varphi^{d_0}$, and $\Fil^i D_{\dR}(\rho_L)_{\sigma}$ is the one dimensional non-trivial Hodge filtration of $D_{\dR}(\rho_L)_{\sigma}$ for $-k_{1,\sigma}<i\leq -k_{2,\sigma}$.

For a continuous very regular character $\chi$ (cf. (\ref{equ: cclg-dal})) of $T(L)$ over $E$, put $$I(\chi):=\soc (\Ind_{\overline{B}(L)}^{\GL_2(L)} \chi)^{\Q_p-\an}, $$ which is an irreducible locally $\Q_p$-analytic representation of $\GL_2(L)$ over $E$. Note if the weight of $\chi$ is dominant, then $I(\chi)$ is locally algebraic. Put $\chi^{\sharp}:=\chi (\unr(p^{-d_0})\otimes \prod_{\sigma\in \Sigma_L} \sigma)$.

Suppose $\rho_L$ is the restriction of certain global modular Galois representation $\rho$ (i.e. $\rho$ is associated to classical automorphic representations; in this paper, we would consider the case of automorphic representations of definite unitary groups). Using global method, we can attach to $\rho$ an admissible unitary Banach representation $\widehat{\Pi}(\rho)$ of $\GL_2(L)$  (e.g. in the case we consider) such that $I(\delta^{\sharp}\delta_B^{-1})\cong I(\widetilde{\delta}^{\sharp}\delta_B^{-1}) \hookrightarrow \widehat{\Pi}(\rho)$ where $\delta_B=\unr(p^{-d_0})\otimes \unr(p^{d_0})$ is the modulus character of the Borel subgroup $B$ (of upper triangular matrices). The representation $\widehat{\Pi}(\rho)$ is expected to be right representation of $\GL_2(L)$ corresponding to $\rho_L$ in the  $p$-adic Langlands program (see \cite{Br0} for a survey). Let $\widehat{\Pi}(\rho)^{\an}$ be the locally $\Q_p$-analytic subrepresentation of $\widehat{\Pi}(\rho)$, which is dense in $\widehat{\Pi}(\rho)$. We have the following Breuil's conjecture (cf. \cite[Conj.8.1]{Br} and \cite{Br13I}) concerning the socle of $\widehat{\Pi}(\rho)^{\an}$.

\begin{conjecture}[Breuil]\label{conj: cclg-las}
For $\chi: T(L)\ra E^{\times}$, $I(\chi)\hookrightarrow \widehat{\Pi}(\rho)^{\an}$ if and only if $\chi=(\delta_J^c)^{\sharp} \delta_B^{-1}$ for $J\subseteq \Sigma$ or $(\widetilde{\delta}_J^c)^{\sharp} \delta_B^{-1}$ for $J\subseteq \widetilde{\Sigma}$.
\end{conjecture}The ``only if" part would be (in many cases) a consequence of global triangulation theory, while the ``if" part is more difficult.
In modular curve case (thus $L=\Q_p$), this was proved in \cite{BE} using $p$-adic comparison theorems and the theory of overconvergent modular forms. In \cite{Berg14}, Bergdall reproved the result of Breuil-Emerton by studying the geometry of Coleman-Mazur eigencurve at classical points. In \cite{Ding}, some partial results (especially for $|J|=1$) were obtained in unitary Shimura curves case, by showing the existence of overconvergent companion forms over unitary Shimura curves following the strategy of Breuil-Emerton. In this paper, under the so-called Taylor-Wiles hypotheses (see below), we  prove this conjecture in definite unitary groups case, by studying the geometry of (certain stratifications of) the patched eigenvariety of Breuil-Hellmann-Schraen (\cite{BHS1}) at \emph{possibly-non-classial} points.



\subsection*{Main results}Let $F^+$ be a totally real field, and suppose for simplicity in the introduction that there's only one place $u$ of $F^+$ above $p$. Let $F$ be a quadratic imaginary extension of $F^+$ such that $u$ is split in $F$. We fix a place $\widetilde{u}$ of $F$ above $F^+$. To be consistent with the notation in the precedent section, we let $L$ be $F_{\widetilde{u}}\cong F^+_{u}$. Let $G$ be a two variables definite unitary group associated to $F/F^+$ with $G(F_u^+)\cong \GL_2(F_u^+)\cong \GL_2(L)$. We fix a compact open subgroup $U^p=\prod_{v\nmid p, \infty} U_v$ of $G(\bA_{F^+}^{\infty,u})$ with $U_v$ a compact open subgroup of $G(F^+_v)$. Put
\begin{equation*}
  \widehat{S}(U^p,E):=\{f: G(F^+)\backslash G(\bA_{F^+}^{\infty})/U^p \ra E\ |\ \text{$f$ is continuous}\}
\end{equation*}
which is an $E$-Banach space equipped with a continuous unitary action of $\GL_2(L)$ and of some commutative Hecke algebra $\cH^p$ over $\co_E$ outside $p$ (where $\cH^p$ would be big enough to determine Galois representations).

Let $\rho$ be a two dimensional continuous representation of $\Gal_F$ over $E$ associated to classical automorphic representations of $G$. Suppose that we can associate to $\rho$ a maximal ideal $\fm_{\rho}$ of $\cH^p\otimes_{\co_E} E$. We also put the following assumptions (which are often referred to as the \emph{Taylor-Wiles hypotheses})
\begin{itemize}
  \item $p>2$,
  \item $F$ is unramified over $F^+$ and $G$ is quasi-split at all finite places of $F^+$,
  \item $U_v$ is hyperspecial when the finite place $v$ of $F^+$ is inert in $F$,
  \item $\overline{\rho}|_{\Gal_{F(\zeta_p)}}$ is adequate (\cite{Tho}), where $\overline{\rho}$ denotes (the semi-simplification of) the reduction of $\rho$ over $k_E$ (the residue field of $E$)
\end{itemize}Let $\widehat{S}(U^p,E)[\fm_{\rho}]$ denote the  maximal subspace of $\widehat{S}(U^p,E)$ killed by $\fm_{\rho}$. The main result is:
\begin{theorem}\label{thm: cclg-ne2}Conj.\ref{conj: cclg-las} is true for $\widehat{\Pi}(\rho):=\widehat{S}(U^p,E)[\mathfrak{m}_{\rho}]$.
\end{theorem}By Emerton's method (\cite{Em1}), one can construct an eigenvariety $\cE(U^p)$ from $\widehat{S}(U^p,E)$. The closed points of $\cE(U^p)$ can be parameterized by $(\rho',\delta')$ where $\rho'$ is a semi-simple continuous representation of $\Gal_F$ over $\overline{E}$ to which one can associate a maximal ideal $\fm_{\rho'}$ of $\cH^p$, and where  $\delta'$ is a continuous character of $T(L)$ over $\overline{E}$. Moreover, $(\rho',\delta')\in \cE(U^p)(\overline{E})$ if and only if the corresponding eigenspace
\begin{equation*}
  J_B\big(\widehat{S}(U^p,E)^{\an}\big)[\fm_{\rho'}, T(L)=\delta']\neq 0,
\end{equation*}
where $J_B(\cdot)$ denotes the Jacquet-Emerton functor for locally analytic representations. By the very construction and adjunction property of $J_B(\cdot)$, one can deduce from Theorem \ref{thm: cclg-ne2}:
\begin{corollary}\label{cor: cclg-eocp}
Let $\chi$ be a continuous character of $T(L)$ in $E^{\times}$, then $(\rho,\chi)\in \cE(U^p)(\overline{E})$ if and only if $\chi=(\delta_J^c)^{\sharp}$ for $J\subseteq \Sigma$ or $\chi=(\widetilde{\delta}_J^c)^{\sharp}$ for $J\subseteq \widetilde{\Sigma}$.
\end{corollary}
In fact, we also get a similar result in trianguline case (cf. Corollary \ref{cor: cclg-net}). The points $(\rho, (\delta_J^c)^{\sharp})$ for $J\subseteq \Sigma$ \big(resp. $\big(\rho, (\widetilde{\delta}_J^c)^{\sharp}\big)$ for $J\subseteq \widetilde{\Sigma}$\big) are called \emph{companion points} of the classical point $(\rho, \delta^{\sharp})$ \big(resp. of $(\rho,\widetilde{\delta}^{\sharp})\big)$.
We refer to \S~\ref{sec: cclg-4.1} for more discussion on the relation between Breuil's locally analytic socle conjecture and the existence of companion points.
\subsection*{Strategy of the proof}Suppose Taylor-Wiles hypotheses, by \cite{CEGGPS}, using Taylor-Wiles-Kisin patching method, one obtains an $R_{\infty}$-admissible unitary Banach representation $\Pi_{\infty}$ of $\GL_2(L)$, where $R_{\infty}$ is the usual patched deformation ring of $\overline{\rho}$. A point is that one can recover $\widehat{S}(U^p,E)_{\overline{\rho}}$ (the localisation at $\overline{\rho}$) from $\Pi_{\infty}$. In particular,
 to prove Conj.\ref{conj: cclg-las}, it is sufficient to prove a similar result for $\Pi_{\infty}[\fm_{\rho}]$ (where we also use $\fm_{\rho}$ to denote the maximal ideal of $R_{\infty}[1/p]$ corresponding to $\rho$).

The patched eigenvariety $X_p(\overline{\rho})$ (cf. \cite{BHS1}) is a rigid space over $E$ with points parameterized by $(\fm_y,\delta')$ where $\fm_y$ is a maximal ideal of $R_{\infty}[1/p]$, and $\delta'$ is a character of $T(L)$ such that $(\fm_y,\delta')\in X_p(\overline{\rho})$ if and only if the eigenspace
\begin{equation*}
  J_B(\Pi_{\infty}^{R_{\infty}-\an})[\fm_y, T(L)=\delta']\neq 0,
\end{equation*}
where ``$R_{\infty}-\an$" denotes the locally $R_{\infty}$-analytic vectors defined in \cite[\S~3.1]{BHS1}. For $J\subseteq \Sigma_L$,
 let $\ul{\lambda}_J:=(k_{1,\sigma}, k_{2,\sigma}+1)_{\sigma\in J}$,  $L(\ul{\lambda}_J)$ be the irreducible algebraic representation of $\Res_{L/\Q_p} \GL_2$ with highest weight $\ul{\lambda}_J$, and $L(\ul{\lambda}_J)'$ be the algebraic dual of $L(\ul{\lambda}_J)$. In particular,  $L(\ul{\lambda}_J)$ and $L(\ul{\lambda}_J)'$ are locally $J$-analytic representations of $\GL_2(L)$. Put
\begin{equation*}\Pi_{\infty}^{R_{\infty}-\an}(\ul{\lambda}_J):=\big(\Pi_{\infty}^{R_\infty-\an} \otimes_E L(\ul{\lambda}_J)'\big)^{\Sigma_L\setminus J-\an} \otimes L(\ul{\lambda}_J), \end{equation*}
which is a closed subrepresentation of $\Pi_{\infty}^{R_{\infty}-\mathrm{an}}$. From $\Pi_{\infty}^{R_{\infty}-\an}(\ul{\lambda}_J)$,  we can construct a rigid closed subspace $X_p(\overline{\rho}, \ul{\lambda}_J)$ of $X_p(\overline{\rho})$ such that $(y,\delta')\in X_p(\overline{\rho},\ul{\lambda}_J)$ if and only if the eigenspace
\begin{equation*}
  J_B(\Pi_{\infty}^{R_{\infty}-\an}(\ul{\lambda}_J))[\fm_y,T(L)=\delta']\neq 0.
\end{equation*}Denote by $\cT_L$ the rigid space over $E$ parameterizing continuous characters of $T(L)$, and put $\cT_L(\ul{\lambda}_J)$ to be the closed subspace of characters $\chi$ with $\wt(\chi)_{\sigma}=(k_{1,\sigma},k_{2,\sigma}+1)$ for $\sigma\in J$ (e.g. see \cite[\S~2]{Ding4} for the weights of characters). By construction, we have the following commutative diagram (which is \emph{not} Cartesian in general)
\begin{equation*}
  \begin{CD}
    X_p(\overline{\rho},\ul{\lambda}_J) @>>> X_p(\overline{\rho}) \\
    @VVV @VVV \\
    \cT_L(\ul{\lambda}_J) @>>> \cT_L
  \end{CD}
\end{equation*}
where the vertical maps send $(y,\delta')$ to $\delta'$. For $J'\subseteq J$, $\Pi_{\infty}^{R_{\infty}-\an}(\ul{\lambda}_J)$ is a subrepresentation of $\Pi_{\infty}^{R_{\infty}-\an}(\ul{\lambda}_{J'})$, and hence $X_p(\overline{\rho},\ul{\lambda}_J)$ is a rigid closed subspace of $X_p(\overline{\rho},\ul{\lambda}_{J'})$ and  the following diagram commutes
\begin{equation*}
  \begin{CD}
    X_p(\overline{\rho},\ul{\lambda}_J) @>>> X_p(\overline{\rho},\ul{\lambda}_{J'}) \\
    @VVV @VVV \\
    \cT_L(\ul{\lambda}_J) @>>> \cT_L(\ul{\lambda}_{J'}).
  \end{CD}
\end{equation*}
Put $X_p(\overline{\rho},\ul{\lambda}_{J},J'):=X_p(\overline{\rho},\ul{\lambda}_{J'})\times_{\cT_L(\ul{\lambda}_{J'})} \cT_L(\ul{\lambda}_J)$, thus $X_p(\overline{\rho},\ul{\lambda}_J)$ is a closed subspace of $X_p(\overline{\rho},\ul{\lambda}_J,J')$.

Now let $J\subseteq \Sigma$ (the case for $\widetilde{\Sigma}$ is the same), and suppose there exists an injection $I((\delta_J^c)^{\sharp} \delta_B^{-1})\hookrightarrow \Pi_{\infty}^{R_{\infty}-\an}[\fm_{\rho}]$ \big(which automatically factors through $\Pi_{\infty}^{R_{\infty}-\an}(\ul{\lambda}_{\Sigma_L\setminus J})[\fm_{\rho}]$\big); suppose $J\neq \Sigma$, we would prove there exists an injection $I((\delta_{J\cup\{\sigma\}}^c)^{\sharp} \delta_B^{-1}) \hookrightarrow \Pi_{\infty}^{R_{\infty}-\an}[\fm_{\rho}]$ for all $\sigma\in \Sigma\setminus J$, from which (the ``if" part of) Conj.\ref{conj: cclg-las} follows (as mentioned before, the ``only if" part is an easy consequence of the global triangulation theory) by induction on $J$ (note the ($J=\emptyset$)-case is known \emph{a priori} by classical local Langlands correspondence). For $S\subseteq \Sigma_L$, we have in fact
\begin{itemize}
  \item $I((\delta_S^c)^{\sharp}\delta_B^{-1}) \hookrightarrow \Pi_{\infty}^{R_{\infty}-\an}[\fm_{\rho}]$ if and only $x_S^c:=(\fm_{\rho}, (\delta_S^c)^{\sharp})\in X_p(\overline{\rho},\ul{\lambda}_{\Sigma_L\setminus S})$.
\end{itemize}
So by assumption, we have a point $x_J^c\in X_p(\overline{\rho},\ul{\lambda}_{\Sigma_L\setminus J})$, and it is sufficient to find the point $x_{J\cup\{\sigma\}}^c$ inside $X_p(\overline{\rho},\ul{\lambda}_{\Sigma_L \setminus (J\cup\{\sigma\})})$ for $\sigma\in \Sigma\setminus J$:
\begin{theorem}Keep the above notation.

  (1) The rigid space $X_p(\overline{\rho},\ul{\lambda}_{\Sigma_L \setminus J})$ is smooth at the point $x_J^c$.

  (2) The following statements are equivalent:

  \begin{enumerate}[label=(\alph*)]\item $\sigma\in \Sigma \setminus J$;
\item the natural projection of complete noetherian local $E$-algebras
    \begin{equation*}
      \widehat{\co}_{X_p(\overline{\rho},\ul{\lambda}_{\Sigma_L\setminus J}, \Sigma_L \setminus (J\cup\{\sigma\})),x_J^c} \twoheadlongrightarrow \widehat{\co}_{X_p(\overline{\rho},\ul{\lambda}_{\Sigma_L\setminus J}),x_J^c}
    \end{equation*}induced by the closed embedding $X_p(\overline{\rho},\ul{\lambda}_{\Sigma_L \setminus J})\hookrightarrow X_p(\overline{\rho}, \ul{\lambda}_{\Sigma_L \setminus J},\Sigma_L\setminus (J\cup\{\sigma\}))$ is \emph{not} an isomorphism;
    \item $x_{J\cup\{\sigma\}}^c\in X_p(\overline{\rho},\ul{\lambda}_{\Sigma_L \setminus (J\cup \{\sigma\})})$.
  \end{enumerate}
\end{theorem}
The smoothness of $X_p(\overline{\rho},\ul{\lambda}_{\Sigma_L \setminus J})$ follows from the same arguments of \cite{BHS2} (see also \cite{Berg15}). A key point is obtaining (a bound for) the dimension of the tangent space of $X_p(\overline{\rho},\ul{\lambda}_{\Sigma_L \setminus J})$ at $x_J^c$ via Galois cohomology calculation (in our case, it would in particular involve some partially de Rham Galois cohomology considered in \cite{Ding4}).

For (2) (from which Theorem \ref{thm: cclg-ne2} follows), the direction (c) $\Rightarrow$ (a) follows easily from global triangulation theory; (b) $\Rightarrow$ (c) follows from some locally analytic representation theory (e.g. Breuil's adjunction formula) and some commutative algebra arguments (cf. Theorem \ref{thm: cclg-lrt}). The equivalence between (a) and (b) is rather a consequence of infinitesimal ``$R=T$" results, which we explain in more details for the rest of the introduction.

Denote by $R_{\infty}^p$ the ``prime-to-$p$" part of $R_{\infty}$ \big(where $R_{\infty}\cong R_{\infty}^p \widehat{\otimes}_{\co_E} R_{\overline{\rho}_p}^{\square}$\big). By \cite[Thm. 1.1]{BHS1}, we have a natural embedding
\begin{equation}\label{equ: cclg-poa}X_p(\overline{\rho})\hooklongrightarrow (\Spf R_{\infty}^p)^{\rig}\times X_{\tri}^{\square}(\overline{\rho}_p)\end{equation}
 which induces moreover an isomorphism between $X_p(\overline{\rho})$ and a union if irreducible components (equipped with the reduced closed rigid subspace structure) of $(\Spf R_{\infty}^p)^{\rig}\times X_{\tri}^{\square}(\overline{\rho}_p)$, where $\overline{\rho}_p:=\overline{\rho}|_{\Gal_L}$ and $X_{\tri}^{\square}(\overline{\rho}_p)$ is the trianguline variety (cf. \cite{Hel}, \cite[\S~2]{BHS1}) whose closed points can be parameterized as $(\rho_p', \delta')$ where $\rho_p'$ is a framed deformation of $\overline{\rho}_p$ over $\overline{E}$, $\delta'$ is a continuous character of $T(L)$. Using partially de Rham data, one can get a stratification of $X_{\tri}^{\square}(\overline{\rho})$:
for $S'\subseteq S \subseteq \Sigma_L$, one can get a closed rigid subspace $X_{\tri,S'-\dR}^{\square}(\overline{\rho},\ul{\lambda}_S)$ of $X_{\tri}^{\square}(\overline{\rho})$ satisfying in particular  $(\rho'_p, \delta')\in X_{\tri,S'-\dR}^{\square}(\overline{\rho},\ul{\lambda}_S)$ if and only if $\rho'_p$ is $S'$-de Rham and $\wt(\delta')_{\sigma}=(k_{1,\sigma},k_{2,\sigma})$ for $\sigma\in S$ \big(in particular, $\rho'_p$ is $S'$-de Rham and $S\setminus S'$-Hodge-Tate\big). For $\sigma\in \Sigma_L\setminus J$, we can prove (cf. Theorem \ref{thm: cclg-pox}) that (\ref{equ: cclg-poa}) induces a commutative diagram
\begin{equation*}
  \begin{CD}
    X_p(\overline{\rho}, \ul{\lambda}_{\Sigma_L\setminus J}) @>>> (\Spf R_{\infty}^p)^{\rig}\times X_{\tri,\Sigma_L\setminus J-\dR}^{\square}(\overline{\rho}, \ul{\lambda}_{\Sigma_L\setminus J}) \\
    @VVV @VVV \\
    X_p(\overline{\rho}, \ul{\lambda}_{\Sigma_L\setminus J},\Sigma_L\setminus (J\cup \{\sigma\})) @>>> (\Spf R_{\infty}^p)^{\rig}\times X_{\tri,\Sigma_L\setminus (J\cup \{\sigma\})-\dR}^{\square}(\overline{\rho}, \ul{\lambda}_{\Sigma_L\setminus J})
  \end{CD},
\end{equation*}
and moreover the horizontal maps are ``local isomorphisms" at $x_J^c$ (see Theorem \ref{thm: cclg-pox} and Corollary \ref{equ: cclg-oms} for precise statements,  note also the vertical maps are closed embeddings). Denote by $z_J^c$ the image of $x_J^c$ in $X_{\tri}^{\square}(\overline{\rho}_p)$, the equivalence of (a) and (b) thus follows from (cf. Corollary \ref{cor: cclg-htdr}, Corollary \ref{cor: cclg-xts}):
\begin{itemize}\item $\sigma\in \Sigma\setminus J$ if and only if the complete local rings of $X_{\tri,\Sigma_L\setminus J-\dR}^{\square}(\overline{\rho}, \ul{\lambda}_{\Sigma_L\setminus J})$ \big(which is a \emph{$\Sigma_L\setminus J$-de Rham} family\big) and $X_{\tri,\Sigma_L\setminus (J\cup \{\sigma\})-\dR}^{\square}(\overline{\rho}, \ul{\lambda}_{\Sigma_L\setminus J})$ \big(which is a \emph{$\Sigma_L\setminus (J\cup\{\sigma\})$-de Rham} and \emph{$\sigma$-Hodge-Tate} family\big) at $z_J^c$ are different,
\end{itemize}
which is a pure Galois result and follows by Galois cohomology calculations.

Let's remark the particular global context that we are working in is not important for the above arguments. For example, assuming similar patching result for the completed $H^1$ of Shimura curves, one can probably prove Breuil's locally analytic socle conjecture in that case using the same arguments. On the other hand, by comparison of eigenvarieties, it might be also possible to deduce from Corollary \ref{cor: cclg-eocp} the existence of overconvergent Hilbert companion forms.

The results indicated above were the contents of a previous version of this paper, which are contained in the section \S~2 - \S~4 of the current version.  In the recent work of Breuil-Hellmann-Schraen \cite{BHS3}, by developing  a deep theory on the local model of the trianguline variety,  Breuil's locally analytic socle conjecture is now proved in general $\GL_n(L)$-case (also under the Taylor-Wiles hypothesis). We find that the existence of such a local model allows a better understanding of the closed  subspaces (e.g. $X_p(\overline{\rho}, \ul{\lambda}_J)$) of the (patched) eigenvariety, that we studied. In (the new section) \S~5, we prove a partial classicality result (Theorem \ref{thm: cclg-pcl}) and obtain a description of the locally analytic socle in \emph{trianguline case} (i.e. may not be crystalline, see Corollary \ref{cor: cclg-soctri}), which were conjectured in the former version of the paper. We refer to the body of the text for more detailed and more precise statements (with slightly different notations).

\subsection*{Acknowledgement}
I would like to thank Christophe Breuil for suggesting to consider the patched eigenvariety when I noticed the existence of companion points could be deduced from infinitesimal ``R=T" results by the above arguments. I benefit a lot from the study workshop on patched eigenvariety in Morningside center organised by Yongquan Hu, Xu Shen, Yichao Tian, and I would like to thank the organisers and all the participants, and especially thank Liang Xiao for many inspiring discussions. I also would like to thank Rebecca Bellovin, Christophe Breuil, Toby Gee, and Yongquan Hu for remarks on this work or helpful discussions. I thank the anonymous referee for the careful reading and helpful comments.

\addtocontents{toc}{\protect\setcounter{tocdepth}{2}}
\section{Trianguline variety revisited}
\subsection{Trianguline variety and some stratifications}\label{sec: cclg-2.1}Let $L$ be a finite extension of $\Q_p$, $\co_L$ the ring of integers of $L$, $\varpi_L$ a uniformizer of $\co_L$, $d_L:=[L:\Q_p]$, $L_0$ the maximal unramified extension over $\Q_p$ in $L$, $q_L:=|\co_L/\varpi_L|=p^{[L_0:\Q_p]}$, $\Gal_L$ the absolute Galois group of $L$, and $\Sigma_L$ the set of $\Q_p$-embeddings of $L$ in $\overline{\Q_p}$. Let $E$ be a finite extension of $\Q_p$ big enough containing all the embeddings of $L$ in $\overline{\Q_p}$, $\co_E$ be the ring of integers of $E$, $\varpi_E$ a uniformizer of $\co_E$, $k_E:=\co_E/\varpi_E$.

Let $\overline{r}_L:=\Gal_L\ra \GL_2(k_E)$ be a two dimensional continuous representation of $\Gal_L$ over $k_E$. Denote by $R_{\overline{r}_L}^{\square}$ the framed universal deformation ring of $\overline{r}_L$, which is a complete noetherian local $\co_E$-algebra. Put $X_{\overline{r}_L}^{\square}:=(\Spf R_{\overline{r}_L}^{\square})^{\rig}$, which is thus a rigid space over $E$.

Denote by $\cT_L$ the rigid space over $E$ parameterizing continuous characters of $T(L)$ (where $T$ denotes the subgroup of $\GL_2$ of diagonal matrices), thus
\begin{equation*}
    \cT_L(\overline{E})=\big\{\delta: T(L) \ra \overline{E}^{\times}, \  \text{$\delta$ is continuous}\big\}.
\end{equation*}
A continuous character $\delta=\delta_1\otimes \delta_2: T(L)\ra \overline{E}^{\times}$ is called \emph{regular} if (where $\unr(z)$ denotes the unramified character of $L^{\times}$ sending uniformizers to $z$)
\begin{equation}\label{equ: cclg-aj1}
\begin{cases}\text{$\delta_i\delta_j^{-1}\neq \prod_{\sigma\in \Sigma_L} \sigma(z)^{k_{\sigma}}$ for all $\ul{k}_{\Sigma_L}=(k_{\sigma})_{\sigma\in \Sigma_L}\in \Z_{\geq 0}^{d_L}$, $i\neq j$},
 \\ \text{$\delta_i \delta_j^{-1}\neq \unr(q_L)\prod_{\sigma\in \Sigma_L} \sigma(z)^{k_{\sigma}}$ for $\ul{k}_{\Sigma_L} \in \Z_{\geq 1}^{d_L}$, $i\neq j$};
 \end{cases}
\end{equation}
$\delta$ is called \emph{very regular} if
\begin{equation}\label{equ: cclg-dal}
  \begin{cases}\text{$\delta_1\delta_2^{-1}\neq \prod_{\sigma\in \Sigma_L} \sigma(z)^{k_{\sigma}}$ for all $\ul{k}_{\Sigma_L}=(k_{\sigma})_{\sigma\in \Sigma_L}\in \Z^{d_L}$},
 \\ \text{$\delta_i \delta_j^{-1}\neq \unr(q_L)\prod_{\sigma\in \Sigma_L} \sigma(z)^{k_{\sigma}}$ for $\ul{k}_{\Sigma_L} \in \Z^{d_L}$, $i\neq j$}.
 \end{cases}
\end{equation}
Let $\wt(\delta)=(\wt(\delta_1)_{\sigma},\wt(\delta_2)_{\sigma})_{\sigma\in \Sigma_L}\in \overline{E}^{2d_L}$  be the weight of $\delta$, $\wt(\delta)$ is called dominant (resp. strictly dominant) if $\wt(\delta_1)_{\sigma}-\wt(\delta_2)_{\sigma}\in \Z_{\geq 0}$ \big(resp. $\wt(\delta_1)_{\sigma}-\wt(\delta_2)_{\sigma}\in \Z_{\geq 1}$\big) for all $\sigma\in \Sigma_L$. If $\wt(\delta)$ is strictly dominant, then  $\delta$ is regular if and only if $\delta$ is very regular.

Let $\cT_L^{\reg}$ be the subset of $\cT_L(\overline{E})$ of regular characters, which is in fact Zariski-open in $\cT_L$. Put
\begin{equation*}U_{\tri}^{\square,\reg}:=\big\{(r,\delta)\in X_{\overline{r}_L}^{\square}(\overline{E}) \times \cT_L^{\reg}\ \big|\ \text{$r$ is trianguline of parameter $\delta$}\big\}.
\end{equation*}
 Following \cite[Def. 2.4]{BHS1}, let $X_{\tri}^{\square}(\overline{r}_L)\hookrightarrow X_{\overline{r}_L}^{\square} \times \cT_L$ be the (reduced) Zariski-closure of $U_{\tri}^{\square, \reg}$ in $X_{\overline{r}_L}^{\square} \times \cT_L$. Recall
\begin{theorem}[$\text{\cite[Thm. 2.6 (i)]{BHS1}}$]\label{thm: cclg-tridim}
 $X_{\tri}^{\square}(\overline{r}_L)$ is equidimensional of dimension $4+3d_L$.
\end{theorem}

Let $A$ be an artinian local $E$-algebra, $r_A$ a free $A$-module of rank $n$ equipped with a continuous action of $\Gal_L$. Using the isomorphism
\begin{equation*}
  L \otimes_{\Q_p} A \xlongrightarrow{\sim} \oplus_{\sigma\in \Sigma_L} A, \ a\otimes b \mapsto (\sigma(a)b)_{\sigma\in \Sigma_L},
\end{equation*}we see $D_{\dR}(r_A):=(B_{\dR}\otimes_{\Q_p} r_A)^{\Gal_L}$ admits a decomposition $D_{\dR}(r_A)\xrightarrow{\sim} \oplus_{\sigma\in \Sigma_L} D_{\dR}(r_A)_{\sigma}$. For $\sigma\in \Sigma_L$, $r_A$ is called \emph{$\sigma$-de Rham} if $D_{\dR}(r_A)_{\sigma}$ is a free $A$-module of rank $n$; for $J\subseteq \Sigma_L$, $r_A$ is called \emph{$J$-de Rham} if $r_A$ is $\sigma$-de Rham for all $\sigma\in J$. In particular, $r_A$ is de Rham if $r_A$ is $\Sigma_L$-de Rham.

Let $J\subseteq \Sigma_L$, $\ul{k}_J:=(k_{1,\sigma},k_{2,\sigma})_{\sigma\in J}\in \Z^{2|J|}$, and suppose $k_{1,\sigma}\neq k_{2,\sigma}$ for all $\sigma\in J$. Denote by $\cT_L(\ul{k}_J)$  the reduced closed subspace of $\cT_L$ with
\begin{equation*}
  \cT_L(\ul{k}_J)(\overline{E})=\big\{\delta=\delta_1\otimes \delta_2\in \cT_L(\overline{E})\ \big|\ \wt(\delta_i)_{\sigma}=k_{i,\sigma}, \forall \sigma\in J, i=1,2\big\}.
\end{equation*}
Actually, $\cT_L(\ul{0}_J)$ is the rigid space parameterizing locally $\Sigma_L\setminus J$-analytic characters of $T(L)$ (cf. \cite[\S~6.1.4]{Ding}).
Let $X_{\tri}^{\square}(\overline{r}_L,\ul{k}_J):=X_{\tri}^{\square}(\overline{r}_L) \times_{\cT_L} \cT_L(\ul{k}_J)$. Note that for $(r,\delta)\in X_{\tri}^{\square}(\overline{r}_L,\ul{k}_J)(\overline{E})$, the Hodge-Tate weights of $r$ at $\sigma\in J$ are $\{-k_{1,\sigma},-k_{2,\sigma}\}$ (e.g. see \cite[Prop. 2.9]{BHS1}, where we use the convention that the Hodge-Tate weight of the cyclotomic character is $-1$).

Let $X_{\tri,J-\dR}^{\square}(\overline{r}_L, |\ul{k}_J|)$ denote the closed subspace of $X_{\tri}^{\square}(\overline{r}_L)$ satisfying that
\begin{itemize}
  \item for any artinian local $E$-algebra $A$ and $f: \Spec A \hookrightarrow X_{\tri}^{\square}(\overline{r}_L)$, the morphism $f$ factors through $X_{\tri,J-\dR}^{\square}(\overline{r}_L, |\ul{k}_J|)$ if and only if the associated $\Gal_L$-representation $r_A$ is $J$-de Rham of Hodge-Tate weights $(-k_{1,\sigma},-k_{2,\sigma})_{\sigma\in J}$, i.e. $$\Fil^{-k_{i,\sigma}}D_{\dR}(r_A)_{\sigma}/\Fil^{-k_{i,\sigma}+1} D_{\dR}(r_A)_{\sigma}$$ is a free $A$-module of rank $1$ for all $i=1,2$, $\sigma\in J$.
\end{itemize}
The existence of $X_{\tri, J-\dR}^{\square}(\overline{r}_L, |\ul{k}_J|)$ follows essentially from \cite[Thm.  5.2.4]{Bel15} (see also \cite{Sha}). Actually, one can modify the proof of \cite[Thm.  5.2.4]{Bel15},  by replacing ``$D_{\dR}^K(V)$" by ``$\oplus_{\sigma\in J}D_{\dR}^K(V)_{\sigma}$", to obtain a closed subspace $X^{[-N,N]}$ with $N\in \Z_{>0}$ sufficiently large (resp. a closed subspace $X_i$ for $i=1, 2$) of $X_{\tri}^{\square}(\overline{r}_L)$ such that a morphism $f: \Spec A \hookrightarrow X_{\tri}^{\square}(\overline{r}_L)$ factors through $X^{[-N,N]}$ (resp. $X_i$) if and only if  the $A$-module $\Fil^{-N} D_{\dR}(r_A)_{\sigma}/\Fil^N D_{\dR}(r_A)_{\sigma}$  (resp. $\Fil^{-k_{i,\sigma}}D_{\dR}(r_A)_{\sigma}/\Fil^{-k_{i,\sigma}+1} D_{\dR}(r_A)_{\sigma}$) is  free of rank bigger than $2$ (resp. of rank bigger than $1$) for $\sigma\in J$. Then it is not difficult to see $X_{\tri,J-\dR}^{\square}(\overline{r}_L, |\ul{k}_J|)=X^{[-N,N]}\cap X_1\cap X_2$.

Let
\begin{multline}\label{equ: cclg-Jir} X_{\tri,J-\dR}^{\square}(\overline{r}_L,\ul{k}_J):= X_{\tri,J-\dR}^{\square}(\overline{r}_L,|\ul{k}_J|)
\times_{X_{\tri}^{\square}(\overline{r}_L)} X_{\tri}^{\square}(\overline{r}_L,\ul{k}_J)\\ \cong X_{\tri,J-\dR}^{\square}(\overline{r}_L,|\ul{k}_J|)\times_{\cT_L} \cT_L(\ul{k}_J).
\end{multline}
which is a closed subspace of $X_{\tri}^{\square}(\overline{r}_L,\ul{k}_J)$ and of $X_{\tri,J-\dR}^{\square}(\overline{r}_L,|\ul{k}_J|)$. For $w=(w_{\sigma})_{\sigma\in J}\in \cS_2^{|J|}$, denote by $\ul{k}_J^w:=(k_{w_{\sigma}^{-1}(1), \sigma}, k_{w_{\sigma}^{-1}(2), \sigma})_{\sigma\in J}$. We have as closed subpaces of $X_{\tri}^{\square}(\overline{r}_L)$:
\begin{equation}\label{equ: cclg-JdR0}
  X_{\tri, J-\dR}^{\square}(\overline{r}_L, |\ul{k}_J|) =\bigsqcup_{w\in \cS_2^{|J|}} X_{\tri, J-\dR}^{\square}(\overline{r}_L, \ul{k}_J^w).
\end{equation}
For $J'\subset J$, we have $X_{\tri}^{\square}(\overline{r}_L, \ul{k}_J)\cong X_{\tri}^{\square}(\overline{r}_L,\ul{k}_{J'})\times_{\cT_L(\ul{k}_{J'})} \cT_L(\ul{k}_J)$.  We put
\begin{equation}\label{equ: cclg-sJj}
  X_{\tri,J'-\dR}^{\square}(\overline{r}_L,\ul{k}_J):= X_{\tri,J'-\dR}^{\square}(\overline{r}_L,\ul{k}_{J'})\times_{\cT_L(\ul{k}_{J'})} \cT_L(\ul{k}_J).
\end{equation}
In particular, for any finite extension $E'$ of $E$, we have
 \begin{multline*}
   X_{\tri,J'-\dR}^{\square}(\overline{r}_L,\ul{k}_J)(E')\\ =\big\{(r,\delta_1\otimes \delta_2)\in X_{\tri}^{\square}(\overline{r}_L)(E')\ |\ \forall \sigma\in J,\  \wt(\delta_i)_{\sigma}=k_{i,\sigma},\ \text{and $r$ is $J'$-de Rham}\big\}.
 \end{multline*}
We have the following commutative diagram of rigid spaces over $E$ (compare with (\ref{equ: cclg-a1j}) below)
\begin{equation}\footnotesize \label{equ: cclg-rdJ}
  \begin{CD}
    X_{\tri,J-\dR}^{\square}(\overline{r}_L,\ul{k}_J) @>>> X_{\tri,J'-\dR}^{\square}(\overline{r}_L,\ul{k}_J)@>>> X_{\tri,J'-\dR}^{\square}(\overline{r}_L,\ul{k}_{J'})@>>>  X_{\tri}^{\square}(\overline{r}_L,\ul{k}_{J'})@>>> X_{\tri}^{\square}(\overline{r}_L) \\
    @VVV @VVV @VVV @VVV @VVV \\
    \cT_L(\ul{k}_J) @>>> \cT_L(\ul{k}_J) @>>> \cT_L(\ul{k}_{J'}) @>>> \cT_L(\ul{k}_{J'}) @>>>\cT_L
  \end{CD}
\end{equation}
where the horizontal maps are all closed embeddings, and the second and fourth square are cartesian. For a closed subspace $X$ of $X_{\tri}^{\square}(\overline{r}_L)$, put $X(\ul{k}_J):=X\times_{X_{\tri}^{\square}(\overline{r}_L)} X_{\tri}^{\square}(\overline{r}_L,\ul{k}_J)\cong X\times_{\cT_L}\cT_L(\ul{k}_J)$, $X_{J'-\dR}(\ul{k}_J):=X\times_{X_{\tri}^{\square}(\overline{r}_L)} X_{\tri,J'-\dR}^{\square}(\overline{r}_L,\ul{k}_J)$.

\subsection{Tangent spaces} \label{sec: cclg-xdt}For a rigid space $X$ over $E$, $x$ a closed point in $X$, denote by $k(x)$ the residue  field at $x$, and  $T_{X,x}$ the tangent space of $X$ at $x$, with can be identified with the $k(x)$-vector space of morphisms $\Spec k(x)[\epsilon]/\epsilon^2 \ra X$ with the induced map $\Spec k(x)\ra X$ corresponding to $x$.


Let $x=(r,\delta=\delta_1\otimes \delta_2)$ be a closed point in $X_{\tri}^{\square}(\overline{r}_L)$. Suppose that $\delta$ is locally algebraic, i.e. $\wt(\delta)\in \Z^{2d_L}$. Let
\begin{equation}\label{equ: cclg-Domi}
  \Sigma^+(\delta):=\{\sigma\in \Sigma_L\ |\ \wt(\delta_1)_{\sigma}>\wt(\delta_2)_{\sigma}\}, \  \Sigma^-(\delta):=\Sigma_L\setminus \Sigma^+(\delta).
\end{equation}
Suppose that $\delta$ is very regular and $\wt(\delta_1)_{\sigma}\neq \wt(\delta_2)_{\sigma}$ for all $\sigma\in \Sigma_L$. By \cite[Thm. 6.3.13]{KPX} and \cite[Prop. 2.9]{BHS1}, there exist $\Sigma(x)\subseteq \Sigma^+(\delta)$, $n_{1,\sigma}\in \Z_{>0}$, $n_{2,\sigma}\in \Z$ such that $r$ admits a triangulation of parameter
\begin{equation*}
  \delta'=\delta'_1\otimes \delta_2'=\delta \prod_{\sigma\in \Sigma(x)}\big(\sigma^{n_{1,\sigma}} \otimes \sigma^{n_{2,\sigma}}\big).
\end{equation*}
By \cite[Prop. 2.9]{BHS1}, $n_{1,\sigma}=\wt(\delta_2)_{\sigma}-\wt(\delta_1)_{\sigma}$, $n_{2,\sigma}=-n_{1,\sigma}$.
It is easy to see $\Sigma(x)=\Sigma^+(\delta) \setminus \Sigma^+(\delta')$. Let
\begin{equation*}
  C(r):=\{\sigma\in \Sigma_L \ |\ \text{$r$ is $\sigma$-de Rham}\}.
\end{equation*}
By \cite[Prop.  A.3]{Ding4}, we have
\begin{equation*}
  \Sigma^+(\delta') \subseteq C(r).
\end{equation*}
Suppose that $r$ is $\Sigma(x)$-de Rham.


Let $J\subseteq \Sigma^+(\delta)$, $\ul{k}_J:=(k_{1,\sigma},k_{2,\sigma})_{\sigma\in J}$ with $k_{i,\sigma}=\wt(\delta_i)_{\sigma}$. Suppose $J\subseteq C(r)$, i.e. $r$ is $J$-de Rham. Thus, $x$ is a closed point of $X_{\tri,J-\dR}^{\square}(\overline{r}_L,\ul{k}_J)\hookrightarrow X_{\tri}^{\square}(\overline{r}_L,\ul{k}_J)\hookrightarrow X_{\tri}^{\square}(\overline{r}_L)$.
Let $X$ be a union of irreducible components\footnote{By \cite[Cor.  3.7.10]{BHS3}, we know now $X_{\tri}^{\square}(\overline{r}_L)$ is irreducible at $x$.} of an open subset of $X_{\tri}^{\square}(\overline{r}_L)$ such that $X$ satisfies the accumulation property at $x$ (cf. \cite[Def. 2.11]{BHS1}). The following theorem is due to Breuil-Hellmann-Schraen (cf. \cite[\S~4]{BHS2}).
\begin{theorem}\label{thm: cclg-1ql}
 Keep the above situation, then

  (1) $\dim_{k(x)} T_{X, x}=4+3d_L$;

  (2) $\dim_{k(x)} T_{X(\ul{k}_J),x}=4+3d_L-2|J\cap (\Sigma_L\setminus \Sigma(x))|-|J\cap \Sigma(x)|$;
\end{theorem}
Which together with Theorem \ref{thm: cclg-tridim} implies:
\begin{corollary}\label{coro: cclg-smte}
The rigid space $X$ is smooth at the point $x$.
\end{corollary}We will give the proof of Theorem \ref{thm: cclg-1ql} for the convenience of the reader and the author. Indeed, from this proof together with some results in \cite{Ding4}, we  can also obtain
\begin{theorem}\label{thm: cclg-dxX}
Let $J'\subset J$, then

(1) $\dim_{k(x)} T_{X_{J-\dR}(\ul{k}_J),x}= 4+3d_L-2|J|$;

(2) $\dim_{k(x)} T_{X_{J'-\dR}(\ul{k}_J),x}=4+3d_L-2|J'|-2|(J\setminus J')\cap (\Sigma_L\setminus \Sigma(x))|-|(J\setminus J')\cap \Sigma(x)|$.
\end{theorem}
\begin{corollary}\label{cor: cclg-htdr}If $(J\setminus J') \cap \Sigma(x)\neq \emptyset$, then $X_{J-\dR}(\ul{k}_J)$ is a proper closed subspace of $X_{J'-\dR}(\ul{k}_J)$.
\end{corollary}
The rest of the section is devoted to the proof of Theorem \ref{thm: cclg-1ql} and \ref{thm: cclg-dxX}.

Since $X_{\tri}^{\square}(\overline{r}_L)$ is equidimensional of dimension $4+3d_L$, to prove Theorem \ref{thm: cclg-1ql} (1), it is sufficient to show $\dim_{k(x)} T_{X,x}\leq 4+3d_L$. As in \cite[(4.21)]{BHS2}, one has an exact sequence:
\begin{equation}\label{equ: cclg-xXL}
  0 \ra K(r) \cap T_{X,x} \ra T_{X,x} \xrightarrow{f} \Ext_{\Gal_L}^1(r,r)
\end{equation}
where $\dim_{k(x)}K(r)$ is a $k(x)$-vector space of dimension $4-\dim_{k(x)} \End_{\Gal_L}(r)$ (see \cite[Lem.  4.13]{BHS2}). Since
\begin{equation}\label{equ: cclg-xlr}\dim_{k(x)} \Ext^1_{\Gal_L}(r,r)=\dim_{k(x)} \End_{\Gal_L}(r)+4d_L,\end{equation} to prove Theorem \ref{thm: cclg-1ql} (1), it is sufficient to prove
\begin{equation}\label{equ: cclg-gll}
  \dim_{k(x)} \Ima(f)\leq \dim_{k(x)} \Ext^{1}_{\Gal_L}(r,r)-d_L.
\end{equation}

We view an element $\widetilde{r}$ of $\Ext^1_{\Gal_L}(r,r)$ as a rank $2$ representation of $\Gal_L$ over $k(x)[\epsilon]/\epsilon^2$, whose Sen weights thus have the form $(\wt(\delta_i)_{\sigma}+\epsilon d_{\sigma,i})_{i=1,2, \sigma\in \Sigma_L}$. We have a $k(x)$-linear map
\begin{equation*}
\nabla: \Ext^1_{\Gal_L}(r,r) \lra k(x)^{2d_L}, \ \tilde{r}\mapsto (d_{\sigma,1}, d_{\sigma,2})_{\sigma\in \Sigma_L},
\end{equation*}
which is known to be surjective (e.g. see \cite[Prop. 4.9]{BHS2}, see also Lemma \ref{lem: cclg-rve} below).

For $t\in T_{X,x}:\Spec k(x)[\epsilon]/\epsilon^2\ra X_{\tri}^{\square}(\overline{r}_L)$, the composition $\Spec k(x)[\epsilon]/\epsilon^2 \ra X_{\tri}^{\square}(\overline{r}_L) \ra X_{\overline{r}_L}^{\square}$ gives a continuous representation $\widetilde{r}: \Gal_L \ra \GL_2(k(x)[\epsilon]/\epsilon^2)$, which in fact equals the image of $t$ in $\Ext_{\Gal_L}^{1}(r,r)$ via $f$. The composition $\Spec k(x)[\epsilon]/\epsilon^2 \ra X_{\tri}^{\square}(\overline{r}_L) \ra  \cT$ gives a character $\widetilde{\delta}=\widetilde{\delta}_1 \otimes \widetilde{\delta}_2: T(L)\ra E[\epsilon]/\epsilon^2$ satisfying $\widetilde{\delta}\equiv \delta \pmod{\epsilon}$. We have \begin{equation*}(\wt(\delta_1)_{\sigma}+\epsilon d_{\sigma,1}, \wt(\delta_2)_{\sigma}+\epsilon d_{\sigma,2})_{\sigma\in \Sigma_L}=\big(\wt(\widetilde{\delta}_1)_{\sigma}, \wt(\widetilde{\delta}_2)_{\sigma}\big)_{\sigma\in \Sigma_L}\in \big(k(x)[\epsilon]/\epsilon^2\big)^{2d_L}.\end{equation*} The representation $\widetilde{r}$ satisfies the following two properties
\begin{equation}\label{equ: cclg-iast}\begin{cases}
  \text{(i) $\nabla(\widetilde{r}) \in W:=\big\{(d_{\sigma,1},d_{\sigma,2})_{\sigma\in \Sigma_L}\ |\ d_{\sigma,1}=d_{\sigma,2}, \ \forall \sigma \in \Sigma(x)\big\}$;}\\
  \text{(ii) there exists an injection of $(\varphi,\Gamma)$-modules over $\cR_{k(x)[\epsilon]/\epsilon^2}$: $\cR_{k(x)[\epsilon]/\epsilon^2}\big(\widetilde{\delta}_1\big) \hookrightarrow D_{\rig}(\widetilde{r})$;}
\end{cases}
\end{equation}
where (ii) follows from the results in \cite{Bergd14} \cite[Prop. 4.3.5]{Liu}, and (i) is due to Bergdall (cf. \cite[Thm. 7.1]{Bergd14}, see also \cite[Lem. 9.6]{Br13II}). Indeed, by an easy variation of the proof of \cite[Lem. 9.6]{Br13II} (or the proof of \cite[Thm. 7.1]{Bergd14}) with the functor $D_{\cris}(\cdot)$ replaced by the $\sigma$-component of $D_{\mathrm{dR}}(\cdot)$ for $\sigma\in\Sigma(x)$, one can deduce from (ii) that $\wt(\delta_2)_{\sigma}-\wt(\delta_1)_{\sigma}$ is a constant Sen weight of $D_{\rig}(\widetilde{r}) \otimes_{\cR_{k(x)[\epsilon]/\epsilon^2}} \cR_{k(x)[\epsilon]/\epsilon^2}\big(\widetilde{\delta}_1^{-1}\big)$, which, by the precedent discussion, has Sen weights $(0, \wt(\delta_2)_{\sigma'}-\wt(\delta_1)_{\sigma'}+\epsilon(d_{\sigma',2}-d_{\sigma',1}))_{\sigma'\in \Sigma_L}$. We deduce then $d_{\sigma,2}-d_{\sigma,1}=0$ for $\sigma\in \Sigma(x)$.

As in \cite[\S~4]{BHS2}, we show the properties (i) (ii) cut off a $k(x)$-vector subspace of $\Ext^1_{\Gal_L}(r,r)$ of dimension $\dim_{k(x)} \Ext^1_{\Gal_L}(r,r)-d_L$ (from which Theorem \ref{thm: cclg-1ql} (1) follows).
Recall $r$ is trianguline of parameter $(\delta_1', \delta_2')$:
\begin{equation*}
  0 \ra \cR_{k(x)}(\delta_1') \ra D_{\rig}(r) \ra \cR_{k(x)}(\delta_2') \ra 0.
\end{equation*}
Consider the composition \big(where the last one is induced by the natural inclusion $\cR_{k(x)}(\delta_1) \hookrightarrow \cR_{k(x)}(\delta_1')$\big)
\begin{multline}\label{equ: cclg-dgr}
  \Ext^1_{(\varphi,\Gamma)}\big(D_{\rig}(r),D_{\rig}(r)\big) \lra \Ext^1_{(\varphi,\Gamma)}\big(\cR_{k(x)}(\delta_1'), D_{\rig}(r)\big)\\
   \lra \Ext^1_{(\varphi,\Gamma)}\big(\cR_{k(x)}(\delta_1'), \cR_{k(x)}(\delta_2')\big) \xlongrightarrow{j} \Ext^1_{(\varphi,\Gamma)}\big(\cR_{k(x)}(\delta_1), \cR_{k(x)}(\delta_2')\big)
\end{multline}
One can check $\widetilde{r}$ satisfies the property (ii) if and only if $D_{\rig}(\widetilde{r})$ lies in the kernel of the above composition, which we denote  by $V_1$. As in \cite[Prop. 4.11]{BHS2}, we have
\begin{lemma}
$\dim_{k(x)}V_1=\dim_{k(x)} \Ext^1_{\Gal_L}(r,r)-(d_L-|\Sigma(x)|)$.
\end{lemma}
\begin{proof}\phantom\qedhere
  We sketch the proof. Since $r$ is very regular, the first two maps in (\ref{equ: cclg-dgr}) are surjective. Using results in \cite[\S~ 4.4]{BHS2} (see also the proof of Lemma \ref{lem: cclg-jrx} below, which is in term of $B$-pairs), one has $\dim_{k(x)} \Ima(j)=d_L-|\Sigma(x)|$, thus
  \begin{multline*}\dim_{k(x)} V_1=\dim_{k(x)} \Ext^1_{(\varphi,\Gamma)}\big(D_{\rig}(r),D_{\rig}(r)\big)-(d_L-|\Sigma(x)|)\\ =\dim_{k(x)} \Ext^1_{\Gal_L}(r,r)-(d_L-|\Sigma(x)|). \qed
  \end{multline*}
\end{proof}As in (the proof of) \cite[Prop. 4.12]{BHS2}, one has
\begin{lemma}\label{lem: cclg-rve}
  The induced map $\nabla: V_1\ra  k(x)^{2d_L}$ is surjective.
\end{lemma}
\begin{proof}
This lemma follows from the fact that the trianguline deformations of $D_{\rig}(r)$ over $E[\epsilon]/\epsilon^2$ are contained in $V_1$ (which is obvious), and the tangent map from the trianguline deformation space to the weight space is surjective. Indeed, for any continuous character $\widetilde{\delta}':L^{\times} \ra (E[\epsilon]/\epsilon^2)^{\times}$ with $\widetilde{\delta}'\equiv \delta_1' (\delta_2')^{-1} \pmod{\epsilon}$, consider the exact sequence of $(\varphi,\Gamma)$-modules over $\cR_{k(x)}$
\begin{equation*}
  0 \ra \cR_{k(x)}(\delta_1'(\delta_2')^{-1})\ra \cR_{k(x)[\epsilon]/\epsilon^2}(\widetilde{\delta}') \ra \cR_{k(x)}(\delta_1'(\delta_2')^{-1}) \ra 0.
\end{equation*}
Since $r$ is very regular, the induced map $H^1_{(\varphi,\Gamma)}\big(\cR_{k(x)[\epsilon]/\epsilon^2}(\widetilde{\delta}')\big)\ra H^1_{(\varphi,\Gamma)}\big(\cR_{k(x)}(\delta_1'(\delta_2')^{-1}\big)$ is surjective. Let $D'\in H^1_{(\varphi,\Gamma)}\big(\cR_{k(x)[\epsilon]/\epsilon^2}(\widetilde{\delta}')\big)$ be a preimage of $[D_{\rig}(r) \otimes_{\cR_{k(x)}} \cR_{k(x)}((\delta_2')^{-1})]$, thus for any continuous character $\widetilde{\delta}_2': L^{\times} \ra (k(x)[\epsilon]/\epsilon^2)^{\times}$ with $\widetilde{\delta}'_2\equiv \delta_2' \pmod{\epsilon}$, we see $D:=D'\otimes_{\cR_{k(x)[\epsilon]/\epsilon^2}}\cR_{k(x)[\epsilon]/\epsilon^2}(\widetilde{\delta}_2')$ is a trianguline deformation of $D_{\rig}(r)$ over $\cR_{k(x)[\epsilon]/\epsilon^2}$ with Sen weights $\big(\wt(\widetilde{\delta}')_{\sigma}+\wt(\widetilde{\delta}_2')_{\sigma}, \wt(\widetilde{\delta}_2')_{\sigma}\big)_{\sigma\in \Sigma_L}$. it is straightforward to see $[D]\in V_1$. For any $(a_{\sigma},b_{\sigma})\in k(x)^{2d_L}$, choose $\widetilde{\delta}'$ and $\widetilde{\delta}_2'$ such that $\wt(\widetilde{\delta}_2')_{\sigma}=b_{\sigma}$, $\wt(\widetilde{\delta}')_{\sigma}=a_{\sigma}-b_{\sigma}$, then $\nabla([D])= (a_{\sigma},b_{\sigma})_{\sigma\in \Sigma_L}$, so $\nabla|_{V_1}$ is surjective.
\end{proof}
By (\ref{equ: cclg-iast}), we have $\Ima(f)\subseteq V_1 \cap \nabla^{-1}(W)$. Since $\nabla|_{V_1}$ is surjective, $\dim_{k(x)} \nabla^{-1}(W)\cap V_1=\Ext^1_{\Gal_L}(r,r)-(d_L-|\Sigma(x)|)-|\Sigma(x)|=\Ext^1_{\Gal_L}(r,r)-d_L$. The  part (1) of Theorem \ref{thm: cclg-1ql} follows (cf. (\ref{equ: cclg-gll})), and one gets equalities \begin{equation}\label{equ: cclg-V1x}
  \begin{cases}
    \Ima(f)=V_1\cap \nabla^{-1}(W),\\
    K(r)\cap T_{X,x}=K(r).
  \end{cases}
\end{equation}
By Lemma \ref{lem: cclg-rve}, the composition
\begin{equation}\label{equ: cclg-jjW}
T_{X,x} \xlongrightarrow{f} \Ima(f) \xlongrightarrow{\nabla}W
\end{equation}
is surjective. Put $W_J:=\big \{(d_{\sigma,1},d_{\sigma,2})_{\sigma\in \Sigma_L}\ |\ d_{\sigma,1}=d_{\sigma,2}=0, \ \forall \sigma \in J\big\}$. By definition, $T_{X(\ul{k}_J),x}\subseteq T_{X,x}$ equals the preimage of $W\cap W_J$ via (\ref{equ: cclg-jjW}), and thus
\begin{equation*}\dim_{k(x)} T_{X(\ul{k}_J),x}=\dim_{k(x)} T_{X,x}-|J\cap \Sigma(x)|-2|J\cap (\Sigma_L\setminus \Sigma(x))|,
\end{equation*}
Theorem \ref{thm: cclg-1ql} (2) follows.

We prove Theorem \ref{thm: cclg-dxX}. For $v\in T_{X,x}$, let $\widetilde{r}$ be the $\Gal_L$-representation over $k(x)[\epsilon]/\epsilon^2$ associated to $f(v)$ (cf. (\ref{equ: cclg-xXL})). By definition and (\ref{equ: cclg-V1x}), $v\in T_{X_{J-\dR}(\ul{k}_J),x}$ if and only if $\widetilde{r}$ satisfies the condition (i) (ii) in (\ref{equ: cclg-iast}) and
\begin{itemize}
  \item[(iii)] $\widetilde{r}$ is $J$-de Rham.
\end{itemize}
Denote by $\Ext^1_{\Gal_L,g,J}(r,r)$ the $k(x)$-vector subspace of $\Ext^1_{\Gal_L}(r,r)$ consisting of $J$-de Rham extensions. By the above discussion, (\ref{equ: cclg-xXL}) and   (\ref{equ: cclg-V1x}), we have an exact sequence
\begin{equation*}
  0 \ra K(r) \ra T_{X_{J-\dR}(\ul{k}_J),x} \xrightarrow{f} \Ext^1_{\Gal_L,g,J}(r,r)\cap V_1 \cap \nabla^{-1}(W)\ra 0.
\end{equation*}
For $J'\subset J$, $v\in T_{X_{J'-\dR}(\ul{k}_J),x}$ if and only if $v\in T_{X_{J'-\dR}(\ul{k}_{J'}),x}$ and $\nabla \circ f(v)\in W_J$, so we have an exact sequence
\begin{equation*}
  0 \ra K(r) \ra T_{X_{J'-\dR}(\ul{k}_{J}),x} \xrightarrow{f}  \Ext^1_{\Gal_L,g,J'}(r,r)\cap V_1 \cap \nabla^{-1}(W\cap W_J) \ra 0.
\end{equation*}
Theorem \ref{thm: cclg-dxX} is then an easy consequence of the following proposition.
\begin{proposition}\label{prop: cclg-rjg}
  (1) $\dim_{k(x)} \Ext^1_{\Gal_L,g,J}(r,r) \cap V_1\cap \nabla^{-1}(W)=\dim_{k(x)} \Ext^1_{\Gal_L}(r,r)-d_L-2|J|$.

  (2) $\dim_{k(x)} \Ext^1_{\Gal_L,g,J'}(r,r) \cap V_1 \cap  \nabla^{-1}(W\cap W_J)=\dim_{k(x)} \Ext^1_{\Gal_L}(r,r)-d_L-2|J'|-2|(J\setminus J') \cap (\Sigma_L\setminus \Sigma(x))|-|(J\setminus J') \cap \Sigma(x)|$.
\end{proposition}
We deduce Proposition \ref{prop: cclg-rjg} from  the following two lemmas.
\begin{lemma}\label{lem: cclg-jrx}
  $\dim_{k(x)} \Ext^1_{\Gal_L,g,J}(r,r)\cap V_1= \dim_{k(x)}\Ext^1_{\Gal_L}(r,r)-3|J|-(d_L-|\Sigma(x)|-|J\cap \Sigma^+(\delta')|)$.
\end{lemma}
\begin{lemma}\label{lem: cclg-arl}
The induced map
\begin{equation*}
    \nabla: \Ext^1_{\Gal_L,g,J}(r,r)\cap V_1 \lra W_J
\end{equation*}
is surjective. \end{lemma}
\begin{proof}[Proof of Proposition \ref{prop: cclg-rjg}](1) 
We have
\begin{multline*}
  \dim_{k(x)} \mathrm{Ext}^1_{\mathrm{Gal}_L,g, J}(r,r)\cap V_1 \cap \nabla^{-1}(W)=\dim_{k(x)}\Ext^1_{\Gal_L,g,J}(r,r)\cap V_1 -|\Sigma(x) \cap (\Sigma_L\setminus J)| \\ = \dim_{k(x)} \Ext^1_{\Gal_L}(r,r)-3|J|-(d_L-|\Sigma(x)|-|J\cap \Sigma^+(\delta')|)-|\Sigma(x) \cap (\Sigma_L\setminus J)|
\\ =\dim_{k(x)} \Ext^1_{\Gal_L}(r,r)-d_L-3|J|+|J\cap \Sigma^+(\delta')|+|J\cap \Sigma(x)|
\\ =\dim_{k(x)} \Ext^1_{\Gal_L}(r,r)-d_L-2|J|,
\end{multline*}
where the first equality follows from Lemma \ref{lem: cclg-arl}, the second from Lemma \ref{lem: cclg-jrx}, and the last from  $\Sigma^+(\delta')\sqcup \Sigma(x)=\Sigma^+(\delta)$ and $J\subset \Sigma^+(\delta)$.

(2) By Lemma \ref{lem: cclg-arl} applied to $J'$, we deduce that the following  map is surjective:
\begin{equation*}
 \Ext^1_{\Gal_L,g,J'}(r,r)\cap  V_1 \cap \nabla^{-1}(W) \twoheadlongrightarrow W \cap W_{J'}.
\end{equation*}
It is straightforward to see $\Ext^1_{\Gal_L,g,J'}(r,r) \cap V_1 \cap  \nabla^{-1}(W\cap W_J)$ is the preimage of $W\cap W_J$ via the above map. So (2) follows from (1) and the easy fact that $|W\cap W_{J'}|-|W\cap W_J|=2|(J\setminus J')\cap \Sigma(x)^c|+|(J\setminus J')\cap \Sigma(x)|$.
\end{proof}
We use the language of $B$-pairs in the proof of Lemma \ref{lem: cclg-jrx} and \ref{lem: cclg-arl}.
\begin{proof}[Proof of Lemma \ref{lem: cclg-jrx}]Let $W(r)$ denote the $k(x)$-$B$-pair associated to $r$, which lies in an exact sequence of $k(x)$-$B$-pairs (cf. Appendix \ref{sec: cclg-A.1} below for the notation for $B$-pairs)
\begin{equation*}
  0 \ra B_{k(x)}(\delta_1') \ra W(r) \ra B_{k(x)}(\delta_2') \ra 0.
\end{equation*} Identifying $\Ext^1_{\Gal_L}(r,r)$ and $H^1(\Gal_L, W(r) \otimes W(r)^{\vee})$, $\Ext^1_{\Gal_L, g,J}(r,r) \cap V_1$ equals thus the kernel of the following composition (see (\ref{equ: cclg-dgr}), note we also identify the cohomology of $(\varphi,\Gamma)$-modules with the cohomology of the corresponding $B$-pairs, e.g. see \cite[\S~5]{Na2})
\begin{multline}\label{equ: cclg-x2b}
  H^1_{g,J}\big(\Gal_L, W(r)\otimes W(r)^{\vee}\big) \lra H^1_{g,J}\big(\Gal_L, B_{k(x)}(\delta_2') \otimes W(r)^{\vee}\big) \\
  \lra H^1_{g,J}\big(\Gal_L, B_{k(x)}(\delta_2')\otimes B_{k(x)}((\delta_1')^{-1})\big)\cong H^1_{g,J}\big(\Gal_L, B_{k(x)}(\delta_2'(\delta_1')^{-1})\big)\\ \lra H^1_{g,J}\big(\Gal_L, B_{k(x)}(\delta_2'\delta_1^{-1})\big)
\end{multline}
where the first two maps are  surjective by Proposition \ref{prop: cclg-exbp} (since $\delta$ is very regular), and the last map is induced by the natural inclusion $B_{k(x)}(\delta_2'(\delta_1')^{-1})\hookrightarrow B_{k(x)}(\delta_2'\delta_1^{-1})$. Denote by $\delta_0:=\delta_2'\delta_1^{-1}$, $\delta_0':=\delta_2'(\delta_1')^{-1}$, thus $\delta_0=\delta_0'\prod_{\sigma\in \Sigma(x)} \sigma^{-n_{\sigma}}$, with $n_{\sigma}:=|\wt(\delta_1)_{\sigma}-\wt(\delta_2)_{\sigma}|\in \Z_{\geq 1}$ for $\sigma \in \Sigma_L$, and
\begin{equation}\label{equ: cclg-nxn}
  \wt(\delta_0')_{\sigma}=\begin{cases}
    n_{\sigma} & \sigma \in \Sigma_L \setminus \Sigma^+(\delta') \\
    -n_{\sigma} & \sigma \in \Sigma^+(\delta')
  \end{cases},\ \
  \wt(\delta_0)_{\sigma}=
  \begin{cases}
    \wt(\delta_0')_{\sigma} & \sigma \in \Sigma_L\setminus \Sigma(x) \\
    0 & \sigma\in \Sigma(x)
  \end{cases}.
\end{equation}
One has an exact sequence (of $\Gal_L$-complexes)
\begin{multline*}
  0 \lra \big[B_{k(x)}(\delta'_0)_e \oplus B_{k(x)}(\delta'_0)_{\dR}^+ \ra B_{k(x)}(\delta_0')_{\dR}\big] \\ \lra \big[B_{k(x)}(\delta_0)_e \oplus B_{k(x)}(\delta_0)_{\dR}^+ \ra B_{k(x)}(\delta_0)_{\dR}\big]\\  \lra \big[\oplus_{\sigma\in \Sigma(x)} B_{k(x)}(\delta_0)_{\dR,\sigma}^+/t^{n_{\sigma}} \ra 0 \big] \lra 0
\end{multline*}
which induces
\begin{multline*}
  0 \ra \oplus_{\sigma\in \Sigma(x)} H^0(\Gal_L, B_{k(x)}(\delta_0)_{\dR,\sigma}^+/t^{n_{\sigma}}) \ra H^1(\Gal_L, B_{k(x)}(\delta_0'))\\ \xrightarrow{j} H^1(\Gal_L, B_{k(x)}(\delta_0)) \ra \oplus_{\sigma\in \Sigma(x)} H^1(\Gal_L, B_{k(x)}(\delta_0)_{\dR,\sigma}^+/t^{n_{\sigma}}) \ra 0.
\end{multline*}
By (\ref{equ: cclg-nxn}), one has $B_{k(x)}(\delta_0)_{\dR,\sigma}^+\cong B_{\dR,\sigma}^+$ for $\sigma\in \Sigma(x)$. Thus $\dim_{k(x)} \Ima(j)=d_L -|\Sigma(x)|$. Moreover, by \cite[(7)]{Ding4}, $\Ima(j)=H^1_{g,\Sigma(x)}(\Gal_L, B_{k(x)}(\delta_0))$. We claim that the map $j$ restricts to a surjective map
\begin{equation*}
    H^1_{g,J}(\Gal_L, B_{k(x)}(\delta_0')) \twoheadlongrightarrow H^1_{g,\Sigma(x) \cup J}(\Gal_L, B_{k(x)}(\delta_0)).
\end{equation*}
Indeed we have a commutative diagram
  \begin{equation*}
    \begin{CD}
      H^1(\Gal_L, B_{k(x)}(\delta_0')) @> j >> H^1_{g,\Sigma(x)}(\Gal_L, B_{k(x)}(\delta_0)) \\
      @VVV @VVV\\
      \oplus_{\sigma\in J\cap(\Sigma_L\setminus \Sigma(x))} H^1(\Gal_L, B_{k(x)}(\delta_0')_{\dR,\sigma}^+) @> \cong >> \oplus_{\sigma\in J\cap(\Sigma_L\setminus \Sigma(x))} H^1(\Gal_L, B_{k(x)}(\delta_0)_{\dR,\sigma}^+)
    \end{CD}
  \end{equation*}
where the top map is surjective, and the bottom isomorphism follows from (\ref{equ: cclg-nxn}). Note the kernel of the left vertical map is $H^1_{g,J\cap (\Sigma_L\setminus \Sigma(x))}\big(\Gal_L,B_{k(x)}(\delta_0')\big)\cong H^1_{g,J}\big(\Gal_L, B_{k(x)}(\delta_0')\big)$, since for all $\sigma\in J\cap \Sigma(x)\subseteq \Sigma_L\setminus \Sigma^+(\delta')$, $H^1_{g,\sigma}\big(\Gal_L,B_{k(x)}(\delta_0')\big)=H^1\big(\Gal_L,B_{k(x)}(\delta_0')\big)$ by \cite[Lem. 1.11]{Ding4} and (\ref{equ: cclg-nxn}). The kernel of the right vertical map is by definition $H^1_{g,\Sigma(x) \cup J}(\Gal_L, B_{k(x)}(\delta_0))$. 
The claim then follows.

By this claim, the composition (\ref{equ: cclg-x2b}) induces a surjection
\begin{equation*}
   H^1_{g,J}(\Gal_L, W(r)\otimes W(r)^{\vee})  \twoheadlongrightarrow H^1_{g,\Sigma(x) \cup J}(\Gal_L, B_{k(x)}(\delta_0)).
\end{equation*}
Since $\delta'$ is very regular, by Proposition \ref{prop: cclg-wrn} and Corollary \ref{cor: cclg-pdR}, we have
\begin{multline*}
  \dim_{k(x)} H^1_{g,J}(\Gal_L, W(r) \otimes W(r)^{\vee})\\ =\dim_{k(x)} H^1(\Gal_L, W(r)\otimes W(r)^{\vee})-\sum_{\sigma\in J} \dim_{k(x)} H^0(\Gal_L, (W(r)\otimes W(r)^{\vee})_{\dR,\sigma}^+)\\ =\dim_{k(x)} \Ext^1_{\Gal_L}(r,r)-3|J|,
\end{multline*}
\begin{equation*}
 \dim_{k(x)} H^1_{g,\Sigma(x) \cup J}(\Gal_L, B_{k(x)}(\delta_0))=d_L-|\Sigma(x)|-|J\cap \Sigma^+(\delta')|.
\end{equation*}
The lemma follows.
\end{proof}
\begin{proof}[Proof of Lemma \ref{lem: cclg-arl}]Let $\widetilde{\delta}': L^{\times} \ra (k(x)[\epsilon]/\epsilon^2)^{\times}$ be a continuous character with $\widetilde{\delta}'\equiv \delta_1'(\delta_2')^{-1} \pmod{\epsilon}$. As discussed in  Example \ref{ex: cclg-bpair} (4), the associated $k(x)[\epsilon]/\epsilon^2$-$B$-pair $B_{(k(x)[\epsilon]/\epsilon^2)^{\times}}(\widetilde{\delta})$ is $\sigma$-de Rham if and only if $\wt(\widetilde{\delta})\in \Z$. For any $\widetilde{\delta}': L^{\times} \ra (k(x)[\epsilon]/\epsilon^2)^{\times}$ as above satisfying moreover $\wt(\widetilde{\delta}')_{\sigma}=\wt(\delta)_{\sigma}$ for all $\sigma\in J$, by Proposition \ref{prop: cclg-exbp}, the natural morphism $B_{(k(x)[\epsilon]/\epsilon^2)^{\times}}(\widetilde{\delta}') \ra B_{k(x)}(\delta_1' (\delta_2')^{-1})$ induces a surjection $H^1_{g,J}\big(\Gal_L, B_{(k(x)[\epsilon]/\epsilon^2)^{\times}}(\widetilde{\delta}') \big) \twoheadrightarrow H^1_{g,J}\big(\Gal_L, B_{k(x)}(\delta_1' (\delta_2')^{-1})\big)$.  Let $W'\in H^1_{g,J}\big(\Gal_L, B_{(k(x)[\epsilon]/\epsilon^2)^{\times}}(\widetilde{\delta}') \big)$ be a preimage of $[W(r) \otimes B_{k(x)}((\delta_2')^{-1})]$ which is thus $J$-de Rham, and let $\widetilde{\delta}_2': L^{\times} \ra (k(x)[\epsilon]/\epsilon^2)^{\times}$ be a continuous character with $\widetilde{\delta}_2'\equiv \delta_2' \pmod{\epsilon}$ and $\wt(\widetilde{\delta}_2')_{\sigma}=\wt(\delta_2')_{\sigma}$ for all $\sigma\in J$ \big(thus $B_{k(x)}(\widetilde{\delta}_2')$ is $J$-de Rham\big). it is straightforward to see that $W:=W'\otimes B_{(k(x)[\epsilon]/\epsilon^2)^{\times}}(\widetilde{\delta}_2')$ lies in $\Ext^1_{\Gal_L, g,J}(r,r)\cap V_1$. For any $(a_{\sigma},b_{\sigma})\in k(x)^{2(d_L-|J|)}$, choose $\widetilde{\delta}'$ and $\widetilde{\delta}_2'$ as above which satisfy moreover $\wt(\widetilde{\delta}_2')_{\sigma}=\wt(\delta_2')_{\sigma}+b_{\sigma}\epsilon$, $\wt(\widetilde{\delta}')_{\sigma}=\wt(\delta_1(\delta_2')^{-1})_{\sigma}+(a_{\sigma}-b_{\sigma})\epsilon$ for $\sigma \in J$. We have $\nabla([D])= (a_{\sigma},b_{\sigma})_{\sigma\in \Sigma_L}$, and thus $\nabla: \Ext^1_{\Gal_L, g,J}(r,r)\cap V_1 \ra W_J$ is surjective.\end{proof}

\section{Patched eigenvarieties revisited}\label{sec: cclg-3}
\subsection{Setup and notation}\label{sec: cclg-3.1}We fix embeddings $\iota_{\infty}: \overline{\Q} \hookrightarrow \bC$, and $\iota_{p}: \overline{\Q} \hookrightarrow \overline{\Q_p}$.
Let $F^+$ be a totally real number field, $F$ be a quadratic imaginary extension of $F^+$ and $c$ be the unique non-trivial element in $\Gal(F/F^+)$. Suppose for any finite place $v|p$ of $F^+$, $v$ is split in $F$. Let $E$ be a finite  extension of $\Q_p$ big enough to contain all the embedding of $F^+$ in $\overline{\Q_p}$, with $\co_E$ its ring of integers and $\varpi_E$ a uniformizer. Denote by $\Sigma_v$ the set of $\Q_p$-embeddings of $F_v^+$ in $\overline{\Q_p}$, $\Sigma_p:=\cup_{v|p} \Sigma_v$.

Let $G$ be a $2$-variables definite unitary group over $F^+$ which splits over $F$. We fix an isomorphism of algebraic groups $G\times_{F^+} F \cong \GL_2/F$. If $v$ is a finite place of $F^+$ split in $F$, $\widetilde{v}$ is a place of $F$ dividing $v$, thus $F^+_v \cong F_{\widetilde{v}}$, and we get an isomorphism $i_{\widetilde{v}}: G(F_v^+) \xrightarrow{\sim} \GL_2(F_{\widetilde{v}})$. Let $U^p$ be a compact open subgroup of  $G(\bA_{F^+}^{\infty,p})$ of the form $U^p=\prod_{v \nmid p} U_v$ where $\bA_{F^+}^{\infty,p}$ denotes the finite adeles away from $p$, and $U_v$ is an open compact subgroup of $G(F_v^+)$.

Let $S$ be a finite set of finite places of $F^+$  that split in $F$ containing all the places above $p$ and the places $v$ where $U_v$ is not maximal. Let $S_p$ denote the set of the places of $F^+$ above $p$. We \textbf{fix} a place $\widetilde{v}$ of $F$ above $v$ for $v\in S$ and denote by $\widetilde{v}^c$ its image under $c$.
In particular, for $v|p$, we identify $F_{\widetilde{v}}$ and $F_v^+$, the set $\Sigma_{\widetilde{v}}$ of $\Q_p$-embeddings of $F_{\widetilde{v}}$ in $\overline{\Q_p}$ and $\Sigma_v$. Denote by $F_{\widetilde{v},0}$ the maximal unramified extension of $\Q_p$ inside $F_{\widetilde{v}}$, $\varpi_{\widetilde{v}}$ a uniformizer of $F_{\widetilde{v}}$, $d_{\widetilde{v}}:=[F_{\widetilde{v}}:\Q_p]$, $q_{\widetilde{v}}:=p^{[F_{\widetilde{v},0}:\Q_p]}$, $\val_{\widetilde{v}}(\cdot)$ the additive valuation on $F_{\widetilde{v}}$ normalized by  sending $\varpi_{\widetilde{v}}$ to $1$, and $\unr_{\widetilde{v}}(z)$ the unramified character of $F_{\widetilde{v}}^{\times}$ sending $\varpi_{\widetilde{v}}$ to $z$.

Suppose moreover that for any finite place  $v\notin S$ and $v$ split in $F$, we have $U_v=i_{\widetilde{v}}^{-1}(\GL_2(\co_{F_{\widetilde{v}}}))$ (note that this condition is actually independent of the choice of $\widetilde{v}|v$). For such a place $v$ of $F^+$, let $\bT_v$  be the commutative spherical Hecke algebra \begin{equation*}\co_E[U_v \backslash G(F_v^+)/U_v]\xrightarrow[\sim]{i_{\widetilde{v}}} \co_E[\GL_2(F_{\widetilde{v}})//\GL_2(\co_{F_{\widetilde{v}}})]\cong \co_E[T_{\widetilde{v}},S_{\widetilde{v}}^{\pm 1}]\end{equation*} where $T_{\widetilde{v}}=\GL_2(\co_{F_{\widetilde{v}}})\begin{pmatrix} \varpi_{\widetilde{v}} & 0 \\0 & 1\end{pmatrix} \GL_2(\co_{F_{\widetilde{v}}})$ and $S_{\widetilde{v}}=\GL_2(\co_{F_{\widetilde{v}}}) \begin{pmatrix}
  \varpi_{\widetilde{v}} & 0 \\ 0 & \varpi_{\widetilde{v}}
\end{pmatrix} \GL_2(\co_{F_{\widetilde{v}}})$. One has in fact $i_{\widetilde{v}}^{-1} (T_{\widetilde{v}})=i_{\widetilde{v}^c}^{-1}(S_{\widetilde{v}^c}^{-1} T_{\widetilde{v}^c})$, $i_{\widetilde{v}}^{-1}(S_{\widetilde{v}})=i_{\widetilde{v}^c}^{-1}(S_{\widetilde{v}^c})$ with $v=\widetilde{v}\widetilde{v}^c$. Put $\bT^S:=\varinjlim_{I} (\otimes_{v\in I} \bT_v)$ for $I$ running over finite sets disjoint from $S$ of finite places of $F^+$ which are completely decomposed in $F$, thus $\bT^S$ is a commutative $\co_E$-algebra.

Let $\overline{\rho}$ be a $2$-dimensional continuous absolutely irreducible representation of $\Gal_F$ over $k_E$ such that $\overline{\rho}^{\vee}\circ c \cong \overline{\rho}\otimes \overline{\varepsilon}$ and $\overline{\rho}$ is unramified outside $S$ (where $\varepsilon$ denotes the cyclotomic character). We associate to $\overline{\rho}$ a maximal ideal $\fm(\overline{\rho})$ of $\bT^S$ generated by $\varpi_E$, $T_{\widetilde{v}}-\tr(\overline{\rho}(\Frob_{\widetilde{v}}))$ and $\Norm(\widetilde{v})S_{\widetilde{v}}- \dett(\overline{\rho}(\Frob_{\widetilde{v}}))$ where $\Norm(\widetilde{v})$ denotes the cardinality of the residue field of $F_{\widetilde{v}}$, $\Frob_{\widetilde{v}}$ denotes a geometric Frobenius. 

For a compact open subgroup $U_p$ of $G(F^+\otimes_{\Q} \Q_p)$ and $s\in \Z_{>0}$, put
\begin{equation*}S(U^pU_p,\co_E/\varpi_E^s):=\{f: G(F^+)\backslash G(\bA_{F^+}^{\infty})/(U^pU_p) \ra \co_E/\varpi_E^s\}\end{equation*}
which is a finite $\co_E/\varpi_E^s$-module. Put $\widehat{S}(U^p,\co_E):=\varprojlim_s \varinjlim_{U_p}S(U^pU_p,\co_E/\varpi_E^s)$, $\widehat{S}(U^p,\co_E)_{\overline{\rho}}:=\varprojlim_s \varinjlim_{U_p}S(U^pU_p,\co_E/\varpi_E^s)_{\fm(\overline{\rho})}$, $\widehat{S}(U^p,E)_{*}:=\widehat{S}(U^p,\co_E)_{*}\otimes_{\co_E} E$ with $*\in \{\emptyset, \overline{\rho}\}$. All these $\co_E$-modules (or $E$-vector spaces) are equipped with a natural action of $\bT^S$ via continuous operators. For $*\in \{\emptyset, \overline{\rho}\}$, $**\in \{\co_E, E\}$, $\widehat{S}(U^p,**)_*$ is equipped with a continuous (unitary) action of $G(F^+\otimes_{\Q} \Q_p)\cong \prod_{v|p} \GL_2(F_{\widetilde{v}})$ commuting with $\bT^S$.

 Recall the automorphic representations of $G(\bA_{F^+})$ are the irreducible constituants of the $\bC$-vector space of functions $f: G(F^+)\backslash G(\bA_{F^+})\ra \bC$, which are
\begin{itemize}
  \item $\cC^{\infty}$ when restricted to $G(F^+\otimes_{\Q} \bR)$,
  \item locally constant when restricted to $G(\bA_{F^+}^{\infty})$,
  \item $G(F^+\otimes_{\Q} \bR)$-finite,
\end{itemize}
where $G(\bA_{F^+})$ acts on this space via right translation. An automorphic representation $\pi$ is isomorphic to $\pi_{\infty}\otimes_{\bC} \pi^{\infty}$ where $\pi_{\infty}=W_{\infty}$ is an irreducible algebraic representation of $G(F^+\otimes_{\Q} \bR)$ over $\bC$ and $\pi^{\infty}\cong \Hom_{G(F^+\otimes_{\Q}\bR)}(W_{\infty}, \pi)\cong \otimes'_{v}\pi_v$ is an irreducible smooth representation of $G(\bA_{F^+}^{\infty})$. The algebraic representation $W_{\infty}$ is defined over $\overline{\Q}$ via $\iota_{\infty}$, and we denote by $W_p$ its base change to $\overline{\Q_p}$, which is thus an irreducible algebraic representation of $G(F^+\otimes_{\Q} \Q_p)$ over $\overline{\Q_p}$. Via the decomposition $G(F^+\otimes_{\Q} \Q_p)\xrightarrow{\sim} \prod_{v\in \Sigma_p} G(F^+_v)$, one has $W_p\cong \otimes_{v\in \Sigma_p} W_v$ where $W_v$ is an irreducible algebraic representation of $G(F^+_v)$. One can also prove $\pi^{\infty}$ is defined over a number field via $\iota_{\infty}$ (e.g. see \cite[\S~6.2.3]{BCh}). Denote by $\pi^{\infty,p}:=\otimes'_{v\nmid p} \pi_v$, thus $\pi \cong \pi^{\infty,p}\otimes_{\overline{\Q}} \pi_p$. Let $m(\pi)\in \Z_{\geq 1}$ be the multiplicity of $\pi$ in the space of functions as above.
\begin{proposition}[$\text{\cite[Prop. 5.1]{Br13II}}$]\label{prop: gln-stm}
One has a $G(F^+ \otimes_{\Q} \Q_p)  \times \bT^S$-invariant isomorphism
\begin{equation*}
  \widehat{S}(U^p,E)^{\lalg} \otimes_E \overline{\Q_p} \cong \bigoplus_{\pi}\Big((\pi^{\infty,p})^{U^p} \otimes_{\overline{\Q}}(\pi_p \otimes_{\overline{\Q}} W_p)\Big)^{\oplus m(\pi)},
\end{equation*}
where   $\widehat{S}(U^p,E)^{\lalg}$ denotes the locally algebraic subrepresentation of $\widehat{S}(U^p,E)$, $\pi\cong \pi_{\infty}\otimes_{\bC} \pi^{\infty}$ runs through the automorphic representations of $G(\bA_{F^+})$ and $W_p$ is associated to $\pi_{\infty}$ as above.
\end{proposition}

Let $G_p:=\prod_{v|p} \GL_2(F_{\widetilde{v}})$ (which is isomorphic to $G(F^+\otimes_{\Q} \Q_p)$), $B_p:=\prod_{v|p} B(F_{\widetilde{v}})$ with $B(F_{\widetilde{v}})$ the Borel subgroup of upper triangular matrices, $\overline{B}_p$ the opposite of $B_p$, $N_p\cong \prod_{v|p} N(F_{\widetilde{v}})$ the unipotent radical of $B_p$, and  $T_p\cong \prod_{v|p} T(F_{\widetilde{v}})=:\prod_{v|p} T_{\widetilde{v}}$ the Levi subgroup of $B_p$ (with $T$ the subgroup of $\GL_2$ of diagonal matrices). Let $K_p:=\prod_{v|p} \GL_2(\co_{F_{\widetilde{v}}})$. For a closed subgroup $H$ of $G_p$, denote by $H^0:=H\cap K_p$. Denote by $\ug_p$, $\ub_p$, $\overline{\ub}_p$, $\fn_p$, $\ft_p$, $\ug_{\widetilde{v}}$, $\ub_{\widetilde{v}}$, $\overline{\ub}_{\widetilde{v}}$, $\fn_{\widetilde{v}}$, $\ft_{\widetilde{v}}$ the Lie algebras of $G_p$, $B_p$, $\overline{B}_p$, $N_p$, $T_p$, $\GL_2(F_{\widetilde{v}})$, $B(F_{\widetilde{v}})$, $\overline{B}(F_{\widetilde{v}})$, $N(F_{\widetilde{v}})$, $T(F_{\widetilde{v}})$ respectively. Recall one has an isomorphism $\ug_p\otimes_{\Q_p} E\cong \prod_{v|p} \prod_{\sigma\in \Sigma_{\widetilde{v}}} \ug_{\widetilde{v}}\otimes_{F_{\widetilde{v}},\sigma}E$, for $J\subseteq \Sigma_p$, put $\ug_J:=\prod_{v|p} \prod_{\sigma\in J_{\widetilde{v}}} \ug_{\widetilde{v}}\otimes_{F_{\widetilde{v}},\sigma}E$ with $J_{\widetilde{v}}:=J\cap \Sigma_{\widetilde{v}}$. Similarly, we have Lie algebras over $E$: $\ub_J$, $\overline{\ub}_J$ etc.

A weight of $\ft_p\otimes_{\Q_p} E$ (with values in $\overline{E}$) will be denoted by $\ul{\lambda}_{\Sigma_p}=(\lambda_{1,\sigma},\lambda_{2,\sigma})_{\sigma\in \Sigma_p}\in \overline{E}^{2|\Sigma_p|}$ with
\begin{equation*}\ul{\lambda}_{\Sigma_p}\big(\prod_{\sigma\in \Sigma_p}\diag(a_{\sigma}, d_{\sigma})\big)=\sum_{\sigma\in \Sigma_p} a_{\sigma}\lambda_{1,\sigma}+d_{\sigma} \lambda_{2,\sigma}.\end{equation*}
For $J\subseteq \Sigma_p$, let $\ul{\lambda}_J:=(\lambda_{1,\sigma},\lambda_{2,\sigma})_{\sigma\in J}$, which we view as a weight of $\ft_p\otimes_{\Q_p} E$ via
\begin{equation*}
\ul{\lambda}_{J}\big(\prod_{\sigma\in \Sigma_p}\diag(a_{\sigma}, d_{\sigma})\big)=\sum_{\sigma\in J} a_{\sigma}\lambda_{1,\sigma}+d_{\sigma} \lambda_{2,\sigma}.
\end{equation*}
We call $\ul{\lambda}_J$ dominant (with respect to $\ub_p$) if $\lambda_{1,\sigma}-\lambda_{2,\sigma}\in \Z_{\geq 0}$ for all $\sigma \in J$. If $\ul{\lambda}_J$ is integral, i.e. $\lambda_{i,\sigma}\in \Z$ for all $\sigma\in J$,  and dominant, then there exists a unique irreducible algebraic (and locally $J$-analytic) representation $L(\ul{\lambda}_J)$ of $G_p$ over $E$ with highest weight $\ul{\lambda}_J$ and one has $L(\ul{\lambda}_J)\cong \otimes_{\sigma\in J} L(\ul{\lambda}_{\{\sigma\}})$. For a locally $\Q_p$-analytic representation $V$ of $G_p$, $\ul{\lambda}_J$ integral and dominant, put
\begin{equation}\label{equ: cclg-lam}V(\ul{\lambda}_J):=\big(V \otimes_E L(\ul{\lambda}_J)'\big)^{\Sigma_p\setminus J-\an} \otimes_E L(\ul{\lambda}_J)
\end{equation}
(where $L(\ul{\lambda}_J)'$ denotes the dual of $L(\ul{\lambda}_J)$, ``$\Sigma_p\setminus J-\an$'' denotes the locally $\Sigma_p\setminus J$-analytic vectors, see \S~\ref{sec: cclg-A.2} below) which is in fact a subrepresentation of $V$.

For a locally $\Q_p$-analytic character $\delta=\delta_1\otimes \delta_2$ of $T_p$ over $E$, let $\delta_{\widetilde{v}}=\delta_{1,\widetilde{v}}\otimes \delta_{2,\widetilde{v}}:=\delta|_{T_{\widetilde{v}}}$ for $v|p$. Let $$\wt(\delta)=(\wt(\delta)_{1,\sigma},\wt(\delta)_{2,\sigma})_{\sigma \in \Sigma_p}:=(\wt(\delta_1)_{\sigma}, \wt(\delta_2)_{\sigma})_{\sigma\in \Sigma_p}$$ be the induced weight of $\ft_p\otimes_{\Q_p} E$. For an  integral weight $\ul{\lambda}_{\Sigma_p}$ of $\ft_p\otimes_{\Q_p} E$, denote by $\delta_{\ul{\lambda}_{\Sigma_p}}$ the algebraic character of $T_p$ with weight $\ul{\lambda}_{\Sigma_p}$, i.e. $\delta_{\ul{\lambda}_{\Sigma_p}}=\prod_{\sigma\in \Sigma_p} \sigma^{\lambda_{1,\sigma}}\otimes \sigma^{\lambda_{2,\sigma}}$. A locally algebraic character $\delta$ of $T_p$ (resp. of $T_{\widetilde{v}}$ for $v|p$) is called spherically algebraic if $\delta\delta_{\wt(\delta)}^{-1}$ is unramified.


Denote by $\cT_p$ (resp. $\cT_{\widetilde{v}}$) the rigid space over $E$ parameterizing locally $\Q_p$-analytic characters of $T_p$ (resp. of $T_{\widetilde{v}}$), thus $\cT_p\cong \prod_{v|p} \cT_{\widetilde{v}}$. For $J\subset \Sigma_p$ (resp. $J_{\widetilde{v}}:=J\cap \Sigma_{\widetilde{v}}$), denote by $\cT_{p,J}$ \big(resp. $\cT_{\widetilde{v},J_{\widetilde{v}}}$\big) the rigid closed subspace of $\cT_p$ \big(resp. of $\cT_{\widetilde{v}}$\big) parameterizing locally $J$-analytic (resp. locally $J_{\widetilde{v}}$-analytic) characters of $T_p$ (resp. of $T_{\widetilde{v}}$), thus $\cT_{p,J}\cong \prod_{v|p} \cT_{\widetilde{v},J_{\widetilde{v}}}$. The rigid space $\cT_{\widetilde{v},J_{\widetilde{v}}}$ is smooth and equidimensional of dimension $2+2|J_{\widetilde{v}}|$, thus $\cT_{p,J}$ is smooth and equidimensional of dimension $2|S_p|+2|J|$. Indeed, one has a smooth morphism (e.g. see \cite[\S~6.1.4]{Ding})
\begin{equation}\label{equ: cclg-arv}
  \cT_{\widetilde{v}} \lra (\bA^1\times \bA^1)^{|\Sigma_{\widetilde{v}}|}, \ \chi\mapsto \wt(\chi),
\end{equation}
and $\cT_{\widetilde{v},J_{\widetilde{v}}}$ is just the preimage of $\big\{(\lambda_{1,\sigma}, \lambda_{2,\sigma})_{\sigma\in \Sigma_{\widetilde{v}}}\ |\ \lambda_{i,\sigma}=0, \ \forall \sigma\in \Sigma_{\widetilde{v}}\setminus J_{\widetilde{v}}, \ i=1,2\big\}$. More generally, let $\ul{\lambda}_J=(\lambda_{1,\sigma},\lambda_{2,\sigma})_{\sigma\in J}\in E^{2|J|}$, $\ul{\lambda}_{J_{\widetilde{v}}}:=(\lambda_{1,\sigma},\lambda_{2,\sigma})_{\sigma\in J_{\widetilde{v}}}\in E^{2|J_{\widetilde{v}}|}$ for $v|p$, denote by $\cT_{\widetilde{v}}(\ul{\lambda}_{J_{\widetilde{v}}})$ the preimage of   $\big\{(\lambda'_{1,\sigma}, \lambda'_{2,\sigma})_{\sigma\in \Sigma_{\widetilde{v}}}\ |\ \lambda'_{i,\sigma}=\lambda_{i,\sigma}, \ \forall \sigma\in J_{\widetilde{v}}, \ i=1,2\big\}$ via (\ref{equ: cclg-arv}), and put $\cT_{p}(\ul{\lambda}_J):=\prod_{v|p} \cT_{\widetilde{v}}(\ul{\lambda}_{J_{\widetilde{v}}})$. Note for any locally $\Q_p$-analytic character $\delta$ of $T_p$ over $E$ with $\wt(\delta)_{i,\sigma}=\lambda_{i,\sigma}$ for $\sigma\in J$, $i=1,2$, the isomorphism
\begin{equation*}
  \cT_{p} \xlongrightarrow{\sim} \cT_p, \  \delta'\mapsto \delta'\delta,
\end{equation*}
induces an isomorphism $\cT_{p,\Sigma_p\setminus J}\xrightarrow{\sim} \cT_{p}(\ul{\lambda}_J)$. Let $T_p^0:=K_p \cap T_p$, and denote by $\cT_p^0$ the rigid space over $E$ parameterizing locally $\Q_p$-analytic characters of $T_p^0$, thus the restriction map induces a projection
\begin{equation*}
  \cT_p \twoheadlongrightarrow \cT_p^0.
\end{equation*}
Denote by $\cT_{p,J}^0$, $\cT_{\widetilde{v}}^0$, $\cT_{\widetilde{v},J_{\widetilde{v}}}^0$, $\cT_{p}^0(\ul{\lambda}_J)$, $\cT_{\widetilde{v}}^0(\ul{\lambda}_{J_{\widetilde{v}}})$ the image in $\cT_p^0$ of $\cT_{p,J}$, $\cT_{\widetilde{v}}$, $\cT_{\widetilde{v},J_{\widetilde{v}}}$, $\cT_{p}(\ul{\lambda}_J)$, $\cT_{\widetilde{v}}(\ul{\lambda}_{J_{\widetilde{v}}})$ respectively.

Let $V$ be an $E$-vector space equipped with an $E$-linear action of $A$ (with $A$ a set of operators) and $\chi$ be a system of eigenvalues of $A$. Denote by $V[A=\chi]$ the $\chi$-eigenspace, $V\{A=\chi\}:=\{v\in V\ |\ \exists N\in \Z_{>0} \text{ such that } (a-\chi(a))^N v=0 \ \forall a\in A\}$. If $A$ is moreover an $\co_E$-algebra with an ideal $\fI$, denote by $V[\fI]$ the subspace of vectors killed by $\fI$, and $V\{\fI\}:=\varinjlim_{n} V[\fI^n]$.


\subsection{Eigenvarieties}\label{sec: cclg-3.2}We briefly recall  some properties of the eigenvariety of $G$ (with tame level $U^p$).

Let $\overline{\rho}$ be a $2$-dimensional continuous absolutely irreducible representation $\overline{\rho}$ of $\Gal_F$ over $k_E$ such that $\overline{\rho}^{\vee}\circ c \cong \overline{\rho}\otimes \overline{\varepsilon}$, $\overline{\rho}$ is unramified outside $S$ and $\widehat{S}(U^p,E)^{\lalg}_{\overline{\rho}}\neq 0$ (where $\widehat{S}(U^p,E)^{\lalg}_{\overline{\rho}}$ denotes the locally analytic subrepresentation of $\widehat{S}(U^p,E)_{\overline{\rho}}$ of $G_p$). Let $J\subseteq \Sigma_p$, $\ul{\lambda}_J=(\lambda_{1,\sigma},\lambda_{2,\sigma})_{\sigma\in J}\in \Z^{2|J|}$ with $\lambda_{1,\sigma}\geq \lambda_{2,\sigma}$ for all $\sigma\in J$. Consider $\widehat{S}(U^p,E)_{\overline{\rho}}^{\an}(\ul{\lambda}_J)$ (cf. (\ref{equ: cclg-lam}), where $\widehat{S}(U^p,E)_{\overline{\rho}}^{\an}$ is the locally $\Q_p$-analytic subrepresentation of $\widehat{S}(U^p,E)_{\overline{\rho}}$ of $G_p$), which is an admissible locally $\Q_p$-analytic representation of $G_p$ equipped with an action of $\bT^S$ which commutes with $G_p$. Note $\widehat{S}(U^p,E)_{\overline{\rho}}^{\an}(\ul{\lambda}_J)$ is in fact a closed subrepresentation of $\widehat{S}(U^p,E)_{\overline{\rho}}^{\an}$. Applying the Jacquet-Emerton functor (\cite{Em11}), we get an essentially admissible locally $\Q_p$-analytic representation of $T_p$: $J_{B_p}\big(\widehat{S}(U^p,E)_{\overline{\rho}}^{\an}(\ul{\lambda}_J)\big)$, which is also equipped with an action of $\bT^S$ commuting with $T_p$. By \cite[\S~6.4]{Em04}, to $J_{B_p}\big(\widehat{S}(U^p,E)_{\overline{\rho}}^{\an}(\ul{\lambda}_J)\big)$, one can naturally associate a coherent sheaf $\cM(U^p,\ul{\lambda}_J)_{\overline{\rho}}$ over $\cT_{p}$ such that
\begin{equation*}
  \Gamma\big(\cT_p, \cM(U^p, \ul{\lambda}_J)_{\overline{\rho}}\big) \xlongrightarrow{\sim} J_{B_p}\big(\widehat{S}(U^p,E)_{\overline{\rho}}^{\an}(\ul{\lambda}_J)\big)'
\end{equation*}
where $\Gamma(X,\cF)$ denotes the sections of a sheaf $\cF$ on a rigid space $X$, $(\cdot)'$ denotes the continuous dual equipped with strong topology. This association is functorial, so $\cM(U^p,\ul{\lambda}_J)_{\overline{\rho}}$ is equipped with an $\co(\cT_p)$-linear action of $\bT^S$. Using Emerton's method (cf. \cite[\S~2.3]{Em1}), we can construct an eigenvariety form the triplet $\big\{\cM(U^p,\ul{\lambda}_J)_{\overline{\rho}}, \co(\cT_p), \bT^S\big\}$:

\begin{theorem}
  There exists a rigid space $\cE(U^p, \ul{\lambda}_J)_{\overline{\rho}}$ over $E$ finite over $\cT_p$ and equipped with a morphism
  \begin{equation*}
    \co(\cT_p) \otimes_{\co_E} \bT^S \lra \co\big(\cE(U^p,\ul{\lambda}_J)_{\overline{\rho}}\big)
  \end{equation*}
  such that
  \begin{enumerate}
    \item a point of $\cE(U^p,\ul{\lambda}_J)_{\overline{\rho}}$ is determined by its induced closed point $\delta: T_p\ra \overline{E}^{\times}$ of $\cT_p$ and its induced system of eigenvalues $\fh: \bT^S \ra \overline{E}$, and will be denoted by $(\fh,\delta)$;
    \item $(\fh,\delta)\in \cE(U^p,\ul{\lambda}_J)_{\overline{\rho}}(\overline{E})$ if and only if the eigenspace $$\big(J_{B_p}\big(\widehat{S}(U^p,E)_{\overline{\rho}}^{\an}(\ul{\lambda}_J)\big)\otimes_E \overline{E}\big)[\bT^S=\fh, T_p=\delta]\neq 0.$$
  \end{enumerate}
\end{theorem}
\begin{remark}\label{rem: cclg-seeo}
  By definition (and \cite[Lem. 7.2.12]{Ding}),
  \begin{equation}\label{equ: JacLambda}J_{B_p}(\widehat{S}(U^p,E)_{\overline{\rho}}^{\an}(\ul{\lambda}_J))\cong J_{B_p}\big((\widehat{S}(U^p,E)^{\an}_{\overline{\rho}}\otimes_E L(\ul{\lambda}_J)')^{\Sigma_p\setminus J-\an}\big) \otimes_E \delta_{\ul{\lambda}_J},\end{equation} from which we deduce $\cM(U^p,\ul{\lambda}_J)_{\overline{\rho}}$ is supported on $\cT_p(\ul{\lambda}_J)$ and hence the natural morphism $\cE(U^p,\ul{\lambda}_J)_{\overline{\rho}}\ra \cT_p$ factors through $\cT_p(\ul{\lambda}_J)$.
\end{remark}
Recall a point $z=(\fh,\delta)$ of $\cE(U^p,\ul{\lambda}_J)_{\overline{\rho}}$ is called \emph{classical} if
\begin{equation*}
   \big(J_{B_p}\big(\widehat{S}(U^p,E)_{\overline{\rho}}(\ul{\lambda}_J)^{\lalg}\big)\otimes_E \overline{E}\big)[\bT^S=\fh, T_p=\delta]\neq 0.
\end{equation*}
\begin{theorem}\label{thm: cclg-eicl}
  The rigid space $\cE(U^p,\ul{\lambda}_J)_{\overline{\rho}}$ is equidimensional of dimension $2|\Sigma_p\setminus J|$, and the set of classical points is Zariski-dense in $\cE(U^p,\ul{\lambda}_J)_{\overline{\rho}}$ and accumulates at points $(\fh,\delta)$ (see \cite[\S~3.3.1]{BCh}) with $\delta$ locally algebraic.
\end{theorem}
\begin{remark}The proof of the theorem is omitted, since it is an easy variation of that in the patched eigenvariety case given below (and the same as in the case of eigenvarieties for unitary Shimura curves \cite[Prop. 7.2.30]{Ding}), where there are two key points:
 \begin{enumerate}\item a classicality result as Proposition \ref{prop: cclg-dtf} (see also \cite[Cor. 7.2.28]{Ding});
  \item there exists an open compact normal subgroup $H$ of $G_p$ such that $\widehat{S}(U^p,E)_{\overline{\rho}}^{\an}|_H\cong \cC^{\Q_p-\an}(H,E)$ as locally $\Q_p$-analytic representations of $H$,  where the latter denotes the space of locally $\Q_p$-analytic functions of $H$ equipped with the right regular action of $H$; this fact in particular allows \cite[Prop. 4.2.36]{Em11} to apply.
 \end{enumerate}
\end{remark}
Put $\cE(U^p)_{\overline{\rho}}:=\cE(U^p,\lambda_{\emptyset})_{\overline{\rho}}$. The following proposition follows easily from the natural $G_p\times \bT^S$-invariant injection $\widehat{S}(U^p,E)_{\overline{\rho}}^{\an}(\ul{\lambda}_J)\hookrightarrow \widehat{S}(U^p,E)_{\overline{\rho}}^{\an}$.
\begin{proposition}
  There exists a natural closed embedding
  \begin{equation*}
    \cE(U^p, \ul{\lambda}_J)_{\overline{\rho}} \hooklongrightarrow \cE(U^p)_{\overline{\rho}}, \ (\fh,\delta)\mapsto (\fh, \delta).
  \end{equation*}
\end{proposition}
By discussion in Remark \ref{rem: cclg-seeo}, we have a natural  commutative diagram
\begin{equation*}
  \begin{CD}\cE(U^p, \ul{\lambda}_J)_{\overline{\rho}} @>>> \cE(U^p)_{\overline{\rho}}  \\
  @VVV @VVV \\
  \cT_{p}(\ul{\lambda}_J) @>>> \cT_p
  \end{CD}.
\end{equation*}Note this diagram should \emph{not} be cartesian in general. Indeed, as we would see later in patched eigenvariety case, the difference between  $\cE(U^p,\ul{\lambda}_J)$ and $\cE(U^p,\ul{\lambda}_J)':=\cE(U^p)_{\overline{\rho}}\times_{\cT_p} \cT_p(\ul{\lambda}_J)$ somehow would be the key point for the existence of companion points.

Recall there exists a family of $\Gal_F$-representations on $\cE(U^p)_{\overline{\rho}}$, in particular, for any  point $z=(\fh,\delta)$ of $\cE(U^p)_{\overline{\rho}}$, there exists a $2$-dimensional continuous representation $\rho_z$ of $\Gal_{F}$ over $k(z)$, which is unramified outside $S$ and satisfies
\begin{equation*}
  \rho_z(\Frob_{\widetilde{v}})^2 -\fh(T_{\widetilde{v}}) \Frob_{\widetilde{v}}+\Norm(\widetilde{v}) \fh(S_{\widetilde{v}})=0
\end{equation*}
for $v\notin S$ completely decomposed in $F$, and $\widetilde{v}|v$, $\Frob_{\widetilde{v}}\in \Gal_{F_{\widetilde{v}}}$ is a geometric Frobenius. The (semi-simplification of the) reduction of $\rho_z$ modulo $\varpi_{k(z)}$ (a uniformizer of $k(z)$) is isomorphic to $\overline{\rho}$.
\begin{theorem}\label{thm: cclg-pen}Let $z\in \cE(U^p)_{\overline{\rho}}$. For $v|p$, the restriction $\rho_{z,\widetilde{v}}:=\rho_z|_{\Gal_{F_{\widetilde{v}}}}$ is trianguline. If $z\in \cE(U^p,\ul{\lambda}_J)_{\overline{\rho}}$ for $J\subseteq \Sigma_p$ (which implies $\wt(\delta)_{i,\sigma}=\lambda_{i,\sigma}$ for $\sigma\in J$, $i=1,2$), then for $v|p$, $\rho_{z,\widetilde{v}}$ is moreover $J_{\widetilde{v}}$-de Rham of Hodge-Tate weights $(-\lambda_{1,\sigma}, 1-\lambda_{2,\sigma})$ for $\sigma\in J_{\widetilde{v}}$ where $J_{\widetilde{v}}:=J\cap \Sigma_{\widetilde{v}}$.
\end{theorem}
\begin{remark}By Theorem \ref{thm: cclg-eicl}, the first part of Theorem \ref{thm: cclg-pen} follows from the global triangulation theory (\cite{KPX}, \cite{Liu}) applied to $\cE(U^p)_{\overline{\rho}}$;  and the second part follows from results of Shah \cite[Thm. 2]{Sha} applied to $\cE(U^p,\ul{\lambda}_J)_{\overline{\rho}}$.
\end{remark}
In particular, roughly speaking, $\cE(U^p,\ul{\lambda}_J)_{\overline{\rho}}$ gives a $J_{\widetilde{v}}$-de Rham family for $v|p$ while $\cE(U^p,\ul{\lambda}_J)_{\overline{\rho}}'$ should \emph{only} give a $J_{\widetilde{v}}$-Hodge-Tate family for $v|p$. We will show a more clear  picture in the patched eigenvariety case.



\subsection{Patched eigenvarieties}
\subsubsection{Patched Banach representation}Denote by $R_{\overline{\rho},S}$ the deformation ring (which is a complete noetherian  local $\co_E$-algebra with residue field $k_E$) which pro-represents the functor associating to  local artinian $\co_E$-algebras $A$ the sets of isomorphism classes of deformations $\rho_A$ of $\overline{\rho}$ over $A$ such that $\rho_A$ is  unramified outside $S$ and  $\rho_A^{\vee}\circ c\cong \rho_A \otimes \varepsilon$.
For a finite place $\widetilde{v}$ of $F$, denote by $\overline{\rho}_{\widetilde{v}}:=\overline{\rho}|_{\Gal_{F_{\widetilde{v}}}}$,  $R_{\overline{\rho}_{\widetilde{v}}}^{\bar{\square}}$  the maximal reduced and $p$-torsion free quotient of $R_{\overline{\rho}_{\widetilde{v}}}^{\square}$, and $R_{\overline{\rho}_S}^{\bar{\square}}:=\widehat{\otimes}_{v\in S} R_{\overline{\rho}_{\widetilde{v}}}^{\bar{\square}}$. Fix $g\in \Z_{\geq 1}$, let $R_{\infty}:=R_{\overline{\rho}_S}^{\bar{\square}}\llbracket x_1,\cdots, x_g \rrbracket$, $S_{\infty}:=\co\llbracket y_1,\cdots, y_q \rrbracket$ with $q=g+[F^+:\Q]+4|S|$, and $\fa=(y_1,\cdots, y_q)\subset S_{\infty}$.
In the following, we assume the so-called Taylor-Wiles hypothesis.
\begin{hypothesis}\label{hypo: cclg-TW}Suppose
  \begin{itemize}
  \item $p>2$,
  \item $F$ is unramified over $F^+$ and $G$ is quasi-split at all finite places of $F^+$,
  \item $U_v$ is hyperspecial when the finite place $v$ of $F^+$ is inert in $F$,
  \item $\overline{\rho}|_{\Gal_{F(\zeta_p)}}$ is adequate (\cite{Tho}).
\end{itemize}
\end{hypothesis}
By \cite[Prop. 6.7]{Tho}, $\widehat{S}(U^p,E)_{\overline{\rho}}$ is naturally equipped with a $R_{\overline{\rho},S}$-action. Moreover, the action of $R_{\overline{\rho},S}$ on $\widehat{S}(U^p,E)_{\overline{\rho}}$ factors through $R_{\overline{\rho},S} \twoheadrightarrow R_{\overline{\rho},\cS}$, where $R_{\overline{\rho},\cS}$ is the deformation ring associated to the deformation problem (as in \cite[\S~2.3]{CHT}, and we use the notation of \emph{loc. cit.})
\begin{equation*}
  \cS=\big(F/F^+, S, \widetilde{S}, \co_E, \overline{\rho}, \varepsilon^{-1} \delta_{F/F^+}^2, \{R_{\overline{\rho}_{\widetilde{v}}}^{\bar{\square}}\}_{v\in S}\big)
\end{equation*}
where $\widetilde{S}=\{\widetilde{v}\ |\ v\in S\}$, and $\delta_{F/F^+}$ is the quadratic character of $\Gal_{F^+}$ associated to the extension $F/F^+$. By shrinking $U_v$ for certain places $v\nmid p$ that split in $F$ (and hence enlarging $S$ consequently), we suppose\begin{equation*}
  G(F^+)\cap (hU^p K_p h^{-1})=\{1\}, \  \forall h\in G(\bA_{F^+}^{\infty})
\end{equation*}
where we also use $K_p$ to denote $\prod_{v|p} i_{\widetilde{v}}^{-1}(\GL_2(\co_{F_{\widetilde{v}}}))$ (which is independent of the choice of $\widetilde{v}$). By \cite{CEGGPS}, we have as in  \cite[Thm. 3.5]{BHS1}
\begin{enumerate}
  \item a continuous $R_{\infty}$-admissible unitary representation $\Pi_{\infty}$ of $G_p$ over $E$ together with a $G_p$-stable and $R_{\infty}$-stable unit ball $\Pi_{\infty}^o\subset \Pi_{\infty}$;
  \item a morphism of local $\co_E$-algebras $S_{\infty}\ra R_{\infty}$ such that $M_{\infty}:= \Hom_{\co_L}(\Pi_{\infty}^o, \co_E)$ is finite projective as $S_{\infty}\llbracket K_p\rrbracket$-module;
  \item a surjection $R_{\infty}/\fa R_{\infty}\twoheadrightarrow R_{\overline{\rho},\cS}$ and a  $G_p\times R_{\infty}/\fa R_{\infty}$-invariant isomorphism $\Pi_{\infty}[\fa]\cong \widehat{S}(U^p,E)_{\overline{\rho}}$, where $R_{\infty}$ acts on $\widehat{S}(U^p,E)_{\overline{\rho}}$ via $R_{\infty}/\fa R_{\infty}\twoheadrightarrow R_{\overline{\rho},\cS}$.
\end{enumerate}
\subsubsection{Patched eigenvariety and some stratifications}\label{sec: cclg-3.3.2}
Recall (cf. \cite[\S~3.1]{BHS1}) for an $R_{\infty}$-admissible representation $\Pi$ of $G_p$ over $E$,  a vector $v\in \Pi$ is called \emph{locally $R_{\infty}$-analytic} if it is locally $\Q_p$-analytic for the action of $\Z_p^s \times G_p$ with respect to a presentation $\co_E\llbracket \Z_p^s\rrbracket \twoheadrightarrow R_{\infty}$. And it is shown in \emph{loc. cit.}, this definition is independent of the choice of the presentation. As in \emph{loc. cit.}, denote by $\Pi^{R_{\infty}-\an}$ the subspace of locally $R_{\infty}$-analytic vectors.

Let $J\subseteq \Sigma_p$, $\ul{\lambda}_J:=(\lambda_{1,\sigma}, \lambda_{2,\sigma})_{\sigma\in J}\in \Z^{2|J|}$, and suppose $\lambda_{1,\sigma}\geq \lambda_{2,\sigma}$ for all $\sigma\in J$. Consider $\Pi_{\infty}^{R_{\infty}-\an}(\ul{\lambda}_J)$ (cf. (\ref{equ: cclg-lam})), which is a locally $\Q_p$-analytic $\text{U}(\ug_J)$-finite representation of $G_p$ (cf. \S~Appendix \ref{sec: cclg-A.2}), stable under $R_{\infty}$ and is moreover admissible as a locally $\Q_p$-analytic representation of $G_p\times \Z_p^s$ \big(where the action  $\Z_p^s$ is induced by that of $R_{\infty}$ via any presentation $\co_E\llbracket \Z_p^s \rrbracket \twoheadrightarrow R_{\infty}$\big). In fact, $\Pi_{\infty}^{R_{\infty}-\an}(\ul{\lambda}_J)$ is a closed subrepresentation of $\Pi_{\infty}^{R_{\infty}-\an}$ by Proposition \ref{prop: cclg-cnfp} (where the closedness follows from the same argument as in Corollary \ref{cor: cclg-clsu}). Note also that, for $J'\subseteq J$, by Lemma \ref{lem: cclg-j'js} (3), $\Pi_{\infty}^{R_{\infty}-\an}(\ul{\lambda}_J)$ is a closed subrepresentation of $\Pi_{\infty}^{R_{\infty}-\an}(\ul{\lambda}_{J'})$.

Applying Jacquet-Emerton functor, we get a locally $\Q_p$-analytic representation $$J_{B_p}\big(\Pi_{\infty}^{R_{\infty}-\an}(\ul{\lambda}_J)\big)$$ of $T_p$ equipped with a continuous action of $R_{\infty}$, which is moreover essentially admissible as locally $\Q_p$-analytic representation of $T_p\times \Z_p^s$. Let $\fX_{\infty}:=(\Spf R_{\infty})^{\rig}$, and $R_{\infty}^{\rig}:=\co\big(\fX_{\infty}\big)$. The strong dual $J_{B_p}\big(\Pi_{\infty}^{R_{\infty}-\an}(\ul{\lambda}_J)\big)'$ is thus a coadmissible $R_{\infty}^{\rig}\widehat{\otimes}_E \co(\cT_p)$-module, which corresponds to a coherent sheaf $\cM_{\infty}(\ul{\lambda}_J)$ over $\fX_{\infty}\times \cT_p$ such that
\begin{equation*}
   \Gamma\big(\fX_{\infty}\times \cT_p, \cM_{\infty}(\ul{\lambda}_J)\big) \cong  J_{B_p}\big(\Pi_{\infty}(\ul{\lambda}_J)\big)'.
\end{equation*}
Let $X_p(\overline{\rho},\ul{\lambda}_J)$ be the support of $\cM_{\infty}(\ul{\lambda}_J)$ on $\fX_{\infty}\times \cT_p$. In particular, for $x=(y,\delta)\in \fX_{\infty}\times \cT_p$, $x\in X_p(\overline{\rho},\ul{\lambda}_J)$ if and only if the corresponding eigenspace
\begin{equation*}
J _{B_p}\big(\Pi_{\infty}^{R_{\infty}-\an}(\ul{\lambda}_J)\big)[\fm_y,T_p=\delta]\neq 0,
\end{equation*}
where $\fm_y$ denote the maximal ideal of $R_{\infty}[\frac{1}{p}]$ corresponding to $y$. Similarly as in (\ref{equ: JacLambda}), we have \begin{equation}\label{equ: cclg-pgyy}J_{B_p}\big(\Pi_{\infty}^{R_{\infty}-\an}(\ul{\lambda}_J)\big)\cong J_{B_p}\big((\Pi_{\infty}^{R_{\infty}-\an}\otimes_E L(\ul{\lambda}_J)')^{\Sigma_p\setminus J-\an}\big)\otimes_E \delta_{\ul{\lambda}_J},\end{equation} it is straightforward to see the action of $\co(\cT_p)$ on $J_{B_p}\big(\Pi_{\infty}^{R_{\infty}-\an}\big)$ factors through $\co(\cT_p(\ul{\lambda}_J))$, and hence $\cM_{\infty}(\ul{\lambda}_J)$ is supported on $\fX_{\infty}\times \cT_p(\ul{\lambda}_J)$. So the natural injection $X_p(\overline{\rho},\ul{\lambda}_J)\hookrightarrow \fX_{\infty}\times \cT_p$ factors through $\fX_{\infty}\times \cT_p(\ul{\lambda}_J)$.

Let $J'\subseteq J$, one has a natural projection of coadmissible $\co\big(\fX_{\infty}\times \cT_p\big)$-modules $\cM(\overline{\rho},\ul{\lambda}_{J'})\twoheadrightarrow \cM(\overline{\rho},\ul{\lambda}_J)$ induced by the natural inclusion $J_{B_p}\big(\Pi_{\infty}^{R_{\infty}-\an}(\ul{\lambda}_J)\big) \hookrightarrow J_{B_p}\big(\Pi_{\infty}^{R_{\infty}-\an}(\ul{\lambda}_{J'})\big)$. Consequently,  $X_p(\overline{\rho},\ul{\lambda}_J)$ is naturally a rigid closed subspace of $X_p(\overline{\rho},\ul{\lambda}_{J'})$, and we have a commutative diagram
\begin{equation*}
  \begin{CD}
    X_p(\overline{\rho},\ul{\lambda}_J) @>>>  X_p(\overline{\rho},\ul{\lambda}_{J'})  \\
    @VVV @VVV \\
    \cT_{p}(\ul{\lambda}_{J}) @>>>  \cT_{p}(\ul{\lambda}_{J'})
  \end{CD}.
\end{equation*}
Let $X_p(\overline{\rho}):=X_p(\overline{\rho},\ul{\lambda}_{\emptyset})$ which is the so-called \emph{patched eigenvariety}  constructed in \cite{BHS1}. Put $X_p(\overline{\rho},\ul{\lambda}_J)':=X_p(\overline{\rho})\times_{\cT_p} \cT_p(\ul{\lambda}_J)$, $X_p(\overline{\rho},\ul{\lambda}_J,J'):=X_p(\overline{\rho},\ul{\lambda}_{J'})\times_{\cT_p(\ul{\lambda}_{J'})} \cT_p(\ul{\lambda}_J)$. We have thus a commutative diagram  (compare with (\ref{equ: cclg-rdJ})):
\begin{equation}\label{equ: cclg-a1j}
   \begin{CD}X_p(\overline{\rho},\ul{\lambda}_J)@>>> X_p(\overline{\rho},\ul{\lambda}_{J},J') @>>> X_p(\overline{\rho},\ul{\lambda}_{J'})  @>>>X_p(\overline{\rho},\ul{\lambda}_{J'})' @>>> X_p(\overline{\rho}) \\
   @VVV @VVV @VVV @VVV @VVV \\
   \cT_{p}(\ul{\lambda}_{J}) @>>> \cT_{p}(\ul{\lambda}_{J}) @>>> \cT_{p}(\ul{\lambda}_{J'}) @>>> \cT_{p}(\ul{\lambda}_{J'}) @>>> \cT_p
   \end{CD}
\end{equation}
where the horizontal maps are closed embeddings, and the second and fourth square are cartesian.
\subsubsection{Structure of $X_p(\overline{\rho},\ul{\lambda}_J)$}We fix an isomorphism $\co_E\llbracket \Z_p^q \rrbracket \cong S_{\infty}$. Since $\Pi_{\infty}^{\vee}$ is a finite projective  $S_{\infty}\llbracket K_p \rrbracket[\frac{1}{p}]$-module, so is $\big(\Pi_{\infty} \otimes_E L(\ul{\lambda}_J)'\big)^{\vee}$. Thus for any pro-$p$ compact open subgroup $K_p'$ of $K_p$, one has
\begin{equation*}
  \big(\Pi_{\infty}\otimes_E L(\ul{\lambda}_J)'\big)|_{\Z_p^q \times K_p'} \cong \cC(\Z_p^q \times K_p', E)^{\oplus r} \cong \big(\cC(\Z_p^q,E)\widehat{\otimes}_E \cC(K_p',E)\big)^{\oplus r},
\end{equation*}
\begin{equation}\label{equ: cclg-PiJ}\big(\Pi_{\infty}^{R_{\infty}-\an}\otimes_E L(\ul{\lambda}_J)'\big)^{\Sigma_p \setminus J-\an}\big|_{\Z_p^q \times K_p'}\cong \big(\cC^{\an}(\Z_p^q, E) \widehat{\otimes}_E \cC^{\Sigma_p\setminus J-\an}(K_p',E)\big)^{\oplus r'},\end{equation}
for certain $r, r'\in \Z_{>0}$.
For $v|p$, let $z_{\widetilde{v}}:=\begin{pmatrix} \varpi_{\widetilde{v}} & 0 \\ 0 & 1\end{pmatrix}\in T_{\widetilde{v}}$, $T_p^+$ be the monoid in $T_p$ generated by $T_p^0$ and $z_{\widetilde{v}}$ for all $v|p$, and let $z:=(z_{\widetilde{v}})_{v|p} \in T_p$. One has a natural projection
\begin{equation}\label{equ: cclg-projwt}\cT_p\twoheadlongrightarrow \cT_{p}^0\times \bG_m,\ \delta\mapsto \delta|_{T_p^0} \times \delta(z)
 \end{equation}which induces projections $\cT_{p,\Sigma_p\setminus J}\twoheadrightarrow \cT^0_{p,\Sigma_p\setminus J}\times \bG_m$, $\cT_p(\ul{\lambda}_J) \twoheadrightarrow \cT^0_p(\ul{\lambda}_J)\times \bG_m$. Denote by $\cW_{\infty}:=(\Spf S_{\infty})^{\rig}\times  \cT_p^0$, $\cW_{\infty,\Sigma_p\setminus J}:=(\Spf S_{\infty})^{\rig}\times \cT_{p,\Sigma_p\setminus J}^0$, and $\cW_{\infty}(\ul{\lambda}_J):=(\Spf S_{\infty})^{\rig}\times \cT_{p}^0(\ul{\lambda}_J)$. The natural morphism $(\Spf R_{\infty})^{\rig} \ra (\Spf S_{\infty})^{\rig}$ together with (\ref{equ: cclg-projwt}) give thus a morphism
\begin{equation*}\kappa: \fX_{\infty}\times \cT_p\lra \cW_{\infty}\times \bG_m.
\end{equation*}
Consider $J_{B_p}\big((\Pi_{\infty}^{R_{\infty}-\an}\otimes_E L(\ul{\lambda}_J)')^{\Sigma_p\setminus J-\an}\big)$, which is a locally $\Sigma_p\setminus J$-analytic representation of $T_p$ equipped with an action of $R_{\infty}$ (commuting with $T_p$), and is essentially admissible as locally analytic representation of $T_p\times \Z_p^q$.  By (\ref{equ: cclg-PiJ}) and (the proof of) \cite[Prop. 4.2.36]{Em11}, $ J_{B_p}\big((\Pi_{\infty}^{R_{\infty}-\an}\otimes_E L(\ul{\lambda}_J)')^{\Sigma_p\setminus J-\an}\big)^{\vee}$ is moreover a coadmissible $\co\big(\cW_{\infty}\times \bG_m\big)$-module, and hence is a coadmissible $\co\big(\cW_{\infty,\Sigma_p\setminus J}\times \bG_m\big)$-module \big(since the action of $\co(\cW_{\infty})$ factors though $\co(\cW_{\infty,\Sigma_p\setminus J})$\big).
By the isomorphism (\ref{equ: cclg-pgyy}), $J_{B_p}\big(\Pi_{\infty}^{R_{\infty}-\an}(\ul{\lambda}_J)\big)^{\vee}$ is thus a coadmissible $\co\big(\cW_{\infty}(\ul{\lambda}_J)\times \bG_m\big)$-module, in other words, the push forward $\cN(\overline{\rho},\ul{\lambda}_J):=\kappa_* \cM(\overline{\rho},\ul{\lambda}_J)$ is a coherent sheaf on $\cW_{\infty}(\ul{\lambda}_J)\times \bG_m$. We have as in \cite[Prop. 5.3]{BHS1} (recall $N_p^0=N_p\cap K_p$):
\begin{lemma}\label{lem: strucPaEv}(1) There exist an admissible covering of $\cW_{\infty,\Sigma_p \setminus J}$ by affinoid opens $U_1\subset U_2\subset \cdots \subset U_h \subset \cdots $, and a commutative diagram
\begin{equation*}
  \begindc{\commdiag}[400]
  \obj(0,1)[a]{$ \Big(\big(\big(\Pi_{\infty}^{R_{\infty}-\an}\otimes_E L(\ul{\lambda}_J)'\big)^{\Sigma_p\setminus J-\an}\big)^{N_p^0}\Big)^{\vee}$}
  \obj(3,1)[b]{$\cdots$}
  \obj(4,1)[c]{$V_{h+1}$}
  \obj(6,1)[d]{$V_{h+1}\otimes_{A_{h+1}} A_h$}
  \obj(8,1)[e]{$V_n$}
    \obj(0,0)[a']{$ \Big(\big(\big(\Pi_{\infty}^{R_{\infty}-\an}\otimes_E L(\ul{\lambda}_J)'\big)^{\Sigma_p\setminus J-\an}\big)^{N_p^0}\Big)^{\vee}$}
  \obj(3,0)[b']{$\cdots$}
  \obj(4,0)[c']{$V_{h+1}$}
  \obj(6,0)[d']{$V_{h+1}\otimes_{A_{h+1}} A_h$}
  \obj(8,0)[e']{$V_h$}
  \mor{a}{b}{}
  \mor{b}{c}{}
  \mor{c}{d}{}
  \mor{d}{e}{$\beta_h$}
    \mor{a'}{b'}{}
  \mor{b'}{c'}{}
  \mor{c'}{d'}{}
  \mor{d'}{e'}{$\beta_h$}
  \mor{a}{a'}{$z$}
  \mor{c}{c'}{$z_{h+1}$}
  \mor{d}{d'}{$z_{h+1}\otimes \id$}[\atright,\solidarrow]
  \mor{e}{e'}{$z_h$}
  \mor{e}{d'}{$\alpha_h$}
  \enddc,
\end{equation*}
where $A_h:=\Gamma(U_h, \co_{\cW_{\infty,\Sigma_p\setminus J}})$, $V_h$ is a Banach $A_h$-module satisfying the condition (Pr) of \cite{Bu}, $z_h$ and $\beta_h$ are  $A_h$-linear and compact, $\alpha_h$ is continuous $A_h$-linear, and there is an isomorphism of $\co(\cW_{\infty,\Sigma_p\setminus J})$-modules
  \begin{equation*}\Big(\big(\big(\Pi_{\infty}^{R_{\infty}-\an}\otimes_E L(\ul{\lambda}_J)'\big)^{\Sigma_p\setminus J-\an}\big)^{N_p^0}\Big)^{\vee}\xlongrightarrow{\sim} \varprojlim_{h} V_h\end{equation*}
  which commutes with the Hecke action of $z$ on the left object, and that of $(z_h)_{h\in \Z_{\geq 1}}$ on the right.

  (2) The statement in (1) also holds with the rigid space $\cW_{\infty,\Sigma_p\setminus J}$ replaced by $\cW_{\infty}(\ul{\lambda}_J)$, and $\big(\big(\Pi_{\infty}^{R_{\infty}-\an}\otimes_E L(\ul{\lambda}_J)'\big)^{\Sigma_p\setminus J-\an}\big)^{N_p^0}$ replaced by $\big(\Pi_{\infty}^{R_{\infty}-\an}(\ul{\lambda}_J)\big)^{N_p^0}$.
\end{lemma}
\begin{proof}
  (1) follows from (\ref{equ: cclg-PiJ}) as in \cite[Prop. 5.3]{BHS1} (see also \cite[Prop. 4.2.36]{Em11}). As in the proof   of \cite[Lem. 7.2.12]{Ding}, we have a natural $R_{\infty}\times T_p^+$-equivariant isomorphism
\begin{equation*}
  \big(\Pi_{\infty}^{R_{\infty}-\an}(\ul{\lambda}_J)\big)^{N_p^0} \xlongrightarrow{\sim} \big(\big(\Pi_{\infty}^{R_{\infty}-\an}\otimes_E L(\ul{\lambda}_J)'\big)^{\Sigma_p\setminus J-\an}\big)^{N_p^0}\otimes_{E} \delta_{\ul{\lambda}_J},
\end{equation*}
which together with (1) imply (2).
\end{proof}Denote by $Z(\overline{\rho},\ul{\lambda}_J)$ the support of $\cN(\overline{\rho},\ul{\lambda}_J)$ in $\cW_{\infty}(\ul{\lambda}_J) \times \bG_m$.  The morphism $\kappa$ induces thus a morphism $\kappa: X(\overline{\rho},\ul{\lambda}_J) \ra Z(\overline{\rho},\ul{\lambda}_J)$. Denote by $g: Z(\overline{\rho},\ul{\lambda}_J)\ra \cW_{\infty}(\ul{\lambda}_J)$ and $\omega_{\infty}: X_p(\overline{\rho},\ul{\lambda}_J) \ra \cW_{\infty}(\ul{\lambda}_J)$ the natural morphisms. Thus $\omega_{\infty}=g\circ \kappa$.  As in \cite[Lem. 2.10, Prop. 3.11]{BHS1}, one can deduce from Lemma \ref{lem: strucPaEv}:
\begin{proposition}\label{prop: cclg-Fen}
  (1) $Z(\overline{\rho},\ul{\lambda}_J)$ is a Fredholm hypersurface in $\cW_{\infty}(\ul{\lambda}_J)\times \bG_m$, and there exists an admissible covering $\{U_i\}_{i\in I}$ of $Z(\overline{\rho},\ul{\lambda}_J)$ by affinoids $U_i$ such that  the morphism $g$ induces a finite surjective map from $U_i$ to an affinoid open $W_i$ of $\cW_{\infty}(\ul{\lambda}_J)$ and $U_i$ is a connected component of $g^{-1}(W_i)$. Moreover, for $i\in I$, $\Gamma(U_i,\cN_{\infty}(\ul{\lambda}_J))$ is a finite projective $\Gamma(W_i,\co_{\cW_{\infty}(\ul{\lambda}_J)})$-module.

  (2) There exists an admissible covering $\{\cU_i\}_{i\in I}$ of $X_p(\overline{\rho},\ul{\lambda}_J)$ by affinoids $\cU_i$ such that for all $i$ there exists an open affinoid $W_i$ of $\cW_{\infty}(\ul{\lambda}_J)$ such that the morphism $\omega_{\infty}$ induces a finite surjective morphism from each irreducible component of $\cU_i$  onto $W_i$ and that $\Gamma(\cU_i,\co_{X_p(\overline{\rho})})$ is isomorphic to a $\Gamma(W_i,\co_{\cW_{\infty}(\ul{\lambda}_J)})$-subalgebra of the endomorphism ring of a finite projective $\Gamma(W_i,\co_{\cW_{\infty}(\ul{\lambda}_J)})$-module.
\end{proposition}
\begin{remark}\label{rem: cclg-neii}
 As in the proof of \cite[Prop. 3.11]{BHS1}, one can take $\cU_i$ (in (2)) to be $\kappa^{-1}(U_i)$ (with $U_i$ as in (1)), and then $\Gamma(\cU_i,X_p(\overline{\rho}, \ul{\lambda}_J))$ is just the $\Gamma(W_i,\co_{\cW_{\infty}(\ul{\lambda}_J)})$-subalgebra of the endomorphism ring of the finite projective $\Gamma(W_i,\co_{\cW_{\infty}(\ul{\lambda}_J)})$-module $\Gamma(U_i,\cN_{\infty}(\ul{\lambda}_J))\cong \Gamma(\cU_i, \cM_{\infty}(\ul{\lambda}_J))$ generated by the operators in $R_{\infty}\times T_p$.
\end{remark}
\begin{corollary}\label{cor: cclg-o0e}
  (1) The rigid space $X_p(\overline{\rho}, \ul{\lambda}_J)$ is equidimensional of dimension
   \begin{equation*}
     g+4|S|+3[F^+:\Q]-2|J|,
   \end{equation*}locally finite over $\cW_{\infty}(\ul{\lambda}_J)$ and does not have embedded components.


  (2) The coherent sheaf $\cM_{\infty}(\ul{\lambda}_J)$ is Cohen-Macaulay over $X_p(\overline{\rho}, \ul{\lambda}_J)$.
\end{corollary}
\begin{proof}
  (1) follows by the same argument as in \cite[Cor. 3.12]{BHS1} (see also \cite[Prop. 6.4.2]{Che}). (2) follows by the same argument as in \cite[Lem. 3.8]{BHS2}.
\end{proof}
\begin{corollary}Let $J'\subseteq J$, we have

(1) $\dim X_p(\overline{\rho},\ul{\lambda}_{J})'=g+4|S|+3[F^+:\Q]-2|J|$ \big(of the same dimension of $X_p(\overline{\rho},\ul{\lambda}_{J})$\big).

  (2) $\dim X_p(\overline{\rho},\ul{\lambda}_{J},J')=g+|4|S|+3[F^+:\Q]-2|J|$ \big(of the same dimension of $X_p(\overline{\rho},\ul{\lambda}_{J})$\big).
\end{corollary}
\begin{proof} (1) is  a special case of (2) since $X_p(\overline{\rho}, \ul{\lambda}_J)'=X_p(\overline{\rho}, \ul{\lambda}_J, \emptyset)$.
We have $X_p(\overline{\rho},\ul{\lambda}_{J},J')\cong X_p(\overline{\rho},\ul{\lambda}_{J'})\times_{\cT_p^0(\ul{\lambda}_{J'})} \cT_p^0(\ul{\lambda}_{J}) \cong X_p(\overline{\rho}, \ul{\lambda}_{J'})\times_{\cW_{\infty}(\ul{\lambda}_{J'})} \cW_{\infty}(\ul{\lambda}_{J})$. Let $\cU_i\subseteq X_p(\overline{\rho}, \ul{\lambda}_{J'})$, $W_i\subseteq \cW_{\infty}(\ul{\lambda}_{J'})$ be as in Remark \ref{rem: cclg-neii} (applied with $J$ replaced by $J'$), it is sufficient to show $\cU_i \times_{\cW_{\infty}(\ul{\lambda}_{J'})} \cW_{\infty}(\ul{\lambda}_{J})$ is of dimension $g+4|S|+3[F^+:\Q]-2|J|$. But this follows from the fact $\cW_{\infty}(\ul{\lambda}_J)$ (hence $W_i$) is equidimensional of dimension $g+4|S|+3[F^+:\Q]-2|J|$ and \cite[Lem. 6.2.5, 6.2.10]{Che}.
\end{proof}
%
\subsubsection{Density of classical points}
Recall a point $(y,\delta)$ of $X_p(\overline{\rho})$ is called \emph{classical} if
\begin{equation*}
  J_{B_p}\big(\Pi_{\infty}^{R_{\infty}-\an, \lalg}\big)[\fm_y,T_p=\delta]\neq 0,
\end{equation*}
where ``$\lalg$" denotes the locally algebraic vectors for the $G_p$-action; and we call a point $(y,\delta)$ of $X_p(\overline{\rho},\ul{\lambda}_J)$ classical if it is classical as a point in $X_p(\overline{\rho})$, which is also equivalent to
\begin{equation*}
  J_{B_p}\big(\Pi_{\infty}^{R_{\infty}-\an}(\ul{\lambda}_J)^{\lalg}\big)[\fm_y,T_p=\delta]\neq 0.
\end{equation*}
A  point $(y,\delta)$ of $X_p(\overline{\rho})$ is called \emph{spherical} if $\delta$ is locally algebraic (i.e. $\wt(\delta)\in \Z^{2|\Sigma_p|}$) and $\psi_{\delta}:=\delta \delta_{\wt(\delta)}^{-1}$ is unramified; $(y,\delta)$ is called \emph{very regular} if $\delta$ is locally algebraic and the character
\begin{equation}\label{equ: cclg-add}\delta_{\widetilde{v}}^{\natural}:=\delta_{\widetilde{v}} \big(\unr_{\widetilde{v}}(q_{\widetilde{v}})\otimes \prod_{\sigma\in \Sigma_{\widetilde{v}}} \sigma^{-1}\big)\end{equation}
 is very regular (cf. (\ref{equ: cclg-dal})) for all $v|p$. Let $\delta^{\natural}:=\prod_{v|p} \delta_{\widetilde{v}}^{\natural}$.

 For a locally algebraic character $\delta$ of $T_p$, let $\Sigma^+(\delta)=\cup_{v|p} \Sigma^+(\delta_{\widetilde{v}})$ (where $\Sigma^+(\delta_{\widetilde{v}}) \subseteq \Sigma_{\widetilde{v}}$ is defined as in  (\ref{equ: cclg-Domi}) for $L=F_{\widetilde{v}}$)
For $J'\subseteq \Sigma^+(\delta)$, put
\begin{equation}\label{equ: cclg-cJc}
  \delta_{J'}^c:=\otimes_{v|p} \delta_{\widetilde{v},J_{\widetilde{v}}'}^c:=\otimes_{v|p} \big(\delta_{\widetilde{v}} (\prod_{\sigma\in J_{\widetilde{v}}'} \sigma^{\wt(\delta)_{2,\sigma}-\wt(\delta)_{1,\sigma}-1} \otimes \prod_{\sigma\in J'_{\widetilde{v}}} \sigma^{\wt(\delta)_{1,\sigma}-\wt(\delta)_{2,\sigma}+1})\big).
\end{equation}
Thus $\wt(\delta_{J'}^c)_{i,\sigma}=\wt(\delta)_{i,\sigma}$ if $\sigma\notin J'$; $\wt(\delta_{J'}^c)_{1,\sigma}=\wt(\delta)_{2,\sigma}-1$, $\wt(\delta_{J'}^c)_{2,\sigma}=\wt(\delta)_{1,\sigma}+1$ if $\sigma\in J'$. And we have  $\Sigma^+(\delta_{J'}^c)=\Sigma^+(\delta)\setminus J'$.
In fact, let $s_{J'}:=\prod_{\sigma\in J'}s_{\sigma}$ where $s_{\sigma}$ denotes the (unique) simple reflection in the Weyl group $\cS_2$ of $\ug_{\widetilde{v}}\otimes_{F_{\widetilde{v}},\sigma} E$ \big(note the Weyl group of $\ug_p\otimes_{\Q_p} E$ is isomorphic to the product of the Weyl groups $\cS_2$ of all $\ug_{\widetilde{v}}\otimes_{F_{\widetilde{v}},\sigma} E$\big), then $\delta_{J'}^c=\psi_{\delta} \delta_{s_{J'}\cdot \wt(\delta)}$ (where ``$\cdot$" denotes the dot  action).

For a locally algebraic character $\delta$ of $T_p$ over $E$, put
\begin{equation}\label{equ: cclg-idel}
  I(\delta):=\big(\Ind_{\overline{B}_p}^{G_p} \delta \delta_{\wt(\delta)_{\Sigma^+(\delta)}}^{-1}\big)^{\Sigma_p\setminus \Sigma^+(\delta)-\an}\otimes_E L(\wt(\delta)_{\Sigma^+(\delta)}).
\end{equation}
By \cite[Thm. 4.1]{Br} (see also Proposition \ref{prop: cclg-cnee} below), the representations $\{I(\delta_{J'}^c)\}_{J'\subseteq \Sigma^+(\delta)}$ give the Jordan-H\"older constituents of  $\big(\Ind_{\overline{B}_p}^{G_p} \delta\big)^{\Q_p-\an}$. Note also that $I(\delta)$ is locally algebraic if $\delta$ is dominant. For $v|p$, put \begin{equation*}I_{\widetilde{v}}(\delta_{\widetilde{v}}):=\big(\Ind_{\overline{B}(F_{\widetilde{v}})}^{\GL_2(F_{\widetilde{v}})} \big(\delta \delta_{\wt(\delta)_{\Sigma^+(\delta)}}^{-1}\big)_{\widetilde{v}}\big)^{\Sigma_{\widetilde{v}}\setminus \Sigma^+(\delta)_{\widetilde{v}}-\an}\otimes_E L(\wt(\delta)_{\Sigma^+(\delta)_{\widetilde{v}}}),\end{equation*}
where $\Sigma^+(\delta)_{\widetilde{v}}:=\Sigma^+(\delta)\cap \Sigma_{\widetilde{v}}$, which is a locally $\Q_p$-analytic representation of $\GL_2(F_{\widetilde{v}})$. We have
\begin{equation}\label{equ: cclg-gpa}
  I(\delta)\cong\widehat{\otimes}_{v|p} I_{\widetilde{v}}(\delta_{\widetilde{v}}).
\end{equation}
Note also that $I_{\widetilde{v}}(\delta_{\widetilde{v}})$ is irreducible if $\Sigma^+(\delta)_{\widetilde{v}}\neq \Sigma_{\widetilde{v}}$. Denote by $\delta_{B_p}=\otimes_{v|p} \big(\unr_{\widetilde{v}}(q_{\widetilde{v}}^{-1}) \otimes \unr_{\widetilde{v}}(q_{\widetilde{v}})\big)$ the modulus character of $B_p$, which factors through $T_p$ and thus can also be viewed as a character of $\overline{B}_p$ via $\overline{B}_p\twoheadrightarrow T_p$.
\begin{lemma}\label{equ: cclg-dtX}
  Let $x=(y, \delta)$ be a point of $X_p(\overline{\rho},\ul{\lambda}_J)$ (hence $\wt(\delta)_{\sigma}=\ul{\lambda}_{\sigma}$ for $\sigma\in J$) with $\wt(\delta)$ integral and dominant. If any irreducible constituent of $I(\delta_{J'}^c\delta_{B_p}^{-1})$ does not have $G_p$-invariant lattice for all $\emptyset \neq J'\subseteq \Sigma_p\setminus J$, then $x$ is classical. We call such classical points \emph{$\Sigma_p\setminus J$-very classical}.
\end{lemma}
\begin{proof}
 By the adjunction formula Proposition \ref{prop: cclg-pve}, one has (see Appendix \ref{sec: cclg-A.2} for the notation)
 \begin{multline}\label{equ: cclg-pBg}
   \Hom_{T_p}\Big(\delta, J_{B_p}\big(\Pi_{\infty}^{R_{\infty}-\an}(\ul{\lambda}_J)\big)[\fm_y]\Big)\\ \xlongrightarrow{\sim} \Hom_{G_p}\Big(\cF_{\overline{B}_p}^{G_p}\big(\overline{M}_J(-\wt(\delta))^{\vee}, \psi_{\delta}\delta_{B_p}^{-1}\big), \Pi_{\infty}^{R_{\infty}-\an}(\ul{\lambda}_J)[\fm_y]\Big).
 \end{multline}
The Jordan-Holder factors of  $\cF_{\overline{B}_p}^{G_p}\big(\overline{M}_J(-\wt(\delta))^{\vee}, \psi_{\delta}\delta_{B_p}^{-1}\big)$ are given by $\{I(\delta_{J'}^c\delta_{B_p}^{-1})\}_{J'\subseteq \Sigma_p\setminus J}$ with multiplicity one. Since $\Pi_{\infty}$ is unitary, by assumption, any non-zero map in the left set of (\ref{equ: cclg-pBg}) factors through $I(\delta\delta_{B_p}^{-1})$, and hence $x$ is classical.
\end{proof}
\begin{proposition} \label{prop: cclg-dtf}Let $(y, \delta)$ be a  point of $X_p(\overline{\rho},\ul{\lambda}_J)$ with $\wt(\delta)$ dominant and integral, if
\begin{equation}\label{equ: cclg-qv1}
\val_{\widetilde{v}}\big(q_{\widetilde{v}}\delta_{\widetilde{v},1}(\varpi_{\widetilde{v}})\big)<\inf_{\sigma\in \Sigma_{\widetilde{v}}\setminus J_{\widetilde{v}}}\{\wt(\delta)_{1,\sigma}-\wt(\delta)_{2,\sigma}+1\}
\end{equation}
  for all $v|p$ with $J_{\widetilde{v}} \neq \Sigma_{\widetilde{v}}$, then the point $(y,\delta)$ is $\Sigma_p\setminus J$-very classical.
\end{proposition}
\begin{proof}
Let $J'\subset \Sigma_p\setminus J$, $J'\neq \emptyset$, and let $v|p$ be such that  $J'_{\widetilde{v}}\neq \emptyset$. We know that the $\GL_2(F_{\widetilde{v}})$-representation $I_{\widetilde{v}}((\delta_{J'}^c)_{\widetilde{v}}\delta_{B(F_{\widetilde{v}})}^{-1})$ is topologically irreducible, and that any irreducible constituent of $I(\delta_{J'}^c\delta_{B_p}^{-1})$ has the form $W \widehat{\otimes}_E I_{\widetilde{v}}((\delta_{J'}^c)_{\widetilde{v}}\delta_{B(F_{\widetilde{v}})}^{-1})$, where $W$ is a certain irreducible representation of $\prod_{v'|p, v'\neq v} \GL_2(F_{\widetilde{v}'})$. If $W \widehat{\otimes}_E I_{\widetilde{v}}((\delta_{J'}^c)_{\widetilde{v}}\delta_{B(F_{\widetilde{v}})}^{-1})$ admits a $G_p$-invariant lattice $\Lambda$, then for any non-zero $w\in W$, one can check that  $\Lambda_w:=\{v\in I_{\widetilde{v}}((\delta_{J'}^c)_{\widetilde{v}}\delta_{B(F_{\widetilde{v}})}^{-1})\ |\ v\otimes w\in \Lambda\}$ is a $\GL_2(F_{\widetilde{v}})$-invariant lattice in $I_{\widetilde{v}}((\delta_{J'}^c)_{\widetilde{v}}\delta_{B(F_{\widetilde{v}})}^{-1})$. By \cite[Prop. 5.1]{Br}, we can deduce \begin{equation*}
 \val_{\widetilde{v}}\big(q_{\widetilde{v}}\delta_{\widetilde{v},1}(\varpi_{\widetilde{v}})\big)\geq \sum_{\sigma\in J_{\widetilde{v}}'}(\wt(\delta)_{1,\sigma}-\wt(\delta)_{2,\sigma}+1),
\end{equation*}
which contradicts to the assumption. The Proposition follows then by Lemma \ref{equ: cclg-dtX}.
\end{proof}
\begin{theorem}\label{thm: cclg-try}
  (1) The set of classical points is Zariski-dense in $X_p(\overline{\rho},\ul{\lambda}_J)$ and accumulates at any point $x=(y,\delta)$ with $\delta$ locally algebraic  (cf. \cite[\S~3.3.1]{BCh}).

  (2)  In $X_p(\overline{\rho},\ul{\lambda}_J)$, the set $\cZ$ of  very regular and $\Sigma_p\setminus J$-very classical points accumulates at points $(y,\delta)$ with $\delta$ locally algebraic and very regular. Moreover, the subset $\cZ_0\subset \cZ$ of spherical points accumulates at spherical very regular points.
\end{theorem}
\begin{proof}
 (1)  For the first part, it is sufficient to prove the classical points are Zariski-dense in any irreducible component $X$ of $X_p(\overline{\rho},\ul{\lambda}_J)$. By \cite[Cor. 6.4.4]{Che} (see also Corollary \ref{cor: cclg-o0e} (1)) and the fact that the locally algebraic characters of $T_p^0$  are Zariski-dense in $\cT_p^0(\ul{\lambda}_J)$, there exists $x=(y,\delta)\in X$ with $\delta$ locally algebraic. Let  $U$  be an irreducible affinoid open neighborhood of $x$ in  $X$ such that
 $\omega_{\infty}: U \ra \omega_{\infty}(U)$ is finite and $\omega_{\infty}(U)=V_1\times V_2$ is an irreducible affinoid open in $\cW_{\infty}(\ul{\lambda}_J)=(\Spf S_{\infty})^{\rig}\times \cT_{p}^0(\ul{\lambda}_J)$ (shrinking $U$ if necessary). Let \begin{equation*}C_1:=\max_{(y',\delta')\in U,v|p, J_{\widetilde{v}}\neq \Sigma_{\widetilde{v}}}\val_{\widetilde{v}}\big(q_{\widetilde{v}}\delta'_{1,\widetilde{v}}(\varpi_{\widetilde{v}})\big).\end{equation*}Let $C_2\geq C_1$, and $Z$ be the set of points $\delta'$ in $V_2$ satisfying
 \begin{enumerate}\item[(i)] $\wt(\delta')$ is integral,
 \item[(ii)] $\inf_{\sigma\in J_{\widetilde{v}}}\{\wt(\delta')_{1,\sigma}-\wt(\delta')_{2,\sigma}+1\}\geq C_2$ for $v|p$ with $J_{\widetilde{v}}\neq \Sigma_{\widetilde{v}}$.
 \end{enumerate} Thus $Z$ is Zariski dense in $V_2$. By \cite[Lem. 6.2.8]{Che}, $\omega_{\infty}^{-1}(V_1\times Z)$ is Zariski-dense in $U$. However, by Proposition \ref{prop: cclg-dtf}, any point in $\omega_{\infty}^{-1}(V_1 \times Z)$ is $\Sigma_p\setminus J$-very classical. The first part of (1) follows. In fact, the above argument shows the classical points accumulate at $x$ (since the set $Z$ in fact accumulates at $\delta$), thus the second part of (1) also follows.

 (2)  Suppose  $x$ (with the above notation) is moreover very regular.
For $v|p$, if $J_{\widetilde{v}}=\Sigma_{\widetilde{v}}$, then for any $(y',\delta')$ in $U$, $\delta'|_{T_{\widetilde{v}}^0}=\delta|_{T_{\widetilde{v}}^0}=\delta_{\ul{\lambda}_{\Sigma_{\widetilde{v}}}}|_{T_{\widetilde{v}}^0}$ \big(since $V_1\times V_2$ is irreducible, and $\cT_{\widetilde{v}}^0(\ul{\lambda}_{\Sigma_{\widetilde{v}}})$ is discrete\big). We shrink $U$ such that $(\delta'_{\widetilde{v}})^{\natural}$ is very regular for all $v|p$ with $J_{\widetilde{v}}=\Sigma_{\widetilde{v}}$ \big(cf. (\ref{equ: cclg-add}) (\ref{equ: cclg-dal}), note that since the weights $\ul{\lambda}_{J_{\widetilde{v}}}=\ul{\lambda}_{\Sigma_{\widetilde{v}}}$ are fixed,  in this case very regular is equivalent to regular\big).
We now consider the places $v|p$ such that $J_{\widetilde{v}}\neq \Sigma_{\widetilde{v}}$. Note first for any point $(y',\delta')$ in $X_p(\overline{\rho},\ul{\lambda}_J)$, one has
 \begin{equation}\label{equ: cclg-etale}
   \val_{\widetilde{v}}\big(\delta'_{1,\widetilde{v}}(\varpi_{\widetilde{v}})\big)+\val_{\widetilde{v}}\big(\delta'_{2,\widetilde{v}}(\varpi_{\widetilde{v}})\big)=0
 \end{equation}
 for all $v|p$ (since $\Pi_{\infty}$ is unitary in particular for the action of the center). Thus for $v|p$, $J_{\widetilde{v}}\neq \Sigma_{\widetilde{v}}$, we can (and do) enlarge $C_2$ such that the property (ii) (together with (\ref{equ: cclg-etale})) implies that $(\delta_{\widetilde{v}}')^{\natural}$ is very regular.
Consequently, any point in $\omega_{\infty}^{-1}(V_1 \times Z)$ will be very regular. The first part of (2) follows.
If $x$ is moreover spherical, the subset of $Z$ of algebraic characters is also Zariski-dense in $V_2$. And any point in $\omega_{\infty}^{-1}(V_1\times Z)$ is spherical, very regular and $\Sigma_p\setminus J$-very classical. The second part of (2)  follows.
\end{proof}
\begin{remark}\label{rem: cclg-tey}
  If $J_{\widetilde{v}}\neq \Sigma_{\widetilde{v}}$ for all $v|p$, one can in fact prove (by the same arguments as above together with \cite[Cor. 6.4.4]{Che}) that the spherical, very regular, $\Sigma_p\setminus J$-very classical points are Zariski dense in $X_p(\overline{\rho},\ul{\lambda}_J)$. However, if $J_{\widetilde{v}}=\Sigma_{\widetilde{v}}$, let $x=(y, \delta)\in X_p(\overline{\rho},\ul{\lambda}_J)$, thus $\wt(\delta)_{\sigma}=(\lambda_{1,\sigma},\lambda_{2,\sigma})$ for $\sigma \in J$ (thus for $\sigma\in \Sigma_{\widetilde{v}}$); since $\cT_{\widetilde{v}}^0(\ul{\lambda}_{\Sigma_{\widetilde{v}}})$ consists of isolated points, if $x'=(r',\delta')$ lies in an irreducible component of $X_p(\overline{\rho},\ul{\lambda}_J)$ containing $x$, then $\delta'|_{T^0_{\widetilde{v}}}=\delta_{T^0_{\widetilde{v}}}$. In particular, if $\delta_{\widetilde{v}}\delta_{\ul{\lambda}_{\widetilde{v}}}^{-1}$ is not unramified, then any irreducible component containing $x$ does not have spherical classical points.
\end{remark}
\subsubsection{Reducedness of patched eigenvarieties}
\begin{theorem}\label{thm: cclg-red}
 The rigid space $X_p(\overline{\rho},\ul{\lambda}_J)$  is reduced at points $x=(y,\delta)$ with $\delta$ very regular and locally algebraic.
\end{theorem}

\begin{proof}The theorem follows from the same argument as in the proof of \cite[Cor. 3.20]{BHS1}.
Let $x=(y,\delta)\in X_p(\overline{\rho}, \ul{\lambda}_J)$ satsify $\delta$ locally algebraic and very regular. By Proposition \ref{prop: cclg-Fen} (see also Remark \ref{rem: cclg-neii}), there exists an affinoid open neighborhood $U=\Spm B$ of $x$ in $X_p(\overline{\rho},\ul{\lambda}_J)$ such that
  \begin{itemize}
    \item $\omega_{\infty}(U)=\Spm A$ is an irreducible affinoid open in $(\Spf S_{\infty})^{\rig} \times \cT_{p}^0(\ul{\lambda}_J)$,
    \item $\omega_{\infty}(U)$ has the form $V_1\times V_2\subseteq (\Spf S_{\infty})^{\rig} \times \cT_{p}^0(\ul{\lambda}_J)$ (with $V_i$ irreducible),
    \item $M:=\Gamma\big(U, \cM_{\infty}(\ul{\lambda}_J)\big)$ is a finite projective $A$-module equipped with an $A$-linear action of $R_{\infty} \times T_p^+$, and $B$ is isomorphic to the $A$-subalgebra of $\End_{A}(M)$ generated by $R_{\infty}$ and $T_p^+$.
  \end{itemize}
  Note if $J_{\widetilde{v}}=\Sigma_{\widetilde{v}}$, then image of $U$ in $\cT_{\widetilde{v}}^0$ is a single point. As in the proof of Theorem \ref{thm: cclg-try}, 
   we can find  a set $Z\subset V_2(\overline{E})$ such that (shrinking $U$ if necessary)
  \begin{enumerate}
    \item $Z$ is Zariski-dense in $V_2$,
    \item any point $x'\in \omega_{\infty}^{-1}(V_1 \times Z)$ is very regular and $\Sigma_p\setminus J$-very classical.
  \end{enumerate}
As in the proof of \cite[Prop. 3.9]{Che05} (see also the proof of \cite[Cor. 3.20]{BHS1}), it is sufficient to prove that for any $\delta_0\in Z$, there exists an affinoid open $U_{\delta_0}$ of $V_1$ (which is thus Zariski-dense in $V_1$) such that for any $z\in U_{\delta_0}\times \{\delta_0\}$,  the action of $B$ (thus of $R_{\infty}$ and $T_p^+$) on $M\otimes_A k(z)$ is semi-simple, where $k(z)$ denotes the residue field at $z$.

  Let $z\in V_1\times \{\delta_0\}$, $\fm_z$ be  the associated maximal ideal of $S_{\infty}[\frac{1}{p}]$, and  $\Sigma:=\Hom_{k(z)}\big(M\otimes_A k(z), k(z)\big)$ (which is of finite dimension over $k(z)$). Since $M\otimes_A k(z)$ is the set of the  sections of $\cM_{\infty}(\ul{\lambda}_J)$ over the finite set $U\cap \omega_{\infty}^{-1}(z)$,  $\Sigma$ is isomorphic to a direct factor of $J_{B_p}\big((\Pi_{\infty}^{R_{\infty}-\an}(\ul{\lambda}_J)\big)[T_p^0=\delta_0]$.  Let $\mu$ be the weight of $\delta_0$, and $\chi$ be a smooth character of $T_p$ such that $\chi|_{T_p^0}=\delta_0\delta_{\mu}^{-1}$. Since $T_p^0$ acts on $\Sigma$ via $\delta_0$, there exists a smooth unramified representation $\Sigma_{\infty}$ of $T_p$ such that $\Sigma\cong \delta_{\mu}\chi \otimes_{k(z)} \Sigma_{\infty}$. By Proposition \ref{prop: cclg-pve}, one has (see Appendix \ref{sec: cclg-A.2} for the notation)
   \begin{multline*}
    \Hom_{T_p}\Big(\delta_{\mu}\chi \otimes_E \Sigma_{\infty}, J_{B_p}\big(\Pi_{\infty}^{R_{\infty}-\an}(\ul{\lambda}_J)[\fm_z]\big) \Big)\\
     \xlongrightarrow{\sim} \Hom_{G_p}\Big(\cF_{\overline{B}_p}^{G_p}\big(\overline{M}_J(-\mu)^{\vee}, \chi \otimes_{k(z)}\Sigma_{\infty}\otimes_{k(z)}\delta_{B_p}^{-1}\big), \Pi_{\infty}^{R_{\infty}-\an}(\ul{\lambda}_J)[\fm_z]\Big).
\end{multline*}
  Since any point in $\omega_{\infty}^{-1}(z)$ is $\Sigma_p\setminus J$-very classical, one can show as in the proof of Lemma \ref{equ: cclg-dtX} that any morphism on the right side factors through the locally algebraic quotient
  \begin{equation*}
    \cF_{\overline{B}_p}^{G_p}\big(\overline{L}(-\mu), \chi\otimes_{k(z)}\Sigma_{\infty}\otimes_{k(z)}\delta_{B_p}^{-1}\big)\cong \big(\Ind_{\overline{B}_p}^{G_p} \chi\otimes_{k(z)}\Sigma_{\infty}\otimes_{k(z)}\delta_{B_p}^{-1} \big)^{\infty}\otimes_E L(\mu).
  \end{equation*}
Since any point in $\omega_{\infty}^{-1}(z)$ is very regular, any irreducible subquotient of
$$W:=\big(\Ind_{\overline{B}_p}^{G_p} \chi\otimes_{k(z)}\Sigma_{\infty}\otimes_{k(z)}\delta_{B_p}^{-1} \big)^{\infty}$$
has the same \emph{smooth type} (cf. \cite[\S~3]{CEGGPS}), denoted by $\sigma_{\sm}$ (which is a finite dimensional smooth representation of $K_p$). Thus $W$ lies in a Berstein component, and there exists a finite dimensional $\cH:=k(z)[K_p\backslash G_p/K_p]$-module $\cM^{\lc}$ such that $W\cong \text{c-}\Ind_{K_p}^{G_p} (\sigma_{\sm}) \otimes_{\cH} \cM^{\lc}$ (e.g. see \cite[\S~3.3]{CEGGPS}). Let $\tau$ be the \emph{inertial type} corresponding to $\sigma_{\sm}$ (cf. \emph{loc. cit.}), $\sigma_{\alg}:=L(\mu)|_{K_p}$ and $\sigma:=\sigma_{\sm}\otimes_{k(z)} \sigma_{\alg}$.
  The theorem then follows from verbatim of the last paragraph of the proof of \cite[Cor. 3.20]{BHS1}, replacing the universal crystalline deformation ring $R_{\overline{\rho}_p}^{\square, \textbf{k}-\crr}$ in \emph{loc. cit. } by the universal potentially crystalline deformation ring of type $\sigma$:
$R_{\overline{\rho}_p}^{\square}(\sigma):=\widehat{\otimes}_{v|p} R_{\overline{\rho}_{\widetilde{v}}}^{\square}(\sigma |_{K_{\widetilde{v}}})$ (see \cite[\S~4]{CEGGPS}, in particular \cite[Lem. 4.17, 4.18]{CEGGPS}).
\end{proof}
\subsection{Infinitesimal ``R=T" results}\label{sec: cclg-3.4}Let $\fX_{\overline{\rho}^p}^{\square}:=\Spf\big(\widehat{\otimes}_{v\in S \setminus \Sigma_p} R_{\overline{\rho}_{\widetilde{v}}}^{\bar{\square}}\big)^{\rig}$, $\fX_{\overline{\rho}_p}^{\square}:=\Spf\big(\widehat{\otimes}_{v|p} R_{\overline{\rho}_{\widetilde{v}}}^{\bar{\square}}\big)^{\rig}$, and $\bU$ be the open unit ball in $\bA^1$. We have thus $(\Spf R_{\infty})^{\rig}\cong \fX_{\overline{\rho}^p}^{\square} \times\bU^g \times \fX_{\overline{\rho}_p}^{\square}$. For $v|p$, denote by $\iota_{\widetilde{v}}$ the following isomorphism (cf. (\ref{equ: cclg-add}))
\begin{equation*}
  \iota_{\widetilde{v}}: \cT_{\widetilde{v}}\xlongrightarrow{\sim} \cT_{\widetilde{v}}, \ \delta_{\widetilde{v}}\mapsto \delta_{\widetilde{v}}^{\natural}.
\end{equation*}
Put $\iota_p:=\prod_{v|p} \iota_{\widetilde{v}}: \cT_p\xrightarrow{\sim} \cT_p$, $\iota_{\widetilde{v}}^{-1}\big(X_{\tri}^{\square}(\overline{\rho}_{\widetilde{v}})\big):=X_{\tri}^{\square}(\overline{\rho}_{\widetilde{v}}) \times_{\cT_{\widetilde{v}}, \iota_{\widetilde{v}}} \cT_{\widetilde{v}}$, $X_{\tri}^{\square}(\overline{\rho}_p):=\prod_{v|p} X_{\tri}^\square(\overline{\rho}_{\widetilde{v}})$, and
\begin{equation*}
  \iota_p^{-1}\big(X_{\tri}^{\square}(\overline{\rho}_p)\big):=X_{\tri}^{\square}(\overline{\rho}_p) \times_{\cT_p, \iota_p} \cT_p\cong \prod_{v|p} \iota_{\widetilde{v}}^{-1}\big(X_{\tri}^{\square}(\overline{\rho}_{\widetilde{v}})\big).
\end{equation*}
Note that $\iota_p^{-1}\big(X_{\tri}^{\square}(\overline{\rho}_p)\big)$ is a closed subspace of $\fX_{\overline{\rho}_p}^{\square}\times \cT_p$. Recall
\begin{theorem}[$\text{\cite[Thm. 3.21]{BHS1}}$]\label{thm: cclg-pwR}
  The natural embedding $X_p(\overline{\rho}) \hookrightarrow (\Spf R_{\infty})^{\rig}\times \cT_p$ factors though
  \begin{equation}\label{equ: cclg-pow}X_p(\overline{\rho})\hooklongrightarrow \fX_{\overline{\rho}^p}^{\square}\times \bU^g \times \iota_p^{-1}\big(X_{\tri}^{\square}(\overline{\rho}_p)\big),\end{equation} and induces an isomorphism between $X_p(\overline{\rho})$ and a union of irreducible components (equipped with the reduced closed rigid subspace structure) of $\fX_{\overline{\rho}^p}^{\square}\times \bU^g \times \iota_p^{-1}\big(X_{\tri}^{\square}(\overline{\rho}_p)\big)$.
\end{theorem}
For  $x=(y,\delta) \in X_p(\overline{\rho})$ and $v\in S$, denote by $\rho_{x,\widetilde{v}}$ the  image of $x$ in $(\Spf R_{\overline{\rho}_{\widetilde{v}}}^{\bar{\square}})^{\rig}$ via the natural morphism. For $v|p$,  $x_{\widetilde{v}}=(\rho_{x,\widetilde{v}}, \delta_{\widetilde{v}}^{\natural})$ is thus the image of $x$ in $X_{\tri}^{\square}(\overline{\rho}_{\widetilde{v}})$ via (\ref{equ: cclg-pow}).
For a finite place $\widetilde{v}$ of $F$ with $v\nmid p$, a $2$-dimensional $p$-adic representation $\rho_{\widetilde{v}}$ of $\Gal_{F_{\widetilde{v}}}$ is called \emph{generic} if the smooth representation $\pi$ of $\GL_2(F_{\widetilde{v}})$ associated to $\WD(\rho_{\widetilde{v}})$ (the associated Weil-Deligne representation) via local Langlands correspondance is generic (which is equivalent to infinite dimensional in this case).

For a weight $\ul{\lambda}_{\Sigma_p}=(\lambda_{1,\sigma},\lambda_{2,\sigma})_{\sigma\in \Sigma_p}$ of $\ft_p\otimes_{\Q_p} E$, $J'\subseteq \Sigma_p$, denote by $\ul{\lambda}_{J'}^{\natural}:=(\lambda_{1,\sigma}, \lambda_{2,\sigma}-1)_{\sigma\in J'}$. Now let $J\subseteq \Sigma_p$, $\ul{\lambda}_J\in \Z^{2|J|}$ be dominant, put $X_{\tri}^{\square}(\overline{\rho}_p,\ul{\lambda}_J^{\natural}):=\prod_{v|p} X_{\tri}^{\square}(\overline{\rho}_{\widetilde{v}}, \ul{\lambda}_{J_{\widetilde{v}}}^{\natural})$ (cf. \S~\ref{sec: cclg-2.1}). By definitions, the closed embedding (\ref{equ: cclg-pow}) induces a closed embedding
\begin{equation}\label{equ: cclg-rtht}
  X_p(\overline{\rho},\ul{\lambda}_J)' \hooklongrightarrow \fX_{\overline{\rho}^p}^{\square}\times \bU^g \times \iota_p^{-1}\big(X_{\tri}^{\square}(\overline{\rho}_p,\ul{\lambda}_J^{\natural})\big).
\end{equation}
Put $X^{\square}_{\tri,J-\dR}(\overline{\rho}_p, \ul{\lambda}_J^{\natural}):=\prod_{v|p} X_{\tri,J_{\widetilde{v}}-\dR}^{\square}(\overline{\rho}_{\widetilde{v}}, \ul{\lambda}_{J_{\widetilde{v}}}^{\natural})$ (cf. (\ref{equ: cclg-Jir})). By Shah's results \cite[Thm. 2]{Sha} and the density of classical points in $X_p(\overline{\rho},\ul{\lambda}_J)$ (cf. Theorem \ref{thm: cclg-try}), the injection (\ref{equ: cclg-pow}) induces a closed embedding
\begin{equation}\label{equ: cclg-adb}
  X_p(\overline{\rho},\ul{\lambda}_J)_{\red} \hooklongrightarrow  \fX_{\overline{\rho}^p} \times \bU^g \times \iota_p^{-1}\big(X^{\square}_{\tri,J-\dR}(\overline{\rho}_p, \ul{\lambda}_J^{\natural})\big).
\end{equation}

Now suppose $\delta$ is very regular and locally algebraic, we have that $X_p(\overline{\rho},\ul{\lambda}_J)$ is reduced at $x$ by Theorem \ref{thm: cclg-red}. Suppose $\rho_{x,\widetilde{v}}$ is generic for $v\in S\setminus S_p$. By \cite[Lem. 1.3.2 (1)]{BLGGT}, $\fX^{\square}_{\overline{\rho}^p}$ is smooth at $(\rho_{x,\widetilde{v}})_{v\in S\setminus S_p}$ (of dimension $4|S\setminus S_p|$). Let $X$ be the union of irreducible components of $X_p(\overline{\rho})$ containing $x$ (equipped with the reduced closed rigid subspace structure). By Theorem \ref{thm: cclg-pwR}, $X$ has the form
 \begin{equation*}\cup_{i} \big(X^p \times \bU^g \times \iota_p^{-1}(X_{i,p})\big)=\cup_i\Big(X^p\times \bU^g\times \big(\prod_{v|p} \iota_{\widetilde{v}}^{-1}(X_{i,\widetilde{v}}\big))\Big)\end{equation*} where $X^p$ is the (unique) irreducible component of $\fX_{\overline{\rho}^p}^{\square}$ containing $(\rho_{x,\widetilde{v}})_{v\in S\setminus S_p}$, and $X_{i,p}$ \big(resp. $X_{i,\widetilde{v}}$\big) is a  certain irreducible component of $X_{\tri}^{\square}(\overline{\rho}_p)$ \big(resp. $X_{\tri}^{\square}(\overline{\rho}_{\widetilde{v}})$\big) containing $(x_{\widetilde{v}})_{v\in S_p}$ (resp. $x_{\widetilde{v}}$)\footnote{By \cite{BHS3}, we know now $X_{\tri}^{\square}(\overline{\rho}_{\widetilde{v}})$ (resp. $X_{\tri}^{\square}(\overline{\rho}_p)$) is irreducible at $x_{\widetilde{v}}$ (resp.  at $(x_{\widetilde{v}})_{v\in S_p}$), and hence $X_{i,\widetilde{v}}$ (resp. $X_{i,p}$) is actually unique.}. Suppose $x$ is spherical, by Theorem \ref{thm: cclg-try} (and the  proof), $\cup_i X_{i,{\widetilde{v}}}$ satisfies the accumulation property at $x_{\widetilde{v}}$ for all $v|p$. Suppose $\mathrm{wt}(\delta)_{1,\sigma}\neq \mathrm{wt}(\delta)_{2,\sigma}-1$ for all $\sigma\in \Sigma_p$ and that $\rho_{x,\widetilde{v}}$ is $\Sigma(x_{\widetilde{v}})$-de Rham for all $v|p$ (cf. \S~\ref{sec: cclg-xdt}). Thus by Corollary \ref{coro: cclg-smte}, $\cup_i X_{i,{\widetilde{v}}}$ is smooth at $x_{\widetilde{v}}$, and hence for all $i$, $X_{i,\widetilde{v}}$ is equal to each other, which we denote by $X_{\widetilde{v}}$. In summary, $X=X^p \times \bU^g \times \iota_p^{-1} (X_p)$ with $X_p=\prod_{v|p} X_{\widetilde{v}}$, and is smooth at the point $x$.
\begin{theorem}\label{thm: cclg-pox}
  Let $x=(y,\delta)\in X_p(\overline{\rho}, \ul{\lambda}_J)$ be such that $\rho_{x,\widetilde{v}}$ is generic for $v\in S \setminus S_p$, $\rho_{\widetilde{v}}$ is $\Sigma(x_{\widetilde{v}})$-de Rham for all $v|p$, $x$ is spherical, \footnote{\label{foot: sphe}This assumption is to used to ensure that $X_{\widetilde{v}}$ satisfies the accumulation property at $x_{\widetilde{v}}$ for $v|p$ so that we can apply the results in \S~\ref{sec: cclg-xdt}. Now using the theory of \cite{BHS3}, one can probably weaken this condition. The situation is the same for Corollary \ref{cor: cclg-xts}, Corollary \ref{cor: cclg-xts} below. } and $\delta$ is very regular with $\mathrm{wt}(\delta)_{1,\sigma}\neq \mathrm{wt}(\delta)_{2,\sigma}-1$ for all $\sigma\in \Sigma_p$.  Then $X_p(\overline{\rho}, \ul{\lambda}_J)$ is smooth at $x$, and we have a natural isomorphism of complete regular noetherian local $k(x)$-algebras:
   \begin{equation}\label{equ: cclg-oaj}
    \widehat{\co}_{X_p(\overline{\rho},\ul{\lambda}_J),x}\xlongrightarrow{\sim} \widehat{\co}_{X^p\times \bU^g \times \iota_p^{-1}((X_p)_{J-\dR}(\ul{\lambda}_J^{\natural})), x},
  \end{equation}
  where $(X_p)_{J-\dR}(\ul{\lambda}_J^{\natural}):=\prod_{v|p} (X_{\widetilde{v}})_{J_{\widetilde{v}}-\dR}(\ul{\lambda}_{J_{\widetilde{v}}}^{\natural})$ (cf. \S~\ref{sec: cclg-2.1}, recall $ (X_{\widetilde{v}})_{J_{\widetilde{v}}-\dR}(\ul{\lambda}_{J_{\widetilde{v}}}^{\natural})=X_{\widetilde{v}} \times_{X_{\tri}^{\square}(\overline{\rho}_{\widetilde{v}})} X_{\tri,J_{\widetilde{v}}-\dR}^{\square}(\overline{\rho}_{\widetilde{v}},\ul{\lambda}_{J_{\widetilde{v}}}^{\natural})$).
\end{theorem}
\begin{proof}
  Let $Z$ be the union of irreducible components of $X_p(\overline{\rho},\ul{\lambda}_J)$ containing $x$ (equipped with the reduced closed rigid subspace structure), which is thus equidimensional of dimension $g+4|S|+3[F^+:\Q]-2|J|$. Moreover, since $X_p(\overline{\rho},\ul{\lambda}_J)$ is reduced at $x$, one has $\widehat{\co}_{X_p(\overline{\rho},\ul{\lambda}_J),x} \cong \widehat{\co}_{Z,x}$. The closed embedding (\ref{equ: cclg-adb}) induces a closed embedding $Z\hookrightarrow X^p\times \bU^g \times \iota_p^{-1}((X_p)_{J-\dR}(\ul{\lambda}_J^{\natural}))$. We get thus a surjective morphism $\widehat{\co}_{X^p\times \bU^g \times \iota_p^{-1}((X_p)_{J-\dR}(\ul{\lambda}_J^{\natural})), x}\twoheadrightarrow\widehat{\co}_{Z,x}$. However, since $X^p$ is smooth at $(\rho_{x,\widetilde{v}})_{v\in S\setminus S_p}$ of dimension $4|S\setminus S_p|$, by Theorem \ref{thm: cclg-dxX} (1), one calculates:
  \begin{equation*}
    \dim_{k(x)} T_{X^p\times \bU^g \times \iota_p^{-1}((X_p)_{J-\dR}(\ul{\lambda}_J^{\natural})), x}=g+4|S|+3[F^+:\Q]-2|J|=\dim \widehat{\co}_{Z,x}.
  \end{equation*}
  The theorem follows.
\end{proof}
For $J'\subseteq J$, recall $X_p(\overline{\rho},\ul{\lambda}_J,J')\cong  X_p(\overline{\rho},\ul{\lambda}_{J'})\times_{\cT(\ul{\lambda}_{J'})} \cT(\ul{\lambda}_J)$. We put (cf. (\ref{equ: cclg-sJj}))
\begin{equation*}
  (X_p)_{J'-\dR}(\ul{\lambda}_J^{\natural}):=\prod_{v|p} (X_p)_{J'_{\widetilde{v}}-\dR}(\ul{\lambda}_{J_{\widetilde{v}}}^{\natural})\cong (X_p)_{J'-\dR}(\ul{\lambda}_{J'}^{\natural}) \times_{\cT(\ul{\lambda}_{J'}^{\natural})} \cT(\ul{\lambda}_J^{\natural}).
\end{equation*}
The following corollary follows easily   from Theorem \ref{thm: cclg-pox} \big(applied to $X_p(\overline{\rho},\ul{\lambda}_{J'})$\big):
 \begin{corollary}\label{equ: cclg-oms}
  Keep the situation of Theorem \ref{thm: cclg-pox}, and let $J'\subseteq J$. The isomorphism (\ref{equ: cclg-oaj}) (with $J$ replaced by $J'$) induces an isomorphism
  \begin{equation*}
    \widehat{\co}_{X_p(\overline{\rho},\ul{\lambda}_J,J'),x} \xlongrightarrow{\sim} \widehat{\co}_{X^p \times \bU^g \times \iota_p^{-1}((X_p)_{J'-\dR}(\ul{\lambda}_J^{\natural})),x}.
  \end{equation*}
 \end{corollary}
 Let $\Sigma(x):=\cup_{v|p} \Sigma(x_{\widetilde{v}})$ (cf. \S~\ref{sec: cclg-xdt}). The following corollary will play a crucial role in our proof of the existence of companion points:
\begin{corollary}\label{cor: cclg-xts}
  Keep the situation of Theorem \ref{thm: cclg-pox}, and let $J'\subseteq J$. The following statements are equivalent:

  (i) the natural projection 
  \big(induced by the closed embedding $X_p(\overline{\rho},\ul{\lambda}_J) \hookrightarrow X_p(\overline{\rho},\ul{\lambda}_J,J')$\big)
  \begin{equation*}
    \widehat{\co}_{X_p(\overline{\rho},\ul{\lambda}_J,J'),x} \twoheadlongrightarrow \widehat{\co}_{X_p(\overline{\rho},\ul{\lambda}_J),x}
  \end{equation*}
  is an isomorphism;

  (ii) $X_p(\overline{\rho},\ul{\lambda}_J,J')$ is smooth at $x$;

  (iii) $(J\setminus J')\cap \Sigma(x)=\emptyset$.
\end{corollary}
\begin{proof}
  The equivalence of (i) and (ii) is clear since $X_p(\overline{\rho},\ul{\lambda}_J,J')$ has the same dimension \big($g+4|S|+3[F^+:\Q]-2|J|$\big) as $X_p(\overline{\rho},\ul{\lambda}_J)$, and $X_p(\overline{\rho},\ul{\lambda}_J)$ is smooth at $x$. As in the proof of Theorem \ref{thm: cclg-pox}, to prove (ii) is equivalent to (iii), it is sufficient to show that $(J\setminus J')\cap \Sigma(x)=\emptyset$ if and only if
  \begin{equation*}
    \dim_{k(x)} T_{X_p(\overline{\rho},\ul{\lambda}_J,J'),x}=g+4|S|+3[F^+:\Q]-2|J|,
  \end{equation*}
  thus by Corollary \ref{equ: cclg-oms}, if and only if
  \begin{equation*}
    \dim_{k(x)} T_{\fX_{\overline{\rho}^p}^{\square}\times \bU^g \times \iota_p^{-1}((X_p)_{J'-\dR}(\ul{\lambda}_J^{\natural})),x}=g+4|S|+3[F^+:\Q]-2|J|.
  \end{equation*}
 However,  this follows from Theorem \ref{thm: cclg-dxX} (2).
\end{proof}
\section{Companion points and local-global compatibility}
\subsection{Breuil's locally analytic socle conjecture}\label{sec: cclg-4.1}We recall Breuil's locally analytic socle conjecture (for the group $G$ of \S~\ref{sec: cclg-3.1}). Let $\rho: \Gal_F \ra \GL_2(E)$ be a continuous representation such that $\rho\otimes \varepsilon \cong \rho^{\vee}\circ c$, $\rho$ is unramified outside $S$, and $\overline{\rho}$ is absolutely irreducible. Suppose moreover
\begin{enumerate}\item $\widehat{S}(U^p,E)_{\overline{\rho}}^{\lalg}[\fm_{\rho}]\neq 0$;
\item $\rho_{\widetilde{v}}:=\rho|_{\Gal_{F_{\widetilde{v}}}}$ is regular crystalline of distinct Hodge-Tate weights for all $v|p$.
\end{enumerate} Note by the results in \cite{BHS2}, assuming (2) and Hypothesis \ref{hypo: cclg-TW}, the condition (1) can be replaced by $J_{B_p}\big(\widehat{S}(U^p,E)^{\an}[\fm_{\rho}]\big)\neq 0$. Let $\ul{\lambda}_{\Sigma_p}:=(\lambda_{1,\sigma},\lambda_{2,\sigma})_{\sigma \in \Sigma_p}\in \Z^{2|\Sigma_p|}$ be such that $\HT(\rho_{\widetilde{v}})=-\ul{\lambda}_{\Sigma_{\widetilde{v}}}^{\natural}=-(\lambda_{1,\sigma},\lambda_{2,\sigma}-1)_{\sigma\in \Sigma_{\widetilde{v}}}$ for $v|p$; let $\alpha_{\widetilde{v},1}$, $\alpha_{\widetilde{v},2}$ be the two eigenvalues of the crystalline Frobenius $\varphi^{[F_{\widetilde{v},0}:\Q_p]}$ on $D_{\cris}(\rho_{\widetilde{v}})$ for $v|p$ (note $\alpha_{\widetilde{v},1}\alpha_{\widetilde{v},2}^{-1}\neq 1, p^{\pm [F_{\widetilde{v},0}:\Q_p]}$ since $\rho_{\widetilde{v}}$ is regular). For $s=(s_{\widetilde{v}})_{v|p}\in \cS_2^{|S_p|}$ (which is in fact the Weyl group of $G_p$), put $\psi_s:=\otimes_{v|p} \unr_{\widetilde{v}}(q_{\widetilde{v}}^{-1}\alpha_{\widetilde{v},s_{\widetilde{v}}^{-1}(1)})\otimes \unr_{\widetilde{v}}(\alpha_{\widetilde{v},s_{\widetilde{v}}^{-1}(2)})$, and $\delta_s:=\psi_s \delta_{\ul{\lambda}_{\Sigma_p}}$,
which is a locally algebraic character of $T_p$. Since $\ul{\lambda}_{\Sigma_p}$ is dominant, $I(\delta_s\delta_{B_p}^{-1})$ is locally algebraic. Moreover, for $s,s'\in \cS_2^{|S_p|}$, we have $I(\delta_s\delta_{B_p}^{-1})\cong I(\delta_{s'}\delta_{B_p}^{-1})=:I_0(\rho_p)$. By the classical local Langlands correspondence, there exists an injection of locally analytic representations of $G_p$:
\begin{equation*}
  I_0(\rho_p) \hooklongrightarrow \widehat{S}(U^p,E)_{\overline{\rho}}^{\an}[\fm_{\rho}].
\end{equation*}
Such an  injection gives (by applying Jacquet-Emerton functor) $|\cS_2^{|S_p|}|$-classical points $\{z_s=(\fm_{\rho},\delta_s)\}_{s\in \cS_2^{|S_p|}}$ in $\cE(U^p)_{\overline{\rho}}$. For $v|p$, let $\Sigma(z_s)_{\widetilde{v}}\subseteq \Sigma_{\widetilde{v}}$ be such that $((\delta_{s,\widetilde{v}})^c_{\Sigma(z_s)_{\widetilde{v}}})^{\natural}$ (cf. (\ref{equ: cclg-add}), (\ref{equ: cclg-cJc})) is a trianguline parameter of $\rho_{\widetilde{v}}$, and put $\Sigma(z_s):=\cup_{v|p} \Sigma(z_s)_{\widetilde{v}}$.
\begin{conjecture}[Breuil]\label{conj: cclg-las2}Keep the situation, and let $\chi$ be a continuous character of $T_p$ over $E$. Then $I(\chi)\hookrightarrow \widehat{S}(U^p,E)^{\an}_{\overline{\rho}}[\fm_{\rho}]$ if and only if there exist $s\in  \cS_2^{|S_p|}$ and $J\subseteq \Sigma(z_s)_{\widetilde{v}}$ such that $\chi=(\delta_s)_J^c\delta_{B_p}^{-1}$.
\end{conjecture}
This conjecture is in fact equivalent to the following conjecture on companion points on the eigenvariety:
\begin{conjecture}\label{conj: cclg-comp}
  (1) Let $\chi: T_p\ra E^{\times}$, $(\fm_{\rho}, \chi)\in \cE(U^p)_{\overline{\rho}}$ if and only if there exist $s\in  \cS_2^{|S_p|}$ and $J\subseteq \Sigma(z_s)_{\widetilde{v}}$ such that $\chi=(\delta_s)_J^c$.

  (2) For $s\in \cS_2^{|S_p|}$ and $J\subseteq \Sigma(z_s)$, the point $(z_s)_J^c:=(\fm_{\rho},(\delta_s)_J^c)$ lies moreover in $\cE(U^p,\ul{\lambda}_{\Sigma_p\setminus J})_{\overline{\rho}}$.
\end{conjecture}
First note the ``only if" part in (1) of Conjecture \ref{conj: cclg-comp} is an easy consequence of the global triangulation theory: if $(\fm_{\rho}, \chi)\in \cE(U^p)_{\overline{\rho}}$, by global triangulation theory \cite{KPX} \cite{Liu} applied to the point $(\fm_{\rho}, \chi)$, there exists $S\subseteq \Sigma^+(\chi)$  such that $((\chi_S^c)^{\natural})_{\widetilde{v}}$ is a trianguline parameter of $\rho_{\widetilde{v}}$ for all $v|p$. So there exists $s\in \cS_2^{|S_p|}$ such that $\chi_S^c=(\delta_s)^c_{\Sigma(z_s)}$. Moreover, since $S\subseteq \Sigma_p\setminus \Sigma^+(\chi_S^c)$, we have $S\subseteq \Sigma(z_s)$ and hence $\chi=(\delta_s)^c_{\Sigma(z_s)\setminus S}$.

We show the equivalence of the above two conjectures. Assuming Conjecture \ref{conj: cclg-las2}, for $s\in \cS_2^{|S_p|}$ and $J\subseteq \Sigma(z_s)$. By applying Jacquet-Emerton functor to an injection \begin{equation*}I\big((\delta_s)_J^c\delta_{B_p}^{-1}\big)\hooklongrightarrow \widehat{S}(U^p,E)_{\overline{\rho}}^{\an}[\fm_\rho]\end{equation*}
\big(which automatically factors through $\widehat{S}(U^p,E)_{\overline{\rho}}^{\an}(\ul{\lambda}_{\Sigma_p\setminus J})[\fm_{\rho}]$\big), we get the point $$(z_s)_J^c\in \cE(U^p,\ul{\lambda}_{\Sigma_p\setminus J}).$$ Conversely, assuming Conjecture \ref{conj: cclg-comp}, if $I(\chi)\hookrightarrow \widehat{S}(U^p,E)_{\overline{\rho}}[\fm_{\rho}]$,  applying Jacquet-Emerton functor, we get a point $(\fm_\rho,\chi\delta_B)\in \cE(U^p)_{\overline{\rho}}$. Thus the ``only if" part of Conjecture \ref{conj: cclg-las2} follows from the ``only if" part of Conjecture \ref{conj: cclg-comp}. The ``if" part of Conjecture \ref{conj: cclg-las2} follows directly from Conjecture \ref{conj: cclg-comp} (2) and the following bijections
\begin{multline}\label{equ: cclg-spE}\Hom_{G_p}\big(I((\delta_s)_J^c\delta_{B_p}^{-1}),\widehat{S}(U^p,E)_{\overline{\rho}}^{\an}[\fm_{\rho}]\big)\cong\Hom_{G_p}\big( I((\delta_s)_J^c\delta_{B_p}^{-1}), \widehat{S}(U^p,E)_{\overline{\rho}}^{\an}(\ul{\lambda}_{\Sigma_p\setminus J})[\fm_{\rho}]\big)\\
\cong\Hom_{T_p}\big((\delta_s)_J^c, J_{B_p}\big(\widehat{S}(U^p,E)^{\an}(\ul{\lambda}_{\Sigma_p\setminus J})\big)[\fm_{\rho}]\big)\end{multline}
for all $s\in \cS_2^{|S_p|}$ and $J\subseteq \Sigma_p$, where the first bijection is clear and the second follows from Proposition \ref{prop: cclg-pve}.

In the following, we assume Hypothesis \ref{hypo: cclg-TW}, and we prove Conjecture \ref{conj: cclg-las2} (and hence Conjecture \ref{conj: cclg-comp}) under the assumption in the next section. We will work with the patched eigenvariety and show a similar result of Conjecture \ref{conj: cclg-las2} with  $\widehat{S}(U^p,E)_{\overline{\rho}}$ replaced by $\Pi_{\infty}$. Indeed, since $\Pi_{\infty}[\fa]\cong \widehat{S}(U^p,E)_{\overline{\rho}}$, if we denote $\fm_{\rho}$ the maximal ideal of $R_{\infty}\otimes_{\co_E} E$ corresponding to $\rho$ (via $R_{\infty}/\fa \twoheadrightarrow \cR_{\cS,\overline{\rho}}$), then $\Pi_{\infty}[\fm_{\rho}]\cong \widehat{S}(U^p,E)[\fm_{\rho}]$. For $s\in \cS_2^{|S_p|}$, one has a point $x_s=(\fm_{\rho}, \delta_s)\in X_p(\overline{\rho})$, which is moreover classical and lies hence in $X_p(\overline{\rho},\ul{\lambda}_{\Sigma_p})$. By $\Pi_{\infty}[\fm_{\rho}]\cong \widehat{S}(U^p,E)[\fm_{\rho}]$ and the adjunction property in (\ref{equ: cclg-spE}), one has the following easy lemma
\begin{lemma}\label{lem: cclg-joa}
 The point $(x_s)_J^c\in X_p(\overline{\rho},\ul{\lambda}_{\Sigma_p\setminus J})$ if and only if $(z_s)_J^c\in \cE(U^p,\ul{\lambda}_{\Sigma_p\setminus J})$.
\end{lemma}


\subsection{Main results}Let $J\subseteq \Sigma_p$, and $\ul{\lambda}_J:=(\lambda_{1,\sigma}, \lambda_{2,\sigma})_{\sigma\in J}\in \Z^{2|J|}$ be dominant.
\begin{theorem}\label{thm: cclg-lrt}
  Let $x=(y,\delta)$ be a spherical\footnote{We put this assumption to apply Theorem \ref{thm: cclg-pox}, but as remarked in the footnote \ref{foot: sphe}, using the recent results of \cite{BHS3}, one can probably weaken this assumption.}, very regular point in $X_p(\overline{\rho},\ul{\lambda}_J)$ with $\Sigma^+(\delta)=J$, and $\rho_{x,\widetilde{v}}$ generic for all $v\in S\setminus S_p$. Suppose $\Sigma(x)\neq \emptyset$ (note $\Sigma(x)\subseteq \Sigma^+(\delta)=J$), then for all $\sigma\in \Sigma(x)$, $x_{\sigma}^c=(y,\delta_{\sigma}^c) \in X_p(\overline{\rho},\ul{\lambda}_{J\setminus \{\sigma\}})$ (note $\Sigma^+(\delta_{\sigma}^c)=J\setminus \{\sigma\}$).
\end{theorem}
\begin{remark}By Proposition \ref{prop: cclg-pve} and the assumption $J=\Sigma^+(\delta)$, 
one sees  $x=(y,\delta)\in X_p(\overline{\rho},\ul{\lambda}_{\Sigma^+(\delta)})$ is equivalent to the existence of an injection of locally $\Q_p$-analytic representations of $G_p$
\begin{equation*}
  I(\delta\delta_{B_p}^{-1}) \hooklongrightarrow \Pi_{\infty}^{R_{\infty}-\an}(\ul{\lambda}_{\Sigma^+(\delta)})[\fm_y].
\end{equation*}
Similarly, $x_{\sigma}^c \in X_p(\overline{\rho}, \ul{\lambda}_{\Sigma^+(\delta)\setminus \{\sigma\}})$ is equivalent to the existence of an injection $$I(\delta_{\sigma}^c\delta_{B_p}^{-1}) \hooklongrightarrow \Pi_{\infty}^{R_{\infty}-\an}(\ul{\lambda}_{\Sigma^+(\delta)\setminus \{\sigma\}})[\fm_y].$$
\end{remark}
\begin{proof}Suppose $\Sigma(x)\neq \emptyset$, and let $\sigma\in \Sigma(x)$, $J':=J\setminus \{\sigma\}$. Consider the following closed rigid subspaces of $X_p(\overline{\rho})$:
  \begin{equation*}
    X_p(\overline{\rho},\ul{\lambda}_{J}) \hooklongrightarrow X_p(\overline{\rho},\ul{\lambda}_J,J') \hooklongrightarrow X_p(\overline{\rho}, \ul{\lambda}_{J'}).
  \end{equation*}
  By Proposition \ref{prop: cclg-Fen}, there exists an open affinoid neighborhood $U_J:=\Spm B_J$ (resp. $U_{J'}:=\Spm B_{J'}$) of $x$ in $X_p(\overline{\rho},\ul{\lambda}_J)$ \big(resp. in $X_p(\overline{\rho},\ul{\lambda}_{J'})$\big), such that
  \begin{itemize}\item $\omega_{\infty}(U_J)$ (resp. $\omega_{\infty}(U_{J'})$) is an affinoid open, denoted by $\Spm A_J$ (resp. $\Spm A_{J'}$) in $\cW_{\infty}(\ul{\lambda}_J)$ \big(resp. $\cW_{\infty}(\ul{\lambda}_{J'})$\big);
  \item $\kappa_{\infty}^{-1}(\{\kappa_{\infty}(x)\})=\{x\}$, i.e. $x$ is the only point in $U_J$ (resp. $U_{J'}$) lying above $\omega:=\omega_{\infty}(x)$;
  \item $\overline{M}_J:=\Gamma(U_J, \cM_{\infty}(\ul{\lambda}_J))$ \big(resp. $M_{J'}:=\Gamma(U_{J'}, \cM_{\infty}(\ul{\lambda}_{J'}))$\big) is a locally free $A_J$ (resp. $A_{J'}$)-module.
  \end{itemize}One has thus (where $\fm_{\omega}$ denotes the maximal ideal of $A_J$ or $A_{J'}$ corresponding to $\omega$, $\fm_x$ the maximal ideal of $B_J$ or $B_{J'}$ corresponding to $x$, recall ``$[\cdot]$" denotes eigenspaces and ``$\{\cdot\}$" denotes generalized eigenspaces)
  \begin{equation*}
    (\overline{M}_J/\fm_{\omega}\overline{M}_J)^{\vee}\xlongrightarrow{\sim}
     J_{B_p}\big(\Pi_{\infty}^{R_{\infty}-\an}(\ul{\lambda}_J)\big)[\fm_{\omega}]\{\fm_x\}
  \end{equation*}
    \begin{equation*}
    \text{\Big(resp. }(M_{J'}/\fm_{\omega}M_{J'})^{\vee}
    \xlongrightarrow{\sim}
     J_{B_p}\big(\Pi_{\infty}^{R_{\infty}-\an}(\ul{\lambda}_{J'})\big)[\fm_{\omega}]\{\fm_x\}\text{\Big)}.
  \end{equation*}

\textbf{Claim:} The natural injection
\begin{equation}\label{equ: cclg-yxx}
  J_{B_p}\big(\Pi_{\infty}^{R_{\infty}-\an}(\ul{\lambda}_J)\big)[\fm_{\omega}]\{\fm_x\}
 \\ \hooklongrightarrow  J_{B_p}\big(\Pi_{\infty}^{R_{\infty}-\an}(\ul{\lambda}_{J'})\big)[\fm_{\omega}]\{\fm_x\}
\end{equation}
is not surjective (note $\sigma\in \Sigma(x)\cap J$).

    Assuming the claim, the theorem then follows by applying Breuil's adjunction formula to the right set of (\ref{equ: cclg-yxx}): Denote by $\fp_y$ (resp. $\fp_y'$) the prime ideal of $R_{\infty}$ (resp. of $S_{\infty}$) corresponding to $y$, thus
    \begin{equation*}
        J_{B_p}\big(W\big)[\fm_{\omega}]\{\fm_x\}=  J_{B_p}\big(W\big)[\fp_y', T_p^0=\delta]\{\fp_y, T_p=\delta\}
    \end{equation*}
    for any locally $R_{\infty}$-analytic closed subrepresentation $W$ of $\Pi_{\infty}^{R_{\infty}-\an}$. By (\ref{equ: cclg-yxx}), there exists \begin{equation*}
      v\in J_{B_p}\big(\Pi_{\infty}^{R_{\infty}-\an}(\ul{\lambda}_{J'})\big)[\fp_y', T_p^0=\delta]\{\fp_y, T_p=\delta\} \setminus J_{B_p}\big(\Pi_{\infty}^{R_{\infty}-\an}(\ul{\lambda}_{J})\big)[\fp_y', T_p^0=\delta]\{\fp_y, T_p=\delta\}.
    \end{equation*}
    Consider the $T_p$-subrepresentation generated by $v$, which has the form $\delta_{\wt(\delta)}\otimes_E \pi_{\psi_{\delta}}$ where $\pi_{\psi_{\delta}}$ is a finite dimensional unramified smooth representation of $T_p$ with Jodan-Holder factors all isomorphic to $\psi_{\delta}=\delta\delta_{\wt(\delta)}^{-1}$.  By Proposition \ref{prop: cclg-pve}, one has (see Appendix \ref{sec: cclg-A.2} for the notation)
    \begin{multline}\label{equ: cclg-MeE}
      \Hom_{G_p}\big(\cF_{\overline{B}_p}^{G_p}\big(\overline{M}_{J'}(-\wt(\delta))^{\vee}, \pi_{\psi_{\delta}}\otimes_E \delta_{B_p}^{-1}\big), \Pi_{\infty}^{R_{\infty}-\an}(\ul{\lambda}_{J'})[\fp_y']\{\fp_y\}\big)\\ \xlongrightarrow{\sim} \Hom_{T_p}\big(\pi_{\psi_{\delta}}\otimes_E \delta_{\wt(\delta)}, J_{B_p}\big(\Pi_{\infty}^{R_{\infty}-\an}(\ul{\lambda}_{J'})[\fp_y']\{\fp_y\}\big)\big).
    \end{multline}
    On the other hand, $\cF_{\overline{B}_p}^{G_p}\big(\overline{M}_{J'}(-\wt(\delta))^{\vee}, \pi_{\psi_{\delta}}\otimes_E \delta_{B_p}^{-1}\big)$ sits in an exact sequence (e.g. see Proposition \ref{prop: cclg-cnee})
    \begin{multline*}
    0 \lra V_{\sigma}:=\cF_{\overline{B}_p}^{G_p}\big(\overline{L}(-s_{\sigma}\cdot \wt(\delta)), \pi_{\psi_{\delta}}\otimes_E \delta_{B_p}^{-1}\big)\lra \cF_{\overline{B}_p}^{G_p}\big(\overline{M}_{J'}(-\wt(\delta))^{\vee}, \pi_{\psi_{\delta}}\otimes_E \delta_{B_p}^{-1}\big)\\  \lra  V_0:=\cF_{\overline{B}_p}^{G_p}\big(\overline{L}(-\wt(\delta)), \pi_{\psi_{\delta}}\otimes_E \delta_{B_p}^{-1}\big)\lra 0.
    \end{multline*}
    Let $f: \cF_{\overline{B}_p}^{G_p}\big(\overline{M}_{J'}(-\wt(\delta))^{\vee}, \pi_{\psi_{\delta}}\otimes_E \delta_{B_p}^{-1}\big) \ra \Pi_{\infty}^{R_{\infty}-\an}(\ul{\lambda}_{J'})[\fp_y']\{\fp_y\}$ be the map in the set of the left hand side of (\ref{equ: cclg-MeE}) corresponding to the injection map induced by $v$ (as an element in the set of the right hand side). We see $f$ does not factor through $V_0$, since otherwise, $\Ima(f)$ would be contained in $\Pi_{\infty}^{R_{\infty}-\an}(\ul{\lambda}_J)[\fp_y']\{\fp_y\}$, and hence $v\in J_{B_p}\big(\Pi_{\infty}^{R_{\infty}-\an}(\ul{\lambda}_{J})[\fp_y']\{\fp_y\}\big)$ by taking Jacquet-Emerton functor, a contradiction. Thus we see
    \begin{equation*}
      \Hom_{G_p}\big(V_{\sigma}, \Pi_{\infty}^{R_{\infty}-\an}(\ul{\lambda}_{J'})\big)[\fp_y']\{\fp_y\}=\Hom_{G_p}\big(V_{\sigma},\Pi_{\infty}^{R_{\infty}-\an}(\ul{\lambda}_{J'})[\fp_y']\{\fp_y\}\big) \neq 0.
    \end{equation*}Since any irreducible constituent of $V_{\sigma}$ is isomorphic to $I(\delta_{\sigma}^c\delta_{B_p}^{-1})$, we deduce
    \begin{equation}\label{equ: cclg-apy}
      \Hom_{G_p}\big(I(\delta_{\sigma}^c\delta_{B_p}^{-1}), \Pi_{\infty}^{R_{\infty}-\an}(\ul{\lambda}_{J'})[\fp_y']\{\fp_y\}\big) \neq 0.
    \end{equation}
    By Proposition \ref{prop: cclg-pve}, the set in (\ref{equ: cclg-apy}) can be identified with $J_{B_p}(\Pi_{\infty}^{R_{\infty}-\an}(\ul{\lambda}_{J'}))[\fp_y', T_p=\delta_{\sigma}^c]\{\fp_y\}$, which is in particular finite dimensional \big(since $J_{B_p}(\Pi_{\infty}^{R_{\infty}-\an}(\ul{\lambda}_{J'}))$ is essentially admissible as a $\Z_p^q \times T_p$-representation, where $S_{\infty}\cong \co_E\llbracket \Z_p^q \rrbracket$\big). One can then deduce from (\ref{equ: cclg-apy}):
    \begin{equation*}\Hom_{G_p}\big(I(\delta_{\sigma}^c\delta_{B_p}^{-1}), \Pi_{\infty}^{R_{\infty}-\an}(\ul{\lambda}_{J'})[\fp_y]\big) \neq 0,\end{equation*}
    the theorem follows.

     We prove the claim. Consider the finite free $(A_J)_{\fm_{\omega}}$-module (resp. $(A_{J'})_{\fm_{\omega}}$-module) $(\overline{M}_J)_{\fm_{\omega}}$ (resp. $(M_{J'})_{\fm_{\omega}}$), and let $s_J:=\rk_{(A_J)_{\fm_{\omega}}}  (\overline{M}_J)_{\fm_{\omega}}$ (resp. $s_{J'}:=\rk_{(A_{J'})_{\fm_{\omega}}}  (M_{J'})_{\fm_{\omega}}$). By the isomorphisms before the claim, it is sufficient to prove $s_J<s_J'$.

     By Corollary \ref{cor: cclg-o0e} (2), $\cM_{\infty}(\ul{\lambda}_J)$ (resp. $\cM_{\infty}(\ul{\lambda}_{J'})$) is Cohen-Macaulay over $X_p(\overline{\rho},\ul{\lambda}_J)$ \big(resp. over $X_p(\overline{\rho},\ul{\lambda}_{J'})$\big). By Theorem \ref{thm: cclg-pox}, $X_p(\overline{\rho},\ul{\lambda}_J)$ (resp. $X_p(\overline{\rho},\ul{\lambda}_{J'})$) is smooth at $x$. Thus by \cite[Cor. 17.3.5 (i)]{EGAiv1}, shrinking $U_J$ (resp. $U_{J'}$), one can assume $M_{J}$ (resp. $M_{J'})$ is locally free as a $B_{J}$-module (resp. as a $B_{J'}$-module). Hence $(\overline{M}_J)_{\fm_{\omega}}=(\overline{M}_J)_{\fm_x}$ \big(resp. $(M_{J'})_{\fm_{\omega}}=(M_{J'})_{\fm_x}$\big) is a free $(B_J)_{\fm_{\omega}}=(B_J)_{\fm_x}$-module  \big(resp. $(B_{J'})_{\fm_{\omega}}=(B_{J'})_{\fm_x}$-module\big), say of rank $r_J$ (resp. $r_{J'}$). Since $(\overline{M}_J)_{\fm_\omega}$ (resp. $(M_{J'})_{\fm_{\omega}}$) is free over $(A_J)_{\fm_{\omega}}$ (resp. $(A_{J'})_{\fm_{\omega}}$), this implies in particular $(B_J)_{\fm_x}$ (resp. $(B_{J'})_{\fm_x}$), as a direct factor of $(\overline{M}_J)_{\fm_{\omega}}$ (resp. of $(M_{J'})_{\fm_{\omega}}$),  is also a free $(A_J)_{\fm_{\omega}}$ (resp. $(A_{J'})_{\fm_{\omega}}$)-module, say of rank $e_J$ (resp. $e_{J'}$). It is straightforward to see $s_J=r_Je_J$, $s_{J'}=r_{J'}e_{J'}$.

We have
    \begin{equation*}
      (\overline{M}_J/\fm_x \overline{M}_J)^{\vee}\xlongrightarrow{\sim}  J_{B_p}\big(\Pi_{\infty}^{R_{\infty}-\an}(\ul{\lambda}_J)[\fm_x]\big)
    \end{equation*}
    \begin{equation*}
      \Big(\text{resp. }(M_{J'}/\fm_x M_{J'})^{\vee}\xlongrightarrow{\sim}  J_{B_p}\big(\Pi_{\infty}^{R_{\infty}-\an}(\ul{\lambda}_{J'})[\fm_x]\big)\Big).
    \end{equation*}
From the natural injection
    \begin{equation*}J_{B_p}\big(\Pi_{\infty}^{R_{\infty}-\an}(\ul{\lambda}_J)[\fm_x]\big)\hooklongrightarrow  J_{B_p}\big(\Pi_{\infty}^{R_{\infty}-\an}(\ul{\lambda}_{J'})[\fm_x]\big),\end{equation*}
  We deduce  $r_J=\dim_{k(x)} \overline{M}_J/\fm_xM_J\leq \dim_{k(x)} M_{J'}/\fm_x M_{J'}= r_{J'}$. It is thus sufficient to prove $e_J<e_{J'}$, which would follow from Corollary \ref{cor: cclg-xts}:

    Let $B_J'$ (resp. $A_J'$) be the quotient of $B_{J'}$  (resp. of $A_{J'}$) with $\Spm B_J'\cong \Spm B_{J'}\times_{\cT_p^0(\ul{\lambda}_{J'})} \cT_p^0(\ul{\lambda}_J)$ \big(resp. $\Spm A_J'\cong \Spm A_{J'}\times_{\cT_p^0(\ul{\lambda}_{J'})} \cT_p^0(\ul{\lambda}_J)$\big), $\overline{M}_J':=M_{J'}\otimes_{B_{J'}} B_J' \cong M_{J'} \otimes_{A_{J'}} A_J'$. Since $B_{J'}$ is locally free over $A_{J'}$, $B_J'$ is locally free over $A_J'$ and in particular, $(B_J')_{\fm_x}$ is a free $(A_J')_{\fm_{\omega}}$-module of rank $e_{J'}$. For a noetherian local $E$-algebra $R$, denote by $R^{\wedge}$ the completion of $R$ at its maximal ideal. One gets in particular complete noetherian local $E$-algebras $(B_J)_{\fm_x}^{\wedge}$, $(B_J')_{\fm_x}^{\wedge}$, $(A_J)_{\fm_{\omega}}^{\wedge}$, $(A_J')_{\fm_{\omega}}^{\wedge}$, which are in fact the complete local algebras of $X_p(\overline{\rho},\ul{\lambda}_J)$ at $x$, $X_p(\overline{\rho},\ul{\lambda}_J, J')$ at $x$, $\cW_{\infty}(\ul{\lambda}_J)$ at ${\omega}$, $\cW_{\infty}(\ul{\lambda}_J)$ at ${\omega}$ respectively. In particular, $(A_J)_{\fm_{\omega}}^{\wedge}\cong (A_J')_{\fm_{\omega}}^{\wedge}$. The natural morphisms $X_p(\overline{\rho},\ul{\lambda}_J) \hookrightarrow X_p(\overline{\rho},\ul{\lambda}_J,J')\ra \cW_{\infty}(\ul{\lambda}_J)$ induce
    \begin{equation*}
      (A_J)_{\fm_{\omega}}^{\wedge} \hooklongrightarrow (B_J')_{\fm_x}^{\wedge} \twoheadlongrightarrow (B_J)_{\fm_x}^{\wedge}.
    \end{equation*}
    By Corollary \ref{cor: cclg-xts}, the last map is \emph{not} bijective \big(since $J\setminus J'=\{\sigma\}\subset \Sigma(x)$\big). Since $(B_J')_{\fm_x}^{\wedge}$ (resp. $(B_J)_{\fm_x}^{\wedge}$) is free of rank $e_{J'}$ (resp. $e_J$) over $(A_J)_{\fm_{\omega}}^{\wedge}$, one gets $e_{J'}>e_J$ and hence $s_{J'}>s_J$, the claim (hence the theorem) follows.
\end{proof}
\begin{corollary}Assume Hypothesis \ref{hypo: cclg-TW}, then
  Conjecture \ref{conj: cclg-comp} (and hence Conjecture \ref{conj: cclg-las2}) is true.
\end{corollary}
\begin{proof}
  With the notation of Conjecture \ref{conj: cclg-comp}, by the discussion following Conjecture \ref{conj: cclg-comp}, it is sufficient to show $(z_s)_J^c:=(\fm_{\rho},(\delta_s)_J^c)\in \cE(U^p,\ul{\lambda}_{\Sigma_p\setminus J})$ for all $s\in \cS_2^{|S_p|}$, $J\subseteq \Sigma(z_s)$.
  We use the notation as in the end of \S~\ref{sec: cclg-4.1}, namely, for each $s\in \cS_2^{|S_p|}$, we have a classical point $x_s=(\fm_{\rho},\delta_s)$ in $X_p(\overline{\rho},\ul{\lambda}_{\Sigma_p})$. Note that $\Sigma(x_s)$ (defined right before  Corollary \ref{cor: cclg-xts}) is no other than the set $\Sigma(z_s)$ defined in \S~\ref{sec: cclg-4.1}. Note also that $\rho_{\widetilde{v}}$ is generic for $v\in S\setminus S_p$. Indeed,  $\rho$ corresponds to an automorphic representation $\pi$ of $G(\bA_{F^+})$ with cuspidal strong base change $\Pi$ to $\GL_2(\bA_{F})$ (e.g. see \cite[Prop. 3.4]{BHS2}) which is generic at all finite places of $F$. Applying Theorem \ref{thm: cclg-lrt} inductively on $|J|$ (starting with $J=\emptyset$), one sees $(x_s)_J^c\in X_p(\overline{\rho},\ul{\lambda}_{\Sigma_p\setminus J})$ for all $J \subseteq \Sigma(x_s)$ \big(note also $\Sigma((x_s)_J^c)=\Sigma(x_s)\setminus J$\big), which together with Lemma \ref{lem: cclg-joa} conclude the proof.
\end{proof}
We can also deduce from Theorem \ref{thm: cclg-lrt} some results on the existence of companion points in trianguline case. Keep Hypothesis \ref{hypo: cclg-TW}.  Let $\rho$ be a continuous representation of $\Gal_F$  such that $\rho\otimes \varepsilon \cong \rho^{\vee}\circ c$,  $\rho$ is unramified outside $S$, and suppose \begin{equation}\label{equ: cclg-fs}J_{B_p}\big(\widehat{S}(U^p,E)_{\overline{\rho}}^{\an}[\fm_{\rho}]\big)\neq 0.\end{equation} The latter is equivalent to that $\rho$ is attached to some  point in $\cE(U^p)_{\overline{\rho}}$ and implies that $\rho_{\widetilde{v}}$ is trianguline for all $v|p$. For $v|p$, let $\chi_{\widetilde{v}}$ be a continuous character of $T_{\widetilde{v}}$ in $E^{\times}$ such that $\chi_{\widetilde{v}}^{\natural}$ is a trianguline parameter of $\rho_{\widetilde{v}}$. We suppose
\begin{itemize}
  \item $\chi_{\widetilde{v}}$ is locally algebraic, very regular, and $\chi_{\widetilde{v}}\chi_{\wt(\chi_{\widetilde{v}})}^{-1}$ is unramified for all $v|p$;
  \item $\wt(\chi_{\widetilde{v}}^{\natural})_{\sigma,1}\neq \wt(\chi_{\widetilde{v}}^{\natural})_{\sigma,2}$ (distinct Hodge-Tate weights condition);
  \item for $v\in S\setminus S_p$, $\rho_{\widetilde{v}}$ is generic.
\end{itemize}
Let $S_p^+(\rho):=\{v\in S_p\ |\ \text{$\rho_{\widetilde{v}}$ is crystalline}\}$, and $\Sigma_p^+(\rho):=\cup_{v\in S_p^+(\rho)}\Sigma_{\widetilde{v}}$.

Let $v\in S_p^+(\rho)$. Let $\ul{\lambda}_{\Sigma_{\widetilde{v}}}:=(\lambda_{1,\sigma},\lambda_{2,\sigma})_{\sigma \in \Sigma_{\widetilde{v}}}\in \Z^{2|\Sigma_{\widetilde{v}}|}$ such that $\HT(\rho_{\widetilde{v}})=-\ul{\lambda}_{\Sigma_{\widetilde{v}}}^{\natural}=-(\lambda_{1,\sigma},\lambda_{2,\sigma}-1)_{\sigma\in \Sigma_{\widetilde{v}}}$. We have thus for $\sigma\in \Sigma_{\widetilde{v}}$,
\begin{equation*}\{\lambda_{1,\sigma}, \lambda_{2,\sigma}-1\}=\{\wt(\chi_{\widetilde{v}})_{\sigma,1}, \wt(\chi_{\widetilde{v}})_{\sigma,2}\}.\end{equation*}
Let $\alpha_{\widetilde{v},1}$, $\alpha_{\widetilde{v},2}$ be the two eigenvalues of the crystalline Frobenius $\varphi^{[F_{\widetilde{v},0}:\Q_p]}$ on $D_{\cris}(\rho_{\widetilde{v}})$. Since $\chi_{\widetilde{v}}$ is very regular, $\alpha_{\widetilde{v},1}\alpha_{\widetilde{v},2}^{-1}\neq 1, q_{\widetilde{v}}^{\pm 1}$. For $w_{\widetilde{v}} \in \cS_2$, put $\psi_{\widetilde{v}, w_{\widetilde{v}}}:= \unr_{\widetilde{v}}(q_{\widetilde{v}}^{-1}\alpha_{\widetilde{v},w_{\widetilde{v}}^{-1}(1)})\otimes \unr_{\widetilde{v}}(\alpha_{\widetilde{v},w_{\widetilde{v}}^{-1}(2)})$, and $\delta_{\widetilde{v}, w_{\widetilde{v}}}:=\psi_{\widetilde{v}, w_{\widetilde{v}}} \delta_{\ul{\lambda}_{\widetilde{v}}}$,
which is a locally algebraic character of $T_{\widetilde{v}}$. For  $w_{\widetilde{v}}\in \cS_2$ there exists $\Sigma(w_{\widetilde{v}}) \subseteq \Sigma_{\widetilde{v}}$ such that $\big((\delta_{w,\widetilde{v}})_{\Sigma(w_{\widetilde{v}})}^c\big)^{\natural}$ is a trianguline parameter of $\rho_{\widetilde{v}}$. 

Let $v\in S_p\setminus S_p^+(\rho)$. Thus $\rho_{\widetilde{v}}$ is trianguline non crystalline. Since $\chi_{\widetilde{v}}$ is very regular, we know $\rho_{\widetilde{v}}$ admits a unique triangulation given by $\chi_{\widetilde{v}}^{\natural}$. Let
 \begin{equation*}\Sigma(\rho_{\widetilde{v}}):=\{\sigma\in \Sigma_{\widetilde{v}}\ |\ \wt(\chi_{\widetilde{v}})_{\sigma,1}<\wt(\chi_{\widetilde{v}})_{\sigma,2}\}.\end{equation*}
  Thus there exists a locally algebraic character $\delta_{\widetilde{v}}$ of $T_{\widetilde{v}}$ with $\wt(\delta_{\widetilde{v}})$ dominant such that $(\delta_{\widetilde{v}})_{\Sigma(\rho_{\widetilde{v}})}^c=\chi_{\widetilde{v}}$.
For $w=(w_{\widetilde{v}})_{v\in S_p^+(\rho)}\in \cS_2^{|S_p^+(\rho)|}$, put
\begin{equation*}
  \delta_w:=(\otimes_{v\in S_p^+(\rho)} \delta_{\widetilde{v}, w_{\widetilde{v}}})\otimes (\otimes_{v\in S_p\setminus S_p^+(\rho)} \delta_{\widetilde{v}}),
\end{equation*}
and $\psi_w:=\delta_w \delta_{\wt(\delta_w)}^{-1}$ (which is an unramified character of $T_p$). The following lemma is an easy consequence of the global triangulation theory (e.g. see the discussion below Conjecture \ref{conj: cclg-comp}).
\begin{lemma}\label{cclg-triptA}
Keep the above situation, if $(\fm_{\rho}, \chi)\in \cE(U^p)_{\overline{\rho}}$, then there exist $w=(w_{\widetilde{v}})_{v\in S_p^+(\rho)}\in \cS_2^{|S_p^+(\rho)|}$, and $J_{\widetilde{v}}\subseteq \Sigma(w_{\widetilde{v}})$ \big(resp. $J_{\widetilde{v}}\subseteq \Sigma(\rho_{\widetilde{v}})$\big) if $v\in S_p^+(\rho)$ (resp. if $v\in S_p\setminus S_p^+(\rho)$) such that $\chi=(\delta_w)_{J}^s$ with $J:=\cup_{v\in S_p} J_{\widetilde{v}}$.
\end{lemma}
By assumption (\ref{equ: cclg-fs}), there exists $(\fm_{\rho}, \chi)\in \cE(U^p)$ for certain $\chi$. Let $\fw \in (\fw_{\widetilde{v}})_{v\in S_p^+(\rho)}\in \cS_2^{|S_p^+(\rho)|}$ be attached to $\chi$ as in the above lemma. By Theorem \ref{thm: cclg-lrt}, we have the following results on the companion points of $(\fm_{\rho}, \chi)$.
\begin{corollary}\label{cor: cclg-net}
We have $(\fm_{\rho},(\delta_{\fw})_J^c)\in \cE(U^p)_{\overline{\rho}}$ for any $J\subseteq \Sigma(\fw):=(\cup_{v\in S_p^+(\rho)} \Sigma(\fw_{\widetilde{v}})) \cup (\cup_{v\in S_p\setminus S_p^+(\rho)} \Sigma(\rho_{\widetilde{v}}))$.
\end{corollary}
\begin{proof}
Let $J\subseteq \Sigma(\fw)$ be such that $(\fm_{\rho}, (\delta_\fw)_J^c)\in \cE(U^p)_{\overline{\rho}}$ (i.e. $\chi=(\delta_\fw)_J^c$, cf. Lemma \ref{cclg-triptA}) . Thus
\begin{equation*}J_{B_p}(\widehat{S}(U^p,E)^{\an}_{\overline{\rho}}[\fm_{\rho}])[T_p=(\delta_\fw)_J^c]\neq 0.\end{equation*}Applying Breuil's adjunction formula, we see
  \begin{equation*}
    \Hom_{G_p}\Big(\cF_{\overline{B}_p}^{G_p}\big(\overline{M}(-\wt(\delta_J^c))^{\vee}, \psi_{\fw}\delta_{B_p}^{-1}\big), \widehat{S}(U^p,E)_{\overline{\rho}}^{\an}[\fm_{\rho}]\Big)=J_{B_p}\big(\widehat{S}(U^p,E)_{\overline{\rho}}^{\an}[\fm_{\rho}]\big)[T_p=(\delta_{\fw})_J^c]\neq 0.
  \end{equation*}Since the irreducible constituents of $\cF_{\overline{B}_p}^{G_p}\big(\overline{M}(-\wt((\delta_{\fw})_J^c))^{\vee}, \psi_{\delta}\delta_{B_p}^{-1}\big)$ are given by $$\{I((\delta_\fw)_{J\cup S}^c\delta_{B_p}^{-1})\}_{S\subseteq \Sigma_p\setminus J},$$ (where the irreducibility follows from the assumption on $\chi_{\widetilde{v}}$), there exists $S\subseteq \Sigma_p\setminus J$ such that
  \begin{equation*}
    I((\delta_\wp)_{J\cup S}^c \delta_{B_p}^{-1}) \hooklongrightarrow  \widehat{S}(U^p,E)_{\overline{\rho}}^{\an}[\fm_{\rho}] \cong \Pi_{\infty}^{R_{\infty}-\an}[\fm_{\rho}]
  \end{equation*}
  where we also use $\fm_{\rho}$ to denote the maximal ideal of $R_{\infty}[1/p]$ associated to $\rho$. We thus get a point $z_{J\cup S}^c$ (resp. $x_{J\cup S}^c$) in $\cE(U^p, \wt(\delta)_{\Sigma_p \setminus (J\cup S)})_{\overline{\rho}}$ \big(resp. in $X_p(\overline{\rho}, \wt(\delta)_{\Sigma_p \setminus (J\cup S)})$\big). By Lemma \ref{cclg-triptA}, we have $J\cup S\subseteq \Sigma(\fw)$. Note also that  the point $x_{J\cup S}^c$ satisfies the conditions in Theorem \ref{thm: cclg-lrt}. Thus by Theorem \ref{thm: cclg-lrt} (inductively), we get $x_{J'}^c\in X_p(\overline{\rho},\wt(\delta)_{\Sigma_p \setminus J'})$ for all $S \cup J \subseteq J' \subseteq \Sigma(\fw)$. Using a similar result as in Lemma \ref{lem: cclg-joa}, one gets $z_{J'}^c=(\fm_{\rho},(\delta_{\fw})_{J'}^c)\in \cE(U^p,\wt(\delta)_{\Sigma_p \setminus J'})_{\overline{\rho}}$ for all $S \cup J \subseteq J' \subseteq \Sigma(\fw)$. The direction that $z_{\Sigma(\fw)}^c\in \cE(U^p)_{\overline{\rho}}\Rightarrow z_{J}^c\in \cE(U^p)_{\overline{\rho}}$ for all $J\subseteq \Sigma(\fw)$ follows from \cite[Prop. 8.1 (ii)]{Br13II} (see \cite[Thm. 5.1]{BHS2} for patched eigenvariety case), which allows to conclude.
\end{proof}At last, we propose a locally analytic socle conjecture in this case\footnote{We will prove this conjecture in the next section, using the theory of \cite{BHS3} (together with Corollary \ref{cor: cclg-net}).}. Denote by $C(\rho):=\cup_{v\in S_p} C(\rho_{\widetilde{v}})$ (cf. \S~\ref{sec: cclg-xdt}). Note that if $C(\rho)_{\widetilde{v}}=\Sigma_{\widetilde{v}}$, then $v\in S_p^+(\rho)$. By \cite[Prop.A.3]{Ding4}, $C(\rho)\supseteq \Sigma_p\setminus \Sigma(w)$ (for all $w$).
\begin{conjecture}\label{conj: cclg-soctri}
  Keep the notation and assumption as above, then $I(\chi\delta_{B_p}^{-1}) \hookrightarrow \widehat{S}(U^p,E)^{\an}_{\overline{\rho}}[\fm_{\rho}]$ if and only if there exists $w=(w_{\widetilde{v}})_{v\in S_p^+(\rho)} \in \cS_2^{|S_p^+(\rho)|}$, and $\Sigma_p\setminus C(\rho) \subseteq J \subseteq \Sigma(w)$ such that $\chi=(\delta_w)_J^c$.
\end{conjecture}
\begin{remark}
 (1) In particular, $\widehat{S}(U^p,E)^{\an}_{\overline{\rho}}(\wt(\delta_w)_{C(\rho)})[\fm_{\rho}]\neq 0$ for all $w\in S_2^{|S_p^+(\rho)|}$. In other words, $\widehat{S}(U^p,E)^{\an}_{\overline{\rho}}[\fm_{\rho}]$ should have non zero $C(\rho)$-classical vectors.

 (2) The ``only if" part is known. Indeed, if $I(\chi\delta_{B_p}^{-1})\hookrightarrow\widehat{S}(U^p,E)^{\an}_{\overline{\rho}}[\fm_{\rho}]$, then $(\fm_{\rho}, \chi)\in \cE(U^p)_{\overline{\rho}}$. By Lemma \ref{cclg-triptA}, there exist $w\in \cS_2^{|S_p^+(\rho)|}$ and $J\subseteq \Sigma(w)$ such that $\chi=\delta_{J}^c$. Moreover, the injection implies also that $(\fm_{\rho}, (\delta_w)_J^c)\in \cE(U^p, (\wt(\delta_w))_{\Sigma_p\setminus J})_{\overline{\rho}}$. By Theorem \ref{thm: cclg-pen}, $\rho_{\widetilde{v}}$ is $\Sigma_{\widetilde{v}}\setminus (J\cap \Sigma_{\widetilde{v}})$-de Rham. Hence $J\supseteq \Sigma_p\setminus C(\rho)$.
\end{remark}
In terms of companion points, we should have
\begin{conjecture}\label{conj: cclg-comptri}
    Keep the notation and assumption as above.

(1) A point $(\fm_{\rho}, \chi)\in \cE(U^p)_{\overline{\rho}}$ if and only if there exists $w\in \cS_2^{|S_p^+(\rho)|}$, and $J \subseteq \Sigma(w)$ such that $\chi=(\delta_w)_J^c$.

(2) For $w\in \cS_2^{|S_p^+(\rho)|}$, if $J\supseteq \Sigma_p\setminus C(\rho)$, then $(\fm_{\rho},(\delta_w)_J^c)$ lies moreover in $\cE(U^p, \wt(\delta_w)_{\Sigma_p\setminus J})_{\overline{\rho}}$.
\end{conjecture}
\begin{remark}
(1) Note that in (1), we do not have $J\supseteq \Sigma_p\setminus C(\rho)$. Actually, as seen in the proof of Corollary \ref{cor: cclg-net}, for $J_1\subseteq J_2$, by \cite[Prop. 8.1 (ii)]{Br13II}, if $(\fm_{\rho}, (\delta_w)_{J_2}^c)\in \cE(U^p)_{\overline{\rho}}$, then $(\fm_{\rho}, (\delta_w)_{J_1}^c)\in \cE(U^p)_{\overline{\rho}}$.

(2) In the case where $S_p^+(\rho)=\emptyset$, (1) is already known by Lemma \ref{cclg-triptA} and Corollary \ref{cor: cclg-net}.
\end{remark}
 By the same argument as in \S~\ref{sec: cclg-4.1}, we have
\begin{lemma}\label{conj: cclg-equivtri}The conjectures \ref{conj: cclg-soctri} and \ref{conj: cclg-comptri} are equivalent.
\end{lemma}

\section{Partial classicality}
In this section, we apply the theory of Breuil-Hellmann-Schraen \cite{BHS3} on the local model of the trianguline variety to the closed subspaces considered in \S~\ref{sec: cclg-2.1}. In particular, we prove Conjecture \ref{conj: cclg-soctri} (hence also Conjecture \ref{conj: cclg-comptri}).
\subsection{Preliminaries}\label{sec: cclg-5.1}Recall some geometric representation theory and we refer to \cite[\S~2]{BHS3} for details. In this section, we let $G:=\GL_2/E$,  $B$ be the Borel subgroup of $G$ of upper triangular matrices, and let $\ug:=\gl_2$, $\ub$ be the Lie algebra (over $E$) of $G$, $B$ respectively. Let $\widetilde{\ug}$ be the $E$-scheme:
\begin{equation*}\widetilde{\ug}:=\{(gB, \psi)\in G/B\times \ug\ |\ \Ad(g^{-1}) \psi \in \ub\}.\end{equation*}
It is known that $\widetilde{\ug}$ is smooth and irreducible. The natural morphism $\widetilde{\ug} \ra \ug$, $(gB, \psi)\mapsto \psi$ is proper and surjective.

Let $X:=\widetilde{\ug}\times_{\ug} \widetilde{\ug}=\{(g_1 B, g_2 B, \psi)\in G/B \times G/B \times \ug\ |\ \Ad(g_1^{-1})\psi \in \ub, \ \Ad(g_2^{-1}) \psi \in \ub\}$. For $w\in \cS_2$ (the Weyl group of $G$), put $U_w:=G(1,\dot{w})B\times B\subset G/B \times G/B$ where $\dot{w}\in N_G(T)$ is some lift of $w$. We have
\begin{equation*}
  G/B \times G/B =\sqcup_{w\in \cS_2} U_w.
\end{equation*}
It is known that $U_w$ is a locally closed subscheme,  and is smooth of dimension $\dim G-\dim B+\lg(w)$. Consider
\begin{equation*}
  \pi: X \hooklongrightarrow G/B \times G/B \times \ug \twoheadlongrightarrow G/B \times G/B.
\end{equation*}
Let $V_w:=\pi^{-1}(U_w)$, and $X_w$ be the Zariski closure of $V_w$ in $X$. It is known that the $E$-schemes $X$, $X_w$ for $w\in \cS_2$ are equidimensional of dimension $\dim G=4$, and are reduced. Moreover, the irreducible components of $X$ are given by $X_w$ for $w\in \cS_2$.

Denote by $\kappa: X \ra \ug$ the natural morphism, and $\overline{X}$ the fiber of $X$ at $0\in \ug$ (which is thus a closed subscheme of $X$). it is easy to see $\overline{X}\cong G/B \times G/B$ (since the fiber of $\widetilde{\ug}$ at $0\in \ug$ is isomorphic to $G/B$). For $w\in \cS_2$, let $\overline{X}_w:=\overline{X}\cap X_w=\overline{X}\times_{X} X_w$, which is thus the fiber of $X_w$ at $0\in \ug$.
\begin{lemma}\label{lem: cclg-grt}
  We have $\overline{X_1}\cong G/B\xrightarrow{\Delta} G/B\times G/B\cong \overline{X}$ (where $\Delta$ denotes the diagonal morphism), and $\overline{X_s}\cong \overline{X}$ (where $1\neq s\in \cS_2$). In particular, $\overline{X_1}$ (resp. $\overline{X_s}$) is smooth of dimension $1$ (resp. $2$).
\end{lemma}
\begin{proof}
  By \cite[Lem.  2.2.4]{BHS3}, $X_1\cong V_1$. By definition $V_1\cong \widetilde{\ug}\xrightarrow{\Delta} \widetilde{\ug}\times_{\ug} \widetilde{\ug}$. Since the fiber of $\widetilde{\ug}$ at $0\in \ug$ is naturally isomorphic to $G/B$. The first part follows. Let $\overline{V_s}$ be the fiber of $V_s$ at $0\in \ug$. Since $U_s$ is Zariski-dense in $G/B\times G/B$, it is not difficult to check $\overline{V_s}$ is Zariski-dense in $\overline{X}$. Since $\overline{X}$ is reduced, we deduce $\overline{X_s}   \xrightarrow{\sim} \overline{X}$.
\end{proof}
\subsection{Local model of the trianguline variety}
Let $J \subseteq \Sigma_L$, $\ul{k}_J=(k_{1,\sigma}, k_{2,\sigma})_{\sigma\in J}\in \Z ^{2|J|}$ with $k_{1,\sigma}> k_{2,\sigma}$ for all $\sigma\in J$. For $\fu=(\fu_{\sigma})_{\sigma\in J}\in \prod_{\sigma\in J} \cS_2$, denote by $\ul{k}_J^{\fu}:=\big(k_{\fu_{\sigma}^{-1}(1),\sigma}, k_{\fu_{\sigma}^{-1}(2), \sigma}\big)_{\sigma\in J}$.
Let $x=(r,\delta)$ be an $E$-point of  $X_{\tri,J-\dR}(\overline{r}_L, \ul{k}_J^{\fu})$ (cf. \S~\ref{sec: cclg-2.1}), and suppose $\delta=\delta_1 \otimes \delta_2$ is locally algebraic, very regular with $\mathrm{wt}(\delta)_{1,\sigma}\neq \mathrm{wt}(\delta)_{2,\sigma}$ for all $\sigma\in \Sigma_L$.
We have $\wt(\delta_i)_{\sigma}=k_{\fu_{\sigma}^{-1}(i), \sigma}$ for all $\sigma\in J$, and
\begin{equation*}
  \Sigma^+(\delta) \cap J=\{\sigma\in J\ |\ \fu_{\sigma} =1\}.
\end{equation*}
Note also  $J\subseteq C(r)$ (cf. \S~\ref{sec: cclg-2.1}) since $x\in X_{\tri,J-\dR}(\overline{r}_L, \ul{k}_J^{\fu})$. 

We recall some results of \cite{BHS3} (where we only consider the $\GL_2$ case and we refer to \emph{loc. cit.} for details and for more general statements), from which, in particular, we deduce Theorem \ref{thm: cclg-locmod}.  Let $\cC_E$ denote the category of local artinian $E$-algebras of residue field isomorphic to $E$.

As in \cite[\S~3.6]{BHS3}, denote by $X_r$ the groupoid over $\cC_{E}$  of deformations of $r: \Gal_{L} \ra \GL_2(E)$, $V$ be the representation of $\Gal_L$ associated to $r$, and $X_V$ be the groupoid over $\cC_E$ of deformations of $V$. Thus $X_r$ is pro-representable by the formal scheme $\Spf R^{\square}_{V}$, where $R^{\square}_V$ denotes the framed universal deformation ring of $V$. There is a natural morphism $X_r\ra X_V$ (forgetting the framing), which is relatively representable, formally smooth of relative dimension $4$.

Let $D:=D_{\rig}(V)$ (which carries the same information as $r$), $X_D$ be the groupoid over $\cC_E$ of deformations of $D$. Thus $X_D\cong X_V$. Let $\cM:=D[1/t]$, which admits a triangulation $\cM_{\bullet}:=(\cM_i)_{i=1,2}$ of $(\varphi,\Gamma)$-modules over $\cR_E[1/t]$ such that $\cM_1\cong \cR_E(\delta_1)[1/t]$, and $\cM_2\cong \cM$, $\cM_2/\cM_1\cong \cR_E(\delta_2)[1/t]$. Let $X_{\cM}$ be the groupoid over $\cC_E$ of deformations of $\cM$, $X_{\cM, \cM_{\bullet}}$ be the groupoid over $\cC_E$ of deformations of $(\cM, \cM_{\bullet})$ (cf. \cite[\S~3.3]{BHS3}). There is a natural morphism (by forgetting the filtration $\cM_{\bullet}$):
\begin{equation*}
  X_{\cM, \cM_{\bullet}} \ra X_{\cM}.
\end{equation*}
There is also a natural morphism (by inverting $t$) $X_D \ra X_{\cM}$, and we put $X_{D,\cM_{\bullet}}:=X_D \times_{X_{\cM}} X_{\cM, \cM_{\bullet}}$, $X_{r,\cM_{\bullet}}:=X_r\times_{X_D} X_{D,\cM_{\bullet}}$.

Recall a little on Fontaine's theory of almost de Rham representations (e.g. see \cite[\S~3.1]{BHS3}). Let $B_{\pdR}^+:= B_{\dR}^+[\log t]$, $B_{\pdR}:=B_{\pdR}^+[1/t]\cong B_{\dR}\otimes_{B_{\dR}^+} B_{\pdR}^+$. The $\Gal_L$-action on $B_{\dR}$ extends uniquely to an action of $\Gal_L$ on $B_{\pdR}$ with $g(\log t)=\log(t)+\log(\chi_{\cyc} (g))$. Let $\nu_{B_{\pdR}}$ denote the unique $B_{\dR}$-derivation  of $B_{\pdR}$ such that $\nu_{B_{\pdR}}(\log(t))=-1$. We have that $\nu_{B_{\pdR}}$ and $\Gal_L$ commute, and both preserve $B_{\pdR}^+$.

Let $W:=B_{\dR}\otimes_{\Q_p} r$,  $W^+:=B_{\dR}^+\otimes_{\Q_p} r\cong W_{\dR}^+(D)$, where $W_{\dR}^+$ denotes the functor from the category of $(\varphi,\Gamma)$-modules over $\cR_E$ to the category of $B_{\dR}^+$-representations (cf. \cite[Prop. 2.2.6 (ii)]{Ber08}). As in \cite[\S~3.3]{BHS3}, one can extend $W_{\dR}^+$ to a functor $W_{\dR}$ from the category of $(\varphi,\Gamma)$-modules over $\cR_E[1/t]$ to the category of $B_{\dR}$-representations. In particular, applying $W_{\dR}$ to $\cM_{\bullet}$, we obtain a filtration of $B_{\dR}$-subrepresentations   of $W\cong W_{\dR}(\cM)$: \begin{equation}\label{equ: cclg-Fil1}
  \cF_{\bullet}=\{\cF_i\}:=\{W_{\dR}(\cM_i)\}.
\end{equation}
 Since $r$ has of distinct Sen weights, we know in particular $r$ is almost de Rham, i.e. $$D_{\pdR}(r):=(r\otimes_{\Q_p} B_{\pdR})^{\Gal_L}\cong (W\otimes_{B_{\dR}} B_{\pdR})^{\Gal_L}=:D_{\pdR}(W)$$ is free of rank $2$ over $L\otimes_{\Q_p} E$. Moreover, $D_{\pdR}(W)$ is equipped with an $L\otimes_{\Q_p} E$-linear (nilpotent) endomorphism $\nu_W$ induced by $\nu_{B_{\pdR}}\otimes 1$ on $B_{\pdR}\otimes_{B_{\dR}} W$. Note that by $L\otimes_{\Q_p} E\cong \prod_{\sigma\in \Sigma_L} E$, we have a natural decomposition $D_{\pdR}(W)\cong \prod_{\sigma\in \Sigma_L} D_{\pdR}(W)_{\sigma}$, where $\dim_{E} D_{\pdR}(W)_{\sigma}=2$, and  $\nu_W=(\nu_{W,\sigma})_{\sigma\in \Sigma_L}$, where $\nu_{W,\sigma}$ is a $E$-linear nilpotent operator on $D_{\pdR}(W)_{\sigma}$. For $\sigma\in \Sigma_L$,   $W$ is $\sigma$-de Rham (i.e. $r$ is $\sigma$-de Rham) if and only if $\nu_{W,\sigma}=0$.   The filtration $\cF_{\bullet}$ induces a complete flag $\cD_{\bullet}=\{\cD_i\}:=\{(\cF_i\otimes_{B_{\dR}} B_{\pdR})^{\Gal_L}\}$ of $D_{\pdR}(W)$. The $B_{\dR}^+$-lattice $W^+$ in $W$ induces another complete flag $\Fil_{W^+,\bullet}$ of $D_{\pdR}(W)$ with
\begin{equation*}
  \Fil_{W^+, i}(D_{\pdR}(W)):=\oplus_{\sigma\in \Sigma_L} \Fil_{W^+}^{k_{i,\sigma}}D_{\pdR}(W)_{\sigma}:=\oplus_{\sigma\in \Sigma_L} (t^{-k_{i,\sigma}} B_{\pdR}^+ \otimes_{B_{\dR}^+} W^+)^{\Gal_L}_{\sigma}.
\end{equation*}
We fix an isomorphism $\alpha: \big(E\otimes_{\Q_p} L\big)^2 \xrightarrow{\sim} D_{\pdR}(W)$. Via the isomorphism $\alpha$,  $\cD_{\bullet}$ (resp. $\Fil_{W^+, \bullet}$) corresponds thus to an $E$-point of the flag variety $\prod_{\sigma\in \Sigma_L} \GL_2/B$, still denoted by $\cD_{\bullet}$ (resp. by $\Fil_{W^+, \bullet}$). We also have an nilpotent element $N_W:=\alpha^{-1} \circ \nu_W \circ \alpha\in \ug_{\Sigma_L}(E)\cong \End_{E\otimes_{\Q_p} L} \big((E\otimes_{\Q_p} L)^2\big)$,\footnote{For a Lie algebra $\fh$ over $L$ and $\sigma\in \Sigma_L$, denote by $\fh_{\sigma}:=\fh\otimes_{L,\sigma} E$. For $J'\subseteq \Sigma_L$, denote by $\fh_{J'}:=\prod_{\sigma\in J'} \fh_{\sigma}$. Note we have $\fh_{\Sigma_L}\cong \fh\otimes_{\Q_p} E$.} which can decompose as $N_W=(N_{W,\sigma})_{\sigma\in \Sigma_L}$. Both the filtrations $\cD_{\bullet}$ and $\Fil_{W^+, \bullet}$ are stable by $\nu_W$. Hence we have
\begin{equation}\label{equ: cclg-point}y:=(y_{\sigma})_{\sigma\in \Sigma_L}:=(\Fil_{W^+, \bullet}, \cD_{\bullet}, N_W) \in \prod_{\sigma\in \Sigma_L} (G/B \times G/B \times \ug_{\sigma})\end{equation}actually lies in $X_{\Sigma_L}:=\prod_{\sigma\in \Sigma_L} X_{\sigma}$ (where $X_{\sigma}$ is isomorphic to the $E$-scheme $X$ in \S~\ref{sec: cclg-5.1}).

Let $X_W$ denote the groupoid over $\cC_E$ of deformations of $W$, and $X_W^{\square}$ be the groupoid over $\cC_E$ as in \cite[\S~3.1]{BHS3} (with respect to the framing $\alpha$\big). By \cite[Cor. 3.1.6]{BHS3}, the functor $D_{\pdR}$ induces an isomorphism of functors between $\iota: |X_W^{\square}|\xrightarrow{\sim} \widehat{\ug}_{{\Sigma_L}, N_W}$, where the latter denotes the completion of $\ug_{\Sigma_L}$ at $N_W$. Recall for $A\in \cC_E$, $|X_W^{\square}|(A)=\{(W_A, \iota_A, \alpha_A)\}/\sim$, where $W_A$ is an $A\otimes_{\Q_p} B_{\dR}$-representation of $\Gal_L$, $\iota_A: W_A\otimes_A E\xrightarrow{\sim} W$ (which implies that $W_A$ is almost de Rham, cf. \cite[Rem. 3.1.5]{BHS3}), and $\alpha_A: (A\otimes_{\Q_p} L)^2\xrightarrow{\sim} D_{\pdR}(W_A)$ is compatible with $\alpha$ and $\iota_A$. The morphism $\iota$ is given by sending $(W_A, \iota_A, \alpha_A)$ to $N_{W_A}:=\alpha_A^{-1} \circ \nu_{W_A} \circ \alpha_A \in \widehat{\ug}_{\Sigma_L, N_W}(A)$.

 Similarly as for $W$, we have $D_{\pdR}(W_A)\cong \prod_{\sigma\in \Sigma_L} D_{\pdR}(W_A)_{\sigma}$, which is equipped with an $L\otimes_{\Q_p} A$-linear nilpotent  operator $N_{W_A}= \prod_{\sigma \in \Sigma_L} N_{W_A, \sigma}$. Moreover, for $\sigma\in \Sigma_L$, $N_{W_A,\sigma}=0$ if and only if $W_A$ is $\sigma$-de Rham\footnote{We call $W_A$ $\sigma$-de Rham if $D_{\dR}(W_A)_{\sigma}$ is free of rank $2$ over $A$, which is equivalent to $\dim_E D_{\dR}(W_A)_{\sigma}=2\dim_E A$ in our case.}. We put $X_{W,J-\dR}$ to be the full subcategory of $X_W$ such that the objects $(A,W_A, \iota_A)$ in $X_{W,J-\dR}$ satisfy moreover that $W_A$ is $J$-de Rham. And we put $X_{W,J-\dR}^{\square}:=X_W^{\square}\times_{X_W} X_{W,J-\dR}$, which is the full subcategory of $X_W^{\square}$ consisting of $J$-de Rham objects. Since the natural morphism $X_W^{\square} \ra X_W$ is formally smooth, so is the induced morphism $X_{W,J-\dR}^{\square} \ra X_{W,J-\dR}$.
 Since $r$ is $J$-de Rham, $N_{W,\sigma}=0$ for $\sigma\in J$, and we put $N_{W}^J:=(N_{W,\sigma})_{\sigma\in \Sigma_L\setminus J}\in \ug_{\Sigma_L\setminus J}$.
 We view $\ug_{\Sigma_L \setminus J}\hookrightarrow \ug_{\Sigma_L}$ as the fiber of $\ug_{\Sigma_L}$ at $0\in \ug_J$, and let $\widehat{\ug}_{\Sigma_L\setminus J, N_{W}^J}$ be the completion of $\ug_{\Sigma_L \setminus J}$ at $N_{W}^J$. Thus
 \begin{equation*}
   \widehat{\ug}_{\Sigma_L \setminus J, N_{W}^{J}} \cong \widehat{\ug}_{\Sigma_L, N_W}\times_{\ug_{\Sigma_L}} \ug_{\Sigma_L \setminus J}.
 \end{equation*}
 By \cite[Cor.  3.1.6]{BHS3} and the above discussion, we have
\begin{lemma}
The isomorphism $\iota$ induces an isomorphism of functors
\begin{equation*}
  \iota: |X_{W, J-\dR}^{\square}| \xlongrightarrow{\sim} \widehat{\ug}_{\Sigma_L \setminus J, N_{W}^{J}},
\end{equation*}
and we have $X_{W,J-\dR}^{\square}\cong X_{W}^{\square} \times_{\widehat{\ug}_{\Sigma_L, N_W}} \widehat{\ug}_{\Sigma_L \setminus J, N_W^J}$.
\end{lemma}
Similarly, we have groupoids $X_{W^+}$, $X_{W^+}^{\square}$ over $\cC_E$ of deformations and framed deformations of $W^+$ respectively (cf. \cite[\S~3.2]{BHS3}). There exists a natural functor $X_{W^+} \ra X_{W}$ (by inverting $t$), and we have $X_{W^+}^{\square}\cong X_{W^+} \times_{X_W} X_W^{\square}$. As in \cite[\S~3.1]{BHS3} (see the discussion below \cite[Def. 3.1.8]{BHS3}), denote by $X_{W,\cF_{\bullet}}$ the groupoid over $\cC_E$ of deformations of $W$ together with the filtration $\cF_{\bullet}$ (cf. (\ref{equ: cclg-Fil1})). Put $X_{W,\cF_{\bullet}}^{\square}:=X_{W}^{\square}\times_{X_W} X_{W, \cF_{\bullet}}$, $X_{W^+, \cF_{\bullet}}:=X_{W^+} \times_{X_W} X_{W, \cF_{\bullet}}$ and $X^{\square}_{W^+, \cF_{\bullet}}:=X_{W^+, \cF_{\bullet}}\times_{X_{W^+}} X_{W^+}^{\square}$. By \cite[Cor.  3.5.8]{BHS3} and the proof, we have
\begin{equation}\label{equ: cclg-locmod}
  \big(|X_{W^+, \cF_{\bullet}}^{\square}|\cong \big) X_{W^+, \cF_{\bullet}}^{\square} \xlongrightarrow{\sim} \widehat{X}_{\Sigma_L, y}
\end{equation}
where the latter denotes the completion of $X_{\Sigma_L}=\prod_{\sigma\in \Sigma_L} X_{\sigma}$ at the point $y$  (cf. (\ref{equ: cclg-point})). Moreover, we have a commutative diagram of natural morphisms
\begin{equation*}
  \begin{CD}   X_{W^+, \cF_{\bullet}}^{\square} @> \sim>> \widehat{X}_{\Sigma_L, y}\\
  @VVV @VVV \\
|X_W^{\square}| @> \sim>> \widehat{\ug}_{\Sigma_L, N_W}
  \end{CD}.
\end{equation*}
For $w=(w_{\sigma})_{\sigma\in \Sigma_L}\in \prod_{\sigma\in \Sigma_L} \cS_2$, denote by $X_{W^+, \cF_{\bullet}}^{\square, w}:=X_{W^+, \cF_{\bullet}}^{\square} \times_{X_{\Sigma_L}} X_{\Sigma_L}^w$ where $X_{\Sigma_L}^w:=\prod_{\sigma\in \Sigma_L} X_{\sigma, w_{\sigma}}$ and $X_{\sigma,w_{\sigma}}$ is isomorphic to $X_{w_{\sigma}}$ of \S~\ref{sec: cclg-5.1} (which is thus an irreducible component of $X_{\sigma}$). 
Denote by $\widehat{X}^w_{\Sigma_L, y}$ the completion of $X_{\Sigma_L}^w$ at the point $y$ (if $y$ does not lie on $X_{\Sigma_L}^w$ then $\widehat{X}^w_{\Sigma_L, y}$ is empty).  By \cite[Thm.  2.3.6]{BHS3}, $X_{\Sigma_L}^w$ is normal, together with the fact that $X_{\Sigma_L}^w$ is irreducible, we deduce $\widehat{X}^w_{\Sigma_L, y}$ is irreducible.  Put
\begin{equation*}
  X_{W^+, \cF_{\bullet}}^{\square, w}:=X_{W^+, \cF_{\bullet}}^{\square}\times_{\widehat{X}_{\Sigma_L, y}} \widehat{X}_{\Sigma_L,y}^w.
\end{equation*}

For any groupoid $Y$ over $X_W$ (in this section), denote by $Y_{J-\dR}:=Y\times_{X_W} X_{W,J-\dR}$. Thus $Y_{J-\dR}$ is the full subcategory of $Y$ consisting of the objects which are sent to $J$-de Rham objects in $X_{W,J-\dR}$ via $Y\ra X_{W}$.
We have an isomorphism
\begin{equation*}X_{\Sigma_L}\times_{\ug_{\Sigma_L}} \ug_{\Sigma_L \setminus J} \cong \prod_{\sigma\in \Sigma_L \setminus J} X_{\sigma}\times \prod_{\sigma\in J} \overline{X_{\sigma}}=:Z_J
 \end{equation*}
 where $\overline{X_{\sigma}}$ is the fiber of $X_{\sigma}$ at  $0\in \ug_{\sigma}$  via the projection $X_{\sigma}\ra \ug_{\sigma}$, and is thus isomorphic to the $\overline{X}$ of \S~\ref{sec: cclg-5.1}. The point $y$ lies in $Z_J$ (since $r$ is $J$-de Rham), and we have $\widehat{X}_{\Sigma_L, y}\times_{\ug_{\Sigma_L}} \ug_{\Sigma_L\setminus J}\cong \widehat{Z}_{J,y}$,  where the latter denotes the completion of $Z_J$ at $y$. Similarly, we have
 \begin{equation}\label{equ cclg-ZJw}
   X_{\Sigma_L}^w\times_{\ug_{\Sigma_L}}  \ug_{\Sigma_L \setminus J} \cong \prod_{\sigma\in \Sigma_L \setminus J} X_{\sigma,w_{\sigma}} \times \prod_{\sigma\in J} \overline{X}_{\sigma, w_{\sigma}}=:Z_J^w
 \end{equation}
where $\overline{X}_{\sigma,w_{\sigma}}$ is isomorphic to the $\overline{X}_{w_{\sigma}}$  of \S~\ref{sec: cclg-5.1}. Let $\widehat{Z}_{J,y}^w$ be the completion of $Z_J^w$ at $y$ (which is empty if $y \notin Z_J^w$). We deduce from (\ref{equ: cclg-locmod}):
\begin{lemma}\label{lem: cclg-LM-lm} The groupoid $ X_{W^+, \cF_{\bullet}, J-\dR}^{\square} $ \big(resp. $  X_{W^+, \cF_{\bullet}, J-\dR}^{\square,w}$\big) is pro-representable by the formal scheme $\widehat{Z}_{J,y}$ (resp. $\widehat{Z}_{J,y}^w$).
\end{lemma}
\begin{proof}
We have (see (\ref{equ: cclg-locmod}) for the second isomorphism)
\begin{equation*}X_{W^+, \cF_{\bullet},J-\dR}^{\square}\cong X_{W^+, \cF_{\bullet}}^{\square}\times_{|X_W^{\square}|} |X_{W, J-\dR}^{\square}|\cong \widehat{X}_{\Sigma_L, y}\times_{\widehat{\ug}_{\Sigma_L, N_W}} \widehat{\ug}_{\Sigma_L \setminus J, N_W^J} \cong \widehat{Z}_{J,y}.
\end{equation*}
The case for $X_{W^+, \cF_{\bullet}, J-\dR}^{\square, w}$ is similar.
\end{proof}
The functor $W_{\dR}^+$ induces a natural morphism of groupoids over $\cC_E$: $X_D \ra X_{W^+}$. Similarly, the functor $W_{\dR}$ induces natural morphisms of groupoids over $\cC_E$: $X_{\cM} \ra X_{W}$, $X_{\cM, \cM_{\bullet}} \ra X_{W, \cF_{\bullet}}$. We deduce then a morphism
\begin{equation}\label{equ: cclg-LM-sm}
  X_{D, \cM_{\bullet}} \cong X_D\times_{X_{\cM}} X_{\cM, \cM_{\bullet}} \lra X_{W^+} \times_{X_W} X_{W, \cF_{\bullet}} \cong X_{W^+, \cF_{\bullet}}.
\end{equation}
Let $X_{D, \cM_{\bullet}}^{\square}:=X_{D, \cM_{\bullet}} \times_{X_W} X_W^{\square}$, and we have a morphism
\begin{equation}\label{equ: cclg-LM-sm2}X_{D, \cM_{\bullet}}^{\square} \lra X_{W^+, \cF_{\bullet}}^{\square}.
 \end{equation}
 By \cite[Cor.  3.5.6]{BHS3}, the morphisms (\ref{equ: cclg-LM-sm}) (\ref{equ: cclg-LM-sm2}) are formally smooth. Denote by $X_{W^+, \cF_{\bullet}}^{w}$ the image of $X_{W^+, \cF_{\bullet}}^{\square, w}$ by the forgetful morphism $X_{W^+,\cF_{\bullet}}^{\square} \ra X_{W^+, \cF_{\bullet}}$. By \cite[(3.26)]{BHS3}, we have actually $$X_{W^+, \cF_{\bullet}}^{\square, w}\cong X_{W^+, \cF_{\bullet}}^w \times_{X_{W^+, \cF_{\bullet}}} X_{W^+, \cF_{\bullet}}^{\square}.$$ Put $X_{D, \cM_{\bullet}}^w:=X_{D, \cM_{\bullet}}\times_{X_{W^+, \cF_{\bullet}}} X_{W^+, \cF_{\bullet}}^w$, $X_{D, \cM_{\bullet}}^{\square, w}:=X_{D, \cM_{\bullet}}^{\square} \times_{X_{W^+, \cF_{\bullet}}^{\square}} X_{W^+, \cF_{\bullet}}^{\square, w}$, which are hence formally smooth over $X_{W^+, \cF_{\bullet}}^w$, $X_{W^+, \cF_{\bullet}}^{\square,w}$ respectively. Let $X_{r, \cM_{\bullet}}^{w}:=X_{r, \cM_{\bullet}}\times_{X_{D, \cM_{\bullet}}} X_{D, \cM_{\bullet}}^{w}$, $X_{r,\cM_{\bullet}}^{\square, w}:=X_{D, \cM_{\bullet}}^{\square, w} \times_{X_D} X_r$. These groupoids are naturally over $X_W$, and we have thus groupoids $X_{D, \cM_{\bullet}, J-\dR}^{\square, w}$, $X_{r, \cM_{\bullet}, J-\dR}^{w}$, and $X_{r,\cM_{\bullet}, J-\dR}^{\square, w}$ etc.

Let $\fw:=\prod_{\sigma \in \Sigma^+(\delta)} s_{\sigma}\in \prod_{\sigma \in \Sigma^+(\delta)} \cS_2 \hookrightarrow \prod_{\sigma\in \Sigma_L} \cS_2$ (where $s_{\sigma}$ denotes the unique non-trivial element). By \cite[Cor.  3.7.8]{BHS3} (see the discussion below \cite[Prop. 3.7.3]{BHS3} and see \cite[(3.33)]{BHS3}), we have
\begin{equation}\label{equ: cclg-LM-key}
  \widehat{X_{\tri}(\overline{r}_L)}_x \xlongrightarrow{\sim} X_{r,\cM_{\bullet}}^{\fw} \longleftarrow X_{r,\cM_{\bullet}}^{\square,\fw} \lra X_{D, \cM_{\bullet}}^{\square, \fw} \lra X_{W^+, \cF_{\bullet}}^{\square, \fw} \cong \widehat{X}_{\Sigma_L, y}^{\fw},
\end{equation}
where all the morphisms are formally smooth.  Denote by $\widehat{X}_{\tri,J-\dR}(\overline{r}_L, \ul{k}_J^{\fu})_x$ the completion of the rigid space $X_{\tri,J-\dR}(\overline{r}_L, \ul{k}_J^{\fu})$ at the point $x$ (which is thus isomorphic to the completion of $X_{\tri,J-\dR}(\overline{r}_L, |\ul{k}_J|)$ at $x$, see (\ref{equ: cclg-JdR0})). By Lemma  \ref{lem: cclg-LM-lm}, we deduce from (\ref{equ: cclg-LM-key})
\begin{theorem} \label{thm: cclg-locmod}
  We have formally smooth morphisms:
  \begin{multline}\label{equ: cclg-locmod1}
  \widehat{X}_{\tri,J-\dR}(\overline{r}_L, \ul{k}_J^{\fu})_x\xlongrightarrow{\sim} X_{r,\cM_{\bullet}, J-\dR}^{\fw} \longleftarrow X_{r,\cM_{\bullet}, J-\dR}^{\square,\fw} \lra X_{D, \cM_{\bullet}, J-\dR}^{\square, \fw} \lra X_{W^+, \cF_{\bullet}, J-\dR}^{\square, \fw}\\ \cong \widehat{Z}_{J, y}^{\fw}.
  \end{multline}
\end{theorem}
\begin{proof}The composition
\begin{equation*}
  \widehat{X_{\tri}(\overline{r}_L)}_x \xlongrightarrow{\sim} X_{r,\cM_{\bullet}}^{\fw} \lra X_W
\end{equation*}
sends $(r_A,\delta_A)$ to $B_{\dR}\otimes_{\Q_p} r_A$ for $A\in \cC_E$. By definition (cf. \S~\ref{sec: cclg-2.1}), for $A\in \cC_E$ and $r_A\in \widehat{X_{\tri}(\overline{r}_L)}_x(A)$, $r_A\in \widehat{X}_{\tri,J-\dR}(\overline{r}_L, \ul{k}_J^{\fu})_x(A)$  if and only if $r_A$ is $J$-de Rham (or equivalently, $B_{\dR}\otimes_{\Q_p} r_A$ is $J$-de Rham). Thus the composition
\begin{equation*}
  \widehat{X}_{\tri,J-\dR}(\overline{r}_L, \ul{k}_J^{\fu})_x \hooklongrightarrow  \widehat{X_{\tri}(\overline{r}_L)}_x \lra X_W
\end{equation*}
factors through $X_{W,J-\dR}$, and induces an isomorphism
\begin{equation*}
  \widehat{X}_{\tri,J-\dR}(\overline{r}_L, \ul{k}_J^{\fu})_x \xlongrightarrow{\sim} \widehat{X_{\tri}(\overline{r}_L)}_x \times_{X_W} X_{W,J-\dR}.
\end{equation*}
The morphisms in (\ref{equ: cclg-locmod1}) then follow from (\ref{equ: cclg-LM-key}) by taking the fiber product $(-)\times_{X_W} X_{W,J-\dR}$.
\end{proof}
\begin{corollary}\label{cor: cclg-pjsd}
Keep the notation and assumption, $X_{\tri, J-\dR}(\overline{r}_L, \ul{k}_J^{\fu})$ is smooth of dimension
  \begin{equation*}4+3d_L-2|J\cap \Sigma^+(\delta)|-3|J\cap \Sigma^-(\delta)|
  \end{equation*}
  at the point $x$.
\end{corollary}
\begin{proof}By Lemma \ref{lem: cclg-grt}, $\dim \widehat{Z}_{J,y}^{\fw}=2|\Sigma^+(\delta)\cap J|+|\Sigma^-(\delta)\cap J|+4|\Sigma_L\setminus J|$. We know the second morphism in (\ref{equ: cclg-locmod1}) is formally smooth of relative dimension $4d_L$, the third morphism is formally smooth of relative dimension $4$, and by \cite[Cor. 3.5.8]{BHS3}, the fourth morphism is of relative dimension $3d_L$. Putting these together, we have
\begin{equation*}
  \dim \widehat{X}_{\tri, J-\dR}(\overline{r}_L, \ul{k}_J^{\fu})_x=4+3d_L-2|J\cap \Sigma^+(\delta)|-3|J\cap \Sigma^-(\delta)|.
\end{equation*}
By the same argument as in the proof of \cite[Cor. 3.6.3]{BHS3}, we have
\begin{equation*}
  \dim_E X_{r,\cM_{\bullet}, J-\dR}^{\fw}(E[\epsilon]/\epsilon^2)=4-d_L+\dim_E T_{Z_J^{\fw},y}.
\end{equation*}
By \cite[Prop. 2.5.3]{BHS3} and Lemma \ref{lem: cclg-grt}, we know $Z_J^{\fw}$ is smooth of dimension $2|\Sigma^+(\delta)\cap J|+|\Sigma^-(\delta)\cap J|+4|\Sigma_L\setminus J|$. Thus
\begin{equation*}
  \dim_E X_{r,\cM_{\bullet}, J-\dR}^{\fw}(E[\epsilon]/\epsilon^2)=\dim \widehat{X}_{\tri, J-\dR}(\overline{r}_L, \ul{k}_J^{\fu})_x.
\end{equation*}
The corollary follows.
\end{proof}
Let $\omega$ denote the natural morphism $X_{\tri}(\overline{r}_L) \ra \cT_L$, which induces a morphism
\begin{equation}\label{equ: cclg-wtm3}
  \omega: X_{\tri, J-\dR}(\overline{r}_L, \ul{k}_J^{\fu}) \lra \cT_L(\ul{k}_J^{\fu}).
\end{equation}
Denote by $\widehat{\ft}_{\Sigma_L}$ (resp. $\widehat{\ft}_{\Sigma_L\setminus J}$) the completion of $\ft_{\Sigma_L}$ (resp. of $\ft_{\Sigma_L\setminus J}$) at $0$, $\widehat{\cT}_{L,\delta}$ (resp. $\widehat{\cT_L}(\ul{k}_J^{\fu})_{\delta}$) the completion of $\cT_L$ (resp. $\cT_L(\ul{k}_J^{\fu})$) at $\delta$. We have natural morphisms $\widehat{\cT}_{L,\delta} \ra \widehat{\ft}_{\Sigma_L}$ and $\widehat{\cT_L}(\ul{k}_J^{\fu})_{\delta} \ra \widehat{\ft}_{\Sigma_L\setminus J}$. Moreover, we have an isomorphism (where $\ft_{\Sigma_L\setminus J}$ is viewed as the fiber  of $\ft_{\Sigma_L}$ at $0\in \ft_{\Sigma_J}$)
\begin{equation*}
  \widehat{\cT}_{L,\delta}\times_{\widehat{\ft}_{\Sigma_L}} \widehat{\ft}_{\Sigma_L\setminus J} \cong \widehat{\cT_L}(\ul{k}_J^{\fu})_{\delta}.
\end{equation*}
Recall (see \cite[(3.3)]{BHS3}) that we have a natural morphism $X_{W, \cF_{\bullet}} \lra \widehat{\ft}_{\Sigma_L}$. It is easy to see this morphism induces $X_{W,\cF_{\bullet},J-\dR} \lra \widehat{\ft}_{\Sigma_L\setminus J}$. Thus we can deduce from \cite[Thm. 3.4.4]{BHS3} a formally smooth morphism
\begin{equation*}
  X_{\cM, \cM_{\bullet},J-\dR} \lra \widehat{\cT}_{L,\delta} \times_{\widehat{\ft}_{\Sigma_L}} X_{W, \cF_{\bullet},J-\dR} \cong \widehat{\cT_L}(\ul{k}_J^{\fu})_{\delta} \times_{\widehat{\ft}_{\Sigma_L\setminus J}} X_{X,\cF_{\bullet}, J-\dR}.
\end{equation*}
\begin{proposition}The rigid space  $X_{\tri, J-\dR}(\overline{r}_L, \ul{k}_J^{\fu})$ is flat over $ \cT_L(\ul{k}_J^{\fu})$ in a neighborhood of $x$.
\end{proposition}
\begin{proof}Recall for $\sigma\in \Sigma_L$, we have a morphism $\kappa_{1,\sigma}: X_{\sigma} \ra \ft_{\sigma}$, $(g_1 B, g_2 B, \psi)\mapsto \overline{\Ad(g_1^{-1}) \psi} \in \ft_{\sigma}$. By \cite[Prop. 2.3.3]{BHS3}, we see that the morphism
\begin{equation*}
  \kappa_{1,\fw}^J: Z_J^{\fw} \lra \prod_{\sigma\in \Sigma_L\setminus J} X_{\sigma, \fw_{\sigma}} \xlongrightarrow{(\kappa_{1,\sigma})_{\sigma\in\Sigma_L\setminus J}} \prod_{\sigma \in \Sigma_L \setminus J} \ft_{\sigma}
\end{equation*}
is flat (and is also flat after completion). The proposition follows then by an easy variation of the proof of \cite[Prop. 4.1.1]{BHS3} (using the discussion that precedes the proposition and Theorem \ref{thm: cclg-locmod}).
\end{proof}
\begin{proposition}\label{prop: cclg-pdacc}Suppose $\fu=1$  (i.e. $J\subseteq \Sigma^+(\delta)$)  and $J\neq \Sigma_L$. Then the set $Z$ of points $z=(r_z, \delta_z)\in X_{\tri,J-\dR}(\overline{r}_L, \ul{k}_J)$ satisfying
  \begin{itemize}\item $\delta_z$ is locally algebraic, very regular and $\Sigma^+(\delta_z)=\Sigma_L$,
    \item $r_z$ is de Rham, non-$\Sigma_L\setminus J$ -critical (i.e. $\Sigma(z)\cap (\Sigma_L \setminus J)=\emptyset$, see \S~\ref{sec: cclg-xdt} for $\Sigma(z)$),
  \end{itemize}  accumulates at $x=(r,\delta)$. If $\delta$ is moreover spherically algebraic, then the points $z=(r_z,\delta_z)\in Z$ with $\delta_z$ spherically algebraic also accumulate at $x$.
\end{proposition}
\begin{proof}By the flatness of $\omega$, there exists an affinoid neighborhood $U$ of $x$ in $X_{\tri,J-\dR}(\overline{r}_L, \ul{k}_J)$ such that $\omega(U)$ is open and isomorphic to a finite union of open affinoids in $\cT_L(\ul{k}_J)$. We have thus
  \begin{equation}\label{equ: cclg-clatri}
    C_1:=\sup_{z\in U} \val_L(\delta_{1,z}(\varpi_L))<+\infty.
  \end{equation}
Let $\cT_L^0$ denote the rigid space over $E$ parameterizing continuous character of $T(\co_L)$, and $\cT_L^0(\ul{k}_J)$ the closed subspace of $\cT_L^0$ of continuous characters $\delta^0$ with $\wt(\delta^0)_{\sigma, i}=k_{i,\sigma}$ for $\sigma\in J$ and $i=1,2$.
We have $\cT_L(\ul{k}_J)\cong \cT_L^0(\ul{k}_J)\times \bG_m^2$, and hence the composition $\omega_0: X_{\tri,J-\dR}(\overline{r}_L,\ul{k}_J) \xrightarrow{\text{(\ref{equ: cclg-wtm3})}} \cT_L(\underline{k}_J)\ra  \cT_L^0(\ul{k}_J)$ is also flat. Shrinking $U$, we assume $\omega_0(U)$ is open in $\cT_L^0(\ul{k}_J)$. Let $C_2>C_1$, $C_2\geq 1$ and let $\cZ$  denote the set of  points $\delta_0\in \omega_0(U)$ (which we view as a character of $T(\co_L)$) satisfying that
\begin{itemize}
  \item $\delta_0$  is locally algebraic,
  \item $\wt(\delta_0)_{1,\sigma}-\wt(\delta_0)_{2,\sigma} \geq C_2$ for all $\sigma\in \Sigma_L\setminus J$ (note for $\sigma\in J$, we have $\wt(\delta_0)_{i,\sigma}=k_{i,\sigma}$, and thus $\wt(\delta_0)_{1,\sigma}>\wt(\delta_0)_{2,\sigma}$).
\end{itemize}
Thus the set $\cZ$ accumulates at $\delta|_{T(\co_L)}\in U$. Shrinking $U$, we can assume $\cZ$ is Zariski-dense in $\omega_0(U)$. Since $\omega_0$ is flat, we deduce that $\omega_0^{-1}(\cZ)$ is Zariski-dense in $U$.

Let $(r_z,\delta_z)\in \omega_0^{-1}(\cZ)$, thus
\begin{equation*}
  \delta_z':=\delta_z \prod_{\sigma\in \Sigma(z)}( \sigma^{\mathrm{wt}(\delta_z)_{2,\sigma}-\mathrm{wt}(\delta_z)_{1,\sigma}} \otimes\sigma^{\wt(\delta_z)_{1,\sigma}-\wt(\delta_z)_{2,\sigma}})
\end{equation*}
is a trianguline parameter of $r_z$. By Kedlaya's slope filtration theory \cite[Thm. 1.7.1]{Ked10}, we have
\begin{equation*}
  \val_L((\delta_z)'_1(\varpi_L))=\val_L(\delta_{z,1}(\varpi_L))+\sum_{\sigma\in \Sigma(z)} (\wt(\delta_z)_{2,\sigma}-\wt(\delta_z)_{1,\sigma}) \geq 0.
\end{equation*}
By our assumption on the points in $\cZ$ and (\ref{equ: cclg-clatri}), we easily deduce $\Sigma(z)\cap (\Sigma_L\setminus J)=\emptyset$. By \cite[Prop. A.3]{Ding4}, $r_z$ is thus  $\Sigma_L\setminus J$-de Rham (hence is de Rham, since $r_z$ is $J$-de Rham \emph{a priori}). Finally,  by the same argument as in the last paragraph of the proof of Theorem \ref{thm: cclg-try} (2), if $C_2$ is big enough, we have that $\delta_z$ is very regular. The first part of the proposition follows. If $\delta$ is moreover spherically algebraic,  by shrinking $U$ if necessary, the subset $\cZ_0$ of $\cZ$ of algebraic characters is also Zariski-dense in $\omega_0(U)$. The second part follows by the same argument  replacing $\cZ$ by $\cZ_0$.
\end{proof}
\begin{remark}\label{rem: cclg-pdacc}(1) If we only assume $x=(r,\delta)\in X_{\tri}^{\square}(\overline{r}_L, \ul{k}_J)$ (with $J\neq \Sigma_L$, $\delta$ locally  algebraic, very regular, and $\wt(\delta)_{1,\sigma}\neq \wt(\delta)_{2,\sigma}$ for all $\sigma\in \Sigma_L$). By Proposition \ref{prop: cclg-pdacc}  and the fact that $X_{\tri,J-\dR}^{\square}(\overline{r}_L, \ul{k}_J)$ is Zariski-closed in $X_{\tri}^{\square}(\overline{r}_L)$,  the followings are actually equivalent:
\begin{itemize}
  \item $x\in X_{\tri,J-\dR}^{\square}(\overline{r}_L, \ul{k}_J)$,
  \item there exists a set $Z$ of points satisfying the properties in Proposition \ref{prop: cclg-pdacc} (note that such points have fixed weights for embeddings in $J$) such that $Z$ accumulates at $x$.
\end{itemize}

(2) The proposition should also be true when $J=\Sigma_L$. Indeed if $J=\Sigma_L$ and $r$ is moreover crystalline (i.e. $\delta$ is spherically algebraic), then it follows from results in \cite[\S~2.2]{BHS2}:  Let $\widetilde{\fX}_{\overline{r}_L}^{\square, \ul{k}_{\Sigma_L}-\crr}$ be the refined crystalline deformation space as in \cite[\S~2.2]{BHS2}. There exists a natural closed embedding (cf. \cite[(2.5)]{BHS2})
\begin{equation*}
 \widetilde{\fX}_{\overline{r}_L}^{\square, \ul{k}_{\Sigma_L}-\crr} \hooklongrightarrow X_{\tri}^{\square}(\overline{r}_L),
\end{equation*}
which obviously factors through $X_{\tri, \Sigma_L-\dR}^{\square}(\overline{r}_L, \ul{k}_{\Sigma_L})$. The closed embedding $$\widetilde{\fX}_{\overline{r}_L}^{\square, \ul{k}_{\Sigma_L}-\crr} \hooklongrightarrow X_{\tri, \Sigma_L-\dR}^{\square}(\overline{r}_L,\ul{k}_{\Sigma_L})$$ is actually a local isomorphism at $x$ since both of the two rigid spaces are smooth at $x$ (see \cite[Lem. 2.4]{BHS2}, Corollary \ref{cor: cclg-pjsd}), and have the same dimension $4+d_L$. The proposition in this case then easily follows from \cite[Lem. 2.4]{BHS2}.
\end{remark}
\subsection{Partial classicality} 
We use the notation of \S~\ref{sec: cclg-3} and suppose Hypothesis \ref{hypo: cclg-TW}. The main result of this section is
\begin{theorem} \label{thm: cclg-pcl}Let $x=(y,\delta)$ be a spherical, very regular point of $X_p(\overline{\rho}, \ul{\lambda}_J)'$ (cf. \S~ \ref{sec: cclg-3.3.2}) with $\mathrm{wt}(\delta)_{1,\sigma}\neq \mathrm{wt}(\delta)_{2,\sigma}-1$ for all $\sigma\in \Sigma_p$. The following statements are equivalent.
\begin{enumerate}
  \item $x\in X_p(\overline{\rho}, \ul{\lambda}_J)$,
  \item $\rho_{x,\widetilde{v}}$ is $J_{\widetilde{v}}$-de Rham for all $v|p$.
\end{enumerate}
\end{theorem}
\begin{proof}
``(1) $\Rightarrow$ (2)" follows from (\ref{equ: cclg-adb}). We prove ``(2) $\Rightarrow$ (1)". Let $X$ be an irreducible component of $X_p(\overline{\rho})$ containing $x$, which thus has the form
\begin{equation*}
  X=X^p \times \bU^g \times \prod_{v|p} \iota_{\widetilde{v}}^{-1}(X_{\widetilde{v}})
\end{equation*}
where $X_{\widetilde{v}}$ is the \emph{unique} irreducible component of $X_{\tri}^{\square}(\overline{\rho}_{\widetilde{v}})$ containing $x_{\widetilde{v}}$ (by \cite[Cor. 3.7.10]{BHS3}). Let  $\cZ_p$ be the set  of points $(z_{\widetilde{v}})_{v\in S_p} \in \prod_{v|p} X_{\widetilde{v}}$ satisfying for $v|p$,
\begin{itemize}
  \item $z_{\widetilde{v}}=(\rho_{z_{\widetilde{v}}}, \delta_{z_{\widetilde{v}}})\in X_{\tri, J-\dR}^{\square}(\overline{\rho}_{\widetilde{v}}, \ul{\lambda}_{J_{\widetilde{v}}}^{\flat})$,
  \item $\delta_{z_{\widetilde{v}}}$ is spherically algebraic, very regular and $\Sigma^+(\delta_{z_{\widetilde{v}}})=\Sigma_{\widetilde{v}}$,
  \item $\rho_{z_{\widetilde{v}}}$ is crystalline and $z_{\widetilde{v}}$ is non-$\Sigma_{\widetilde{v}}\setminus J_{\widetilde{v}}$-critical.
\end{itemize}
By Proposition \ref{prop: cclg-pdacc} (and Remark \ref{rem: cclg-pdacc} (2) for the case $J_{\widetilde{v}}=\Sigma_{\widetilde{v}}$), $\cZ_p$ accumulates at the point $(x_{\widetilde{v}})_{v|p}\in \prod_{v|p} X_{\widetilde{v},J_{\widetilde{v}}-\dR}(\ul{\lambda}_{J_{\widetilde{v}}}^{\flat})$ (where $X_{\widetilde{v},J_{\widetilde{v}}-\dR}(\ul{\lambda}_{J_{\widetilde{v}}}^{\flat}):=X_{\widetilde{v}}\times_{X_{\tri}^{\square}(\overline{\rho}_{\widetilde{v}})} X_{\tri, J_{\widetilde{v}}-\dR}^{\square}(\overline{\rho}_{\widetilde{v}}, \ul{\lambda}_{J_{\widetilde{v}}}^{\flat})$. By \cite[Thm. 3.9, Cor. 3.12]{BHS2} (although the statement of \emph{loc. cit.} is for the eigenvariety, the proof actually shows that the same holds for the  eigenvariety), any point $z\in X$ with $(z_{\widetilde{v}})_{v|p}\in \cZ_p$ is classical. Hence the Zariski closure of $X^p \times \bU^g \times  \iota_p^{-1} (\cZ_p)$ is contained in $X_p(\overline{\rho}, \ul{\lambda}_J)$ and contains $x$, (1) follows.
\end{proof}
\begin{corollary}\label{cor: cclg-pcl}.
Let $z=(\fm_{\rho}, \delta)\in \cE(U^p, \ul{\lambda}_J)'_{\overline{\rho}}$ (cf. \S~\ref{sec: cclg-3.2}) such that $\delta$ is spherically algebraic, very regular, and $\wt(\delta)_{1,\sigma}\neq \wt(\delta)_{2,\sigma}-1$ for all $\sigma\in \Sigma_p$. Then the followings are equivalent
    \begin{enumerate}
    \item $z \in \cE(U^p, \ul{\lambda}_J)_{\overline{\rho}}$,
    \item $\rho_{z, \widetilde{v}}$ is $J_{\widetilde{v}}$-de Rham for all $v|p$.
  \end{enumerate}
\end{corollary}
\begin{proof}
  The direction ``(1) $\Rightarrow$ (2)" follows from Theorem \ref{thm: cclg-pen}.

(2) $\Rightarrow$ (1): We can (and do) view $z$ as a point in $X_p(\overline{\rho})$ (e.g. see the arguments above Lemma \ref{lem: cclg-joa}), and hence a point in $X_p(\overline{\rho}, \ul{\lambda}_J)'$ (since $z\in \cE(U^p, \ul{\lambda}_J)'_{\overline{\rho}}$). By Theorem \ref{thm: cclg-pcl}, (3) implies  $z\in X_p(\overline{\rho}, \ul{\lambda}_J)$ which is equivalent to that
  \begin{equation*}
J_{B_p}\big(\widehat{S}(U^p,E)^{\an}(\ul{\lambda}_J)[\fm_{\rho}]\big)    \cong J_{B_p}\big(\Pi_{\infty}^{R_{\infty}-\an}(\ul{\lambda}_J)[\fm_{\rho}]\big)\neq 0.
  \end{equation*}
Hence $z\in \cE(U^p, \ul{\lambda}_J)_{\overline{\rho}}$.
\end{proof}
\begin{remark}
  This corollary gives support for \cite[Conj. A.9]{Ding4}.
\end{remark}
\begin{corollary}\label{cor: cclg-soctri}Suppose Hypothesis \ref{hypo: cclg-TW}, then  Conjecture \ref{conj: cclg-comptri} (hence Conjecture \ref{conj: cclg-soctri}, see Lemma \ref{conj: cclg-equivtri}) is true.
\end{corollary}
\begin{proof}We use the notation of Conjecture \ref{conj: cclg-comptri}. By Corollary \ref{cor: cclg-net}, there exists $\fw\in \cS_2^{|S_p^+(\rho)|}$ (see the discussion above Lemma \ref{cclg-triptA}) such that $(z_{\fw})_J^c:=(\fm_{\rho}, (\delta_{\fw})_J^c)\in \cE(U^p)_{\overline{\rho}}$ for $J\subseteq \Sigma(\fw)$ (see Corollary \ref{cor: cclg-net} for $\Sigma(\fw)$). Note  that $(z_{\fw})_J^c$ actually lies in $\cE(U^p, \wt(\delta_{\fw})_{\Sigma_L\setminus J})'_{\overline{\rho}}$. By Corollary \ref{cor: cclg-pcl} (applied to $(z_{\fw})_J^c$), if $\Sigma_p\setminus J \subseteq C(\rho)$, then $(z_{\fw})_J^c\in \cE(U^p, \wt(\delta_{\fw})_{\Sigma_p\setminus J})_{\overline{\rho}}$ \big(in summary, we have proved Conjecture \ref{conj: cclg-comptri} for the refinement $\delta_{\fw}$\big).

We need to find the points for other refinements. By Proposition \ref{prop: cclg-pve} (see also (\ref{equ: cclg-spE})), we have
\begin{multline}\label{equ: cclg-Adj}
  \Hom_{T_p}\big((\delta_{\fw})_J^c, J_{B_p}\big(\widehat{S}(U^p,E)^{\an}(\wt(\delta_{\fw})_{\Sigma_p\setminus J})\big)[\fm_{\rho}]\big)\big) \\
\cong \Hom_{G_p}\big( I((\delta_{\fw})_J^c\delta_{B_p}^{-1}), \widehat{S}(U^p,E)^{\an}(\wt(\delta_{\fw})_{\Sigma_p\setminus J})[\fm_{\rho}]\big).
\end{multline}
Suppose $J_{\widetilde{v}}=\emptyset$ for $v\in S_p^+(\rho)$ (recall $\rho_{\widetilde{v}}$ is crystalline for such $v$), and consider
\begin{equation*}
  I((\delta_{\fw})_J^c\delta_{B_p}^{-1}) \cong \big(\otimes_{v\in S_p^+(\rho)} I((\delta_{\fw})_{\widetilde{v}}\delta_{B_{\widetilde{v}}}^{-1})\big) \otimes_E \big(\otimes_{v\in S_p\setminus S_p^+(\rho)} I((\delta_{\fw})_{\widetilde{v}, J_{\widetilde{v}}}^c\delta_{B_{\widetilde{v}}}^{-1}) \big).
\end{equation*}
Note that $I((\delta_{\fw})_{\widetilde{v}}\delta_{B_{\widetilde{v}}}^{-1})$ is locally algebraic (since $\delta_{\fw}$ is dominant), and its associated Jacquet-Emerton module gives two refinement of $\rho_{\widetilde{v}}$. Since $(z_{\fw})_J^c\in \cE(U^p, \wt(\delta_{\fw})_{\Sigma_p\setminus J})_{\overline{\rho}}$, by (\ref{equ: cclg-Adj}), there exists an injection
\begin{equation*}
 I((\delta_{\fw})_J^c\delta_{B_p}^{-1})  \hooklongrightarrow \widehat{S}(U^p,E)^{\an}_{\overline{\rho}}[\fm_{\rho}].
\end{equation*}
Applying the Jacquet-Emerton functor, we obtain $2^{|S_p^+(\rho)|}$-points $(z_w)_J^c:=(\fm_{\rho}, (\delta_w)_J^c)\in \cE(U^p)_{\overline{\rho}}$ for all $w\in \cS_2^{|S_p^+(\rho)|}$. By Corollary \ref{cor: cclg-net}, we get then  $(z_w)_J^c\in \cE(U^p)_{\overline{\rho}}$ for  any $J\subseteq \Sigma(w)$ (where $\Sigma(w)$ is defined in the same way as $\Sigma(\fw)$). Then by the argument in the first paragraph with $\fw$ replaced by $w$, we have
\begin{equation*}(z_w)_J^c\in \cE(U^p, \wt(\delta_w)_{\Sigma_p\setminus J})_{\overline{\rho}}.
\end{equation*}
This concludes the proof (note that the ``only if" part of Conjecture \ref{conj: cclg-comptri} (1) follows from Lemma \ref{cclg-triptA}).
\end{proof}

\appendix
\section{Partially de Rham $B$-pairs}\label{sec: cclg-A.1}Recall some results on $B$-pairs.
Let $A$ be an artinian local $E$-algebra, recall (cf. \cite[\S~2]{Ber08}, \cite[Def. 2.11]{Na2} and \cite[Def. 1.3]{Ding4})
\begin{definition}
  (1) An $A$-$B$-pair $W$ is a pair $(W_e,W_{\dR}^+)$ where $W_e$ is a free $B_e\otimes_{\Q_p} A$-module, and $W_{\dR}^+$ is a $\Gal_L$-invariant $B_{\dR}^+ \otimes_{\Q_p} A$-lattice in $W_{\dR}:=W_e \otimes_{B_e} B_{\dR}$. The rank of $W$ is defined to be the rank of $W_e$ over $B_e\otimes_{\Q_p} A$.

  (2) An $A$-$B$-pair $W$ is called $J$-de Rham if $W_{\dR,\sigma}^{\Gal_L}$ is a free $A$-module of rank $\rk W$ for all $\sigma\in J$ \big(where $W_{\dR}\cong \oplus_{\sigma\in \Sigma_L} W_{\dR,\sigma}$ with respect to the isomorphism $B_{\dR}\otimes_{\Q_p} A\cong \oplus_{\sigma \in \Sigma_L} B_{\dR,\sigma}$, $B_{\dR,\sigma}:=B_{\dR}\otimes_{L,\sigma} A$\big).
\end{definition}
\begin{example}\label{ex: cclg-bpair}
  (1) Let $V$ be a continuous representation of $\Gal_L$ over $A$, one can associate to $V$ an $A$-$B$-pair: $W(V)=(B_e\otimes_{\Q_p} V, B_{\dR}^+ \otimes_{\Q_p} V)$. For $J\subseteq \Sigma_L$, $W(V)$ is $J$-de Rham if and only if $V$ is $J$-de Rham.

  (2) Let $\delta: L^{\times} \ra A^{\times}$ be a continuous character, as in \cite[\S~2.1.2]{Na2}, one can associate to $\delta$ a rank $1$ $A$-$B$-pair denoted by $B_A(\delta)$. In fact, by \cite[Prop. 2.16]{Na2}, for any rank one $A$-$B$-pair $W$, there exists $\delta: L^{\times} \ra A^{\times}$ such that $W\cong B_A(\delta)$.

  (3) Let $\delta: L^{\times} \ra E^{\times}$ be a continuous character, $J\subset \Sigma_L$, by \cite[Lem. A.1]{Ding4}, $B_E(\delta)$ is $J$-de Rham if and only if $\wt(\delta)_{\sigma}\in \Z$ for all $\sigma\in J$.

  (4) Let $\delta: L^{\times} \ra (E[\epsilon]/\epsilon^2)^{\times}$ be a continuous character, $J\subset \Sigma_L$, by \cite[Lem. 1.15]{Ding4}, $B_{E[\epsilon]/\epsilon^2}(\delta)$ is $J$-de Rham if and only if $\wt(\delta)_{\sigma}\in \Z$ for all $\sigma\in J$. Let $\overline{\delta}:=\delta \pmod{\epsilon}: L^{\times} \ra E^{\times}$, thus for any $\sigma \in \Sigma_L$, there exists $d_{\sigma}\in E$ such that $\wt(\delta)_{\sigma}=\wt(\overline{\delta})_{\sigma}+d_{\sigma} \epsilon$, so $B_{E[\epsilon]/\epsilon^2}(\delta)$ is $J$-de Rham if and only if $B_E(\overline{\delta})$ is $J$-de Rham and $d_{\sigma}=0$ for all $\sigma\in J$.
\end{example}
For an $A$-$B$-pair $W$, denote by $C^{\bullet}(W)$ the $\Gal_L$-complex:  $[W_e\oplus W_{\dR}^+ \xrightarrow{(x,y)\mapsto x-y} W_{\dR}]$. Following \cite[App.]{Na2}, let $H^i(\Gal_L, W):=H^i(\Gal_L, C^{\bullet}(W))$. By definition, one has an exact sequence
\begin{multline*}
  0 \ra H^0(\Gal_L, W) \ra W_e^{\Gal_L} \oplus (W_{\dR}^+)^{\Gal_L} \ra W_{\dR}^{\Gal_L}\\ \ra H^1(\Gal_L,W) \ra H^1(\Gal_L, W_e) \oplus H^1(\Gal_L, W_{\dR}^+) \ra H^1(\Gal_L, W_{\dR}).
\end{multline*}
One can (and does) identify $H^1(\Gal_L,W)$ with the group of extensions of $A$-$B$-pairs $\Ext^1(B_A, W)$. If $W\cong W(V)$ for some continuous representation $V$ of $\Gal_L$ over $A$, then we have canonical isomorphisms $H^i(\Gal_L, W)\cong H^i(\Gal_L, W(V))$.
For $J\subseteq \Sigma_L$, consider the following morphism $C^{\bullet}(W)\ra [W_e\ra 0]\ra [W_{\dR}\ra 0]\ra [W_{\dR,J}\ra 0]$ where $W_{\dR,J}:=\oplus_{\sigma\in J} W_{\dR,\sigma}$, and the last map is the natural projection. As in \cite[\S~1.2]{Ding4}, put $H^1_{g,J}(\Gal_L,W):=\Ker\big(H^1(\Gal_L, W)\ra H^1(\Gal_L, W_{\dR,J})\big)$. it is obviously that $H^1_{g,J_1}(\Gal_L,W)\cap H^1_{g,J_2}(\Gal_L,W)=H^1_{g,J_1\cup J_2}(\Gal_L,W)$. Moreover, if $W$ is $J$-de Rham, for $[W']\in H^1(\Gal_L, W)$, $[W']\in H^1_{g,J}(\Gal_L, W)$ if and only if $W'$ is $J$-de Rham. Note that a morphism of $E$-$B$-pairs: $f: W\ra W'$ induces a natural morphism $H^1_{g,J}(\Gal_L, W)\ra H^1_{g,J}(\Gal_L,W')$.

For an $E$-$B$-pair $W$, an algebraic character $\delta=\prod_{\sigma\in \Sigma_L} \sigma^{\wt(\delta)_{\sigma}}$ of $L^{\times}$ over $E$, denote by $W(\delta):=W \otimes B_E(\delta)$ \big(recall the tensor product $W_1\otimes W_2$ of $E$-$B$-pairs is defined to be the pair $(W_{1,e}\otimes_{B_e} W_{2,e}, W_{1,\dR}^+ \otimes_{B_{\dR}^+} W_{2,\dR}^+)$\big). One has in fact
\begin{equation*}
  W(\delta)_e\cong W_e, \ W(\delta)_{\dR,\sigma}^+\cong t^{\wt(\delta)_{\sigma}} W_{\dR,\sigma}^+, \ \forall \sigma\in \Sigma_L.
\end{equation*}
For algebraic characters $\delta_1$, $\delta_2$ with $\wt(\delta_1)\geq \wt(\delta_2)$, one has a natural morphism $i=(i_e,i_{\dR}^+): W(\delta_1) \ra W(\delta_2)$ with $i_e: W(\delta_1)_e\xrightarrow{\sim} W(\delta_2)_e$, and $i_{\dR}^+: W(\delta_1)_{\dR}^+ \hookrightarrow W(\delta_2)_{\dR}^+$ the natural injection. One gets thus an exact sequence of $\Gal_L$-complexes:
\begin{multline}\label{equ: cclg-s1R}
  0\lra [W(\delta_1)_e \oplus W(\delta_1)_{\dR}^+\ra W(\delta_1)_{\dR}] \lra [W(\delta_2)_e \oplus W(\delta_2)_{\dR}^+\ra W(\delta_2)_{\dR}] \\ \lra [\oplus_{\sigma\in \Sigma_L} W(\delta_2)_{\dR,\sigma}^+/t^{\wt(\delta_1)_{\sigma}-\wt(\delta_2)_{\sigma}}\ra 0] \ra 0.
\end{multline}
Thus $H^0(\Gal_L, W(\delta_1))\xrightarrow{\sim} H^0(\Gal_L, W(\delta_2))$ if $H^0(\Gal_L, W(\delta_2)_{\dR,\sigma}^+/t^{\wt(\delta_1)_{\sigma}-\wt(\delta_2)_{\sigma}})=0$ for all $\sigma\in J$.
 Suppose $W$ is $J$-de Rham, let $\delta=\prod_{\sigma\in J} \sigma^{k_{\sigma}}$ be an algebraic character of $L^{\times}$ with $k_{\sigma}\in \Z_{\geq 0}$ such that $H^0(\Gal_L, W(\delta)_{\dR,\sigma}^+)=0$ for all $\sigma\in J$, in other words, the generalized Hodge-Tate weights of $W(\delta)$ are negative for $\sigma\in J$ (which holds when $\wt(\delta)_{\sigma}$ for all $\sigma\in J$ are sufficiently large). Put \begin{equation}\widetilde{H}^2_{J}(\Gal_L, W):=H^2(\Gal_L, W(\delta)),\end{equation} which is in fact independent of the choice of $\delta$.
 Indeed, for $\delta_1$, $\delta_2$ algebraic characters of $L^{\times}$ satisfying the above assumptions (for $\delta$), suppose $\wt(\delta_1)_{\sigma}\geq \wt(\delta_2)_{\sigma}$ for all $\sigma\in J$ (the general case can be easily reduced to this case), the Tate dual of the natural map $H^2(\Gal_L, W(\delta_1))\ra H^2(\Gal_L,W(\delta_2))$ is thus $H^0(\Gal_L, W^{\vee}(\delta_2^{-1}\chi_{\cyc}))\hookrightarrow H^0(\Gal_L, W^{\vee}(\delta_1^{-1}\chi_{\cyc}))$, which is in fact bijective by (\ref{equ: cclg-s1R}): by the assumption $H^0(\Gal_L, W(\delta_i)_{\dR,\sigma}^+)=0$ for all $\sigma\in J$, $i=1,2$, we have
 \begin{equation*}
   H^0\big(\Gal_L, (W^{\vee}(\delta_1^{-1}\chi_{\cyc}))^+_{\dR,\sigma}/t^{k_{1,\sigma}-k_{2,\sigma}}\big)=0
 \end{equation*}
for all $\sigma\in J$. Note also that one has a natural projection $\widetilde{H}^2_{J}(\Gal_L,W)\twoheadrightarrow H^2(\Gal_L,W)$.

\begin{proposition}\label{prop: cclg-wrn}Let $W$ be a $J$-de Rham $E$-$B$-pair, then
  \begin{multline}\label{prop: cclg-fjd}
  \dim_{E}H^1_{g,J}(\mathrm{Gal}_L,W)\\ =[L:\Q_p]\rk W+\dim_E H^0(\Gal_L, W) + \dim_E \widetilde{H}^2_J(\Gal_L,W)
  -\dim_E H^0(\Gal_L, W_{\dR,J}^+).
  \end{multline}
\end{proposition}
\begin{proof}
  We use the notation of \cite[(1.7)]{Ding4}. By \emph{loc. cit.} we have By the exact sequence
  \begin{equation*}
    \dim_{E}H^1_{g,J}(\Gal_L,W)=\dim_E H^1(\Gal_L,W(\delta))+\dim_E H^0(\Gal_L, W)-\dim_E H^0(\Gal_L, W_{\dR,J}^+).
  \end{equation*}
However, by our assumption on $\delta$, we have $H^0(\Gal_L,W(\delta))=0$, $H^2(\Gal_L,W(\delta))\cong\widetilde{H}^2_J(\Gal_L,W)$, and hence $\dim_E H^1(\Gal_L,W(\delta))=[L:\Q_p]\rk W+\dim_E \widetilde{H}^2_J(\Gal_L,W)$. The proposition follows.
\end{proof}
\begin{corollary}\label{cor: cclg-pdR}
  Let $W$ be a $J$-de Rham $E$-$B$-pair, if $\widetilde{H}^2(\Gal_L,W)=H^2(\Gal_L,W)$, then we have
  \begin{equation*}\dim_E H^1_{g,J}(\Gal_L,W)=\dim_E H^1(\Gal_L,W)-\dim_E H^0(\Gal_L, W_{\dR,J}^+).\end{equation*}
\end{corollary}
\begin{proposition}\label{prop: cclg-exbp}Given an exact sequence of $E$-$B$-pairs
\begin{equation}\label{equ: cclg-a3w}
  0 \ra W_1 \ra W_2 \ra W_3\ra 0,
\end{equation}
suppose $W_i$ are $J$-de Rham for all $i=1,2,3$ and $\widetilde{H}^2_J(\Gal_L,W_1)=0$ (which implies in particular $H^2(\Gal_L, W_1)=0$), then (\ref{equ: cclg-a3w}) induces a long exact sequence
\begin{multline}\label{equ: cclg-a02}
  0 \ra H^0(\Gal_L, W_1) \ra H^0(\Gal_L, W_2) \ra H^0(\Gal_L,W_3) \ra H^1_{g,J}(\Gal_L,W_1) \\ \ra H^1_{g,J}(\Gal_L,W_2) \ra H^1_{g,J}(\Gal_L,W_3) \ra 0.
\end{multline}
\end{proposition}
\begin{proof}
  Since $H^2(\Gal_L,W_1)=0$, (\ref{equ: cclg-a3w}) induces a long exact sequence
 \begin{multline*}
  0 \ra H^0(\Gal_L, W_1) \ra H^0(\Gal_L, W_2) \ra H^0(\Gal_L,W_3) \\ \ra H^1(\Gal_L,W_1) \ra H^1(\Gal_L,W_2) \ra H^1(\Gal_L,W_3) \ra 0.
\end{multline*}
Since $W_i$ are $J$-de Rham for all $i=1,2,3$, the exact sequence
\begin{equation*}
  0 \ra W_{1,\dR,J} \ra W_{2,\dR,J}\ra W_{3,\dR,J}\ra 0
\end{equation*}
induces an exact sequence
\begin{equation*}
  0 \ra H^1(\Gal_L, W_{1,\dR,J}) \ra H^1(\Gal_L, W_{2,\dR,J}) \ra H^1(\Gal_L, W_{3,\dR,J}) \ra 0;
\end{equation*}moreover, 
the following diagram commutes:
\begin{equation*}\footnotesize%
\begin{CD}
  H^0(\Gal_L,W_3)@>>>H^1(\Gal_L,W_1) @>>> H^1(\Gal_L,W_2) @>>> H^1(\Gal_L,W_3) @>>> 0 \\
  @VVV @VVV @VVV @VVV \\
  H^0(\Gal_L, W_{3,\dR,J}) @>>> H^1(\Gal_L,W_{1,\dR,J}) @>>> H^1(\Gal_L,W_{2,\dR,J}) @>>> H^1(\Gal_L,W_{3,\dR,J}) @>>> 0
\end{CD}.
\end{equation*}From which we get the exact sequence (\ref{equ: cclg-a02}) except the last map \big(note that the map $$H^0(\Gal_L, W_{3,\dR,J})\lra H^1(\Gal_L, W_{1,\dR,J})$$ equals zero, since $W_2$ is $J$-de Rham\big).
Thus it is sufficient to show the induced map $$H^1_{g,J}(\Gal_L,W_2) \lra H^1_{g,J}(\Gal_L,W_3)$$ is surjective and we prove it by dimension calculation. By Proposition \ref{prop: cclg-wrn}, \begin{multline*}\dim_{E}H^1_{g,J}(\Gal_L,W_i)=d_L \rk W_i+\dim_{E} \widetilde{H}^2_J(\Gal_L, W_i)+\dim_{E}H^0(\Gal_L,W_i)\\ -\dim_{E} H^0(\Gal_L, W_{i,\dR,J}^+).\end{multline*}
Since $\widetilde{H}^2_J(\Gal_L,W_1)=0$, we have $\widetilde{H}^2_J(\Gal_L,W_2)\cong \widetilde{H}^2_J(\Gal_L,W_3)$. Indeed, let $\delta$ be an algebraic character of $L^{\times}$ over $E$ with $\wt(\delta)_{\sigma}\in \Z_{\geq 0}$ for $\sigma\in J$ and $\wt(\delta)_{\sigma}=0$ for $\sigma \notin J$ satisfying that $H^0(\Gal_L, W_2(\delta)_{\dR,J}^+)=0$ \big(thus $H^0(\Gal_L, W_i(\delta)_{\dR,J}^+)=0$ for all $i=1,2,3$\big). One gets an isomorphism $H^2(\Gal_L, W_2(\delta))\cong H^2(\Gal_L,W_3(\delta))$ (hence the precedent isomorphism) from the exact sequence of $E$-$B$-pairs $0 \ra W_1(\delta) \ra W_2(\delta) \ra W_3(\delta) \ra 0$ (using $H^2(\Gal_L, W_1(\delta))=0$).
Since $W_i$ is $J$-de Rham for all $i$, $\dim_{E}H^0(\Gal_L,W_{1,\dR,J}^+) +\dim_{E}H^0(\Gal_L, W_{3,\dR,J}^+)=\dim_{E} H^0(\Gal_L, W_{2,\dR,J}^+)$. Combining the above calculation, we see
\begin{multline*}
  \dim_{E} H^0(\Gal_L,W_1)+\dim_{k(x)}H^0(\Gal_L,W_3)+\dim_{k(x)}H^1_{g,J}(\Gal_L,W_2)\\
  =\dim_{E} H^0(\Gal_L, W_2) + \dim_{k(x)}H^1_{g,J}(\Gal_L, W_1)+\dim_{k(x)} H^1_{g,J}(\Gal_L,W_3).
\end{multline*}
The proposition follows.
\end{proof}
\section{Some locally analytic representation theory}\label{sec: cclg-A.2}
We use the notation of \S~\ref{sec: cclg-3.1}. Let $V$ be a locally $\Q_p$-analytic representation of $G_p$ over $E$ and $J\subseteq \Sigma_p$, a vector $v\in V$ is called \emph{locally $J$-analytic}, if the induced action of $\ug_{\Sigma_p}$ on $v$ factors through $\ug_{J}$. We denote by $V^{J-\an}$ the subrepresentation of $V$ of locally $J$-analytic vectors and  $V$ is called \emph{locally $J$-analytic} if $V=V^{J-\an}$. A vector $v\in V$ is called \emph{$\text{U}(\ug_{J})$-finite} (or \emph{$J$-classical}) if the $E$-vector space $\text{U}(\ug_{J})v$ is finite dimensional, and $V$ is called \emph{$\text{U}(\ug_{J})$-finite} if all the vectors in $V$ are $\text{U}(\ug_{J})$-finite. In particular, $V$ is $\text{U}(\ug_{\Sigma_p\setminus J})$-finite if $V$ is locally $J$-analytic. As in \cite[Prop. 6.1.3]{Ding}, we have
\begin{proposition}\label{prop: cclg-cnfp}
  Let $V$ be a locally $\Q_p$-analytic representation of $G_p$ over $E$, $J\subseteq \Sigma_p$, $W$ an irreducible algebraic representation of $G_p$ over $E$ which is moreover locally $J$-analytic, then the following composition
  \begin{equation*}
    (V\otimes_E W)^{\Sigma_p \setminus J-\an}\otimes_E W'\lra V\otimes_E W\otimes_E W'\lra V, \ v\otimes w \otimes w'\mapsto w'(w)v,
  \end{equation*}
  is injective.
\end{proposition}
Let $J\subseteq \Sigma_p$, $\ul{\lambda}_J=(\lambda_{1,\sigma}, \lambda_{2,\sigma})_{\sigma \in J} \in \Z^{2|J|}$ be a dominant integral weight of $\ft_p$ (with respect to $B_p$, i.e. $\lambda_{1,\sigma}\geq \lambda_{2,\sigma}$), and $L(\ul{\lambda}_J)$ be the irreducible algebraic representation of $G_p$ with highest weight $\ul{\lambda}_J$. Let $V$ be a locally $\Q_p$-analytic representation of $G_p$ over $E$, put $V(\ul{\lambda}_J):=(V\otimes_E L(\ul{\lambda}_J)')^{\Sigma_p\setminus J-\an}\otimes_E L(\ul{\lambda}_J)$.
\begin{corollary}\label{cor: cclg-clsu}
  $V(\ul{\lambda}_J)$ is a subrepresentation of $V$. If $V$ is moreover admissible, then $V(\ul{\lambda}_J)$ is a closed admissible subrepresentation of $V$.
\end{corollary}
\begin{proof}
  The first statement follows directly from Proposition \ref{prop: cclg-cnfp}. If $V$ is admissible, so is $V\otimes_E L(\ul{\lambda}_J)' \otimes_E L(\ul{\lambda}_J)$. Since $V(\ul{\lambda}_J)$ is obviously a closed subrepresentation of $V\otimes_E L(\ul{\lambda}_J)'\otimes_E L(\ul{\lambda}_J)$, by \cite[Prop. 6.4]{ST03}, we see $V(\ul{\lambda}_J)$ is also admissible. By \emph{loc. cit.,} the map $V(\ul{\lambda}_J)\hookrightarrow V$ is strict and has closed image, which concludes the proof.
\end{proof}
We have moreover the following easy lemma.
\begin{lemma}\label{lem: cclg-j'js}
  (1) An morphism $V\ra W$ of locally $\Q_p$-analytic representations induces $V(\ul{\lambda}_J)\ra W(\ul{\lambda}_J)$.

  (2) Suppose there exists a locally $\Sigma_p \setminus J$-analytic representation $W$ such that $V\cong W\otimes_E L(\ul{\lambda}_J)$, then $V(\ul{\lambda}_J)\cong V$.

  (3) Let $J'\subseteq J$, then $V(\ul{\lambda}_J)$ is a subrepresentation of $V(\ul{\lambda}_{J'})$.
\end{lemma}
\begin{proof}
  (1) is obvious.

  For (2), it is sufficient to prove $(W\otimes_E L(\ul{\lambda}_J)\otimes_E L(\ul{\lambda}_J)')^{\Sigma_p\setminus J-an}\cong W$. Note taking locally $\Sigma_p \setminus J$-analytic vectors is the same as taking $\text{U}(\ug_J)$-invariant vectors. One has however $(W\otimes_E L(\ul{\lambda}_J)'\otimes_E L(\ul{\lambda}_J))^{\text{U}(\ug_J)}\cong W\otimes_E \End_E(L(\ul{\lambda}_J))^{\text{U}(\ug_J)}\cong W\otimes_E \End_{\ug_J}(L(\ul{\lambda}_J))\cong W$, where the first isomorphism is from the fact that $W$ is locally $\Sigma_p\setminus J$-analytic, and the last one from the irreducibility of $L(\ul{\lambda}_J)$ as a representation of $\ug_J$.

  For (3), one has
  \begin{multline*}
    V(\ul{\lambda}_J)\cong(V\otimes_E L(\ul{\lambda}_J)')^{\Sigma_p\setminus J-\an} \otimes_E L(\ul{\lambda}_J)\\ \cong
    (V\otimes_E L(\ul{\lambda}_{J'})'\otimes_E L(\ul{\lambda}_{J\setminus J'})')^{\Sigma_p\setminus J-\an}\otimes_E L(\ul{\lambda}_{J'})\otimes_E L(\ul{\lambda}_{J \setminus J'}) \\
    \cong \big((V\otimes_E L(\ul{\lambda}_{J'})')^{\Sigma_p\setminus J'-\an}\otimes_E L(\ul{\lambda}_{J\setminus J'})'\big)^{\Sigma_p\setminus (J\setminus J')-\an}\otimes_E L(\ul{\lambda}_{J'})\otimes_E L(\ul{\lambda}_{J\setminus J'})\\
    \cong \big((V \otimes_E L(\ul{\lambda}_{J'}')^{\Sigma_p\setminus J'-\an}\otimes_E L(\ul{\lambda}_{J'}) \otimes_E L(\ul{\lambda}_{J\setminus J'})'\big)^{\Sigma_p\setminus (J\setminus J')-\an} \otimes_E L(\ul{\lambda}_{J\setminus J'}) \\
    \cong V(\ul{\lambda}_J')(\ul{\lambda}_{J\setminus J'}),
  \end{multline*}
  where the third isomorphism follows from the fact that $L(\ul{\lambda}_{J\setminus J'})$ is locally $\Sigma_p\setminus J'$-analytic, and the fourth from that $L(\ul{\lambda}_{J'})$ is locally $\Sigma_p\setminus (J\setminus J')$-analytic.  (3) follows.
\end{proof}

Following \cite{OS}, to any object $M$ in the BGG category $\co^{\overline{\ub}_{\Sigma_p}}_{\alg}$, and any finite length smooth representation $\pi$ of $T_p$ over $E$, can be associated a locally $\Q_p$-analytic representation $\cF_{\overline{B}_p}^{G_p}(M,\pi)$ of $G_p$ over $E$. The functor $\cF_{\overline{B}_p}^{G_p}(\cdot, \cdot)$ is exact in both arguments, covariant for finite length smooth representations of $T_p$, and contravariant for $\co_{\alg}^{\overline{\ub}_{\Sigma_p}}$ (cf. \cite[Thm]{OS}).

Let $\ul{\lambda}_{\Sigma_p}=(\lambda_{1,\sigma}, \lambda_{2,\sigma})_{\sigma\in \Sigma_p}\in E^{2|\Sigma_p|}$ be an integral weight of $\ft_p$ and $\ul{\lambda}_{\sigma}:=(\lambda_{1,\sigma}, \lambda_{2,\sigma})$, denote by $\overline{M}(\ul{\lambda}_{\Sigma_p}):=\text{U}(\ug_{p,\Sigma_p}) \otimes_{\text{U}(\overline{\ub}_{p,\Sigma_p})} \ul{\lambda}_{\Sigma_p}\cong \otimes_{\sigma\in \Sigma_p} \text{U}(\ug_{p,\sigma}) \otimes_{\text{U}(\overline{\ub}_{p,\sigma})} \ul{\lambda}_{\sigma}=\otimes_{\sigma\in \Sigma_p} \overline{M}_{\sigma}(\ul{\lambda}_{\sigma})$ the Verma module of highest weight $\ul{\lambda}_{\Sigma_p}$, which admits a unique simple quotient denoted by $\overline{L}(\ul{\lambda}_{\Sigma_p})$.  Put
\begin{equation*}\overline{\Sigma}^+(\ul{\lambda}):=\{\sigma\in \Sigma_p\ |\ \lambda_{2,\sigma}\geq \lambda_{1,\sigma}\}.\end{equation*}
 For $\sigma\in \Sigma_p$, $\overline{M}_{\sigma}(\ul{\lambda}_{\sigma})$ is irreducible if $\sigma\notin {\overline{\Sigma}^+(\ul{\lambda})}$; for $\sigma\in {\overline{\Sigma}^+(\ul{\lambda})}$, $\overline{M}_{\sigma}(\ul{\lambda}_{\sigma})$ lies in an exact sequence
 \begin{equation*}
   0 \ra \overline{M}_{\sigma}(s_{\sigma}\cdot \ul{\lambda}_{\sigma}) \ra \overline{L}_{\sigma}(\ul{\lambda}_{\sigma}) \ra \overline{L}(\ul{\lambda}_{\sigma})\ra 0.
 \end{equation*}
We deduce that $\overline{M}(\ul{\lambda}_{\Sigma_p})$ admits a decreasing filtration \big(of objects in $\co_{\alg}^{\overline{\ub}_{\Sigma_p}}$\big):
\begin{equation*}
  0=\Fil^{|{\overline{\Sigma}^+(\ul{\lambda})}|+1} \overline{M}(\ul{\lambda}_{\Sigma_p}) \subset \Fil^{|{\overline{\Sigma}^+(\ul{\lambda})}|} \overline{M}(\ul{\lambda}_{\Sigma_p}) \subset \cdots \subset \Fil^{0} \overline{M}(\ul{\lambda}_{\Sigma_p})=\overline{M}(\ul{\lambda}_{\Sigma_p})
\end{equation*}
such that
\begin{equation*}
  \Fil^{i} \overline{M}(\ul{\lambda}_{\Sigma_p})/\Fil^{i+1} \overline{M}(\ul{\lambda}_{\Sigma_p}) \cong \oplus_{J\subseteq {\overline{\Sigma}^+(\ul{\lambda})}, |J|=i} \overline{L}(s_J \cdot \ul{\lambda}_{\Sigma_p}).
\end{equation*}
In fact, one has
\begin{equation*}\Fil^i \overline{M}(\ul{\lambda}_{\Sigma_p})=\Ima\big(\oplus_{J\subseteq {\overline{\Sigma}^+(\ul{\lambda})}, |J|=i} \overline{M}(s_J \cdot \ul{\lambda}_{\Sigma_p}) \ra \overline{M}(\ul{\lambda}_{\Sigma_p})\big).\end{equation*}
For $J\subseteq {\overline{\Sigma}^+(\ul{\lambda})}$, denote by $\overline{M}_J(\ul{\lambda}_{\Sigma_p}):=\overline{M}(\ul{\lambda}_{\Sigma_p})/\Ima\big(\oplus_{\sigma\in J} \overline{M}(s_{\sigma}\cdot \ul{\lambda}_{\Sigma_p})\ra \overline{M}(\ul{\lambda}_{\Sigma_p})\big)$, which is in fact the maximal $\text{U}(\ug_J)$-finite quotient of $\overline{M}(\ul{\lambda}_{\Sigma_p})$. Indeed, one has
\begin{equation}\label{equ: cclg-jad}
  \overline{M}_J(\ul{\lambda}_{\Sigma_p})\cong \overline{M}(\ul{\lambda}_{\Sigma_p\setminus J}) \otimes_E \overline{L}(\ul{\lambda}_J).
\end{equation}
By \cite[Thm]{OS}, one easily deduces from the above discussion:
\begin{proposition}\label{prop: cclg-cnee}
  Keep the above notation and let $\pi$ be a finite length smooth representation of $T_p$ over $E$,
then the locally $\Q_p$-analytic representation $\cF_{\overline{B}_p}^{G_p}\big(\overline{M}(\ul{\lambda}_{\Sigma_p})^{\vee}, \pi\big)$ admits a decreasing filtration (where ``$\vee$" denotes the dual in the BGG category $\co^{\overline{\ub}_{\Sigma_p}}$) \begin{equation*}\Fil^i\cF_{\overline{B}_p}^{G_p}\big(\overline{M}(\ul{\lambda}_{\Sigma_p})^{\vee}, \pi\big):=\cF_{\overline{B}_p}^{G_p}\big((\Fil^i \overline{M}(\ul{\lambda}_{\Sigma_p}))^{\vee}, \pi\big),\end{equation*}
  with $(\Fil^i/\Fil^{i+1}) \cF_{\overline{B}_p}^{G_p}\big(\overline{M}(\ul{\lambda}_{\Sigma_p})^{\vee}, \pi\big) \cong\oplus_{J\subseteq {\overline{\Sigma}^+(\ul{\lambda})}, |J|=i}  \cF_{\overline{B}_p}^{G_p}\big(\overline{L}(s_J\cdot \ul{\lambda}_{\Sigma_p}), \pi\big)$.
\end{proposition}
Let $\psi$ be a smooth character of $T_p$ over $E$,
we have in fact (cf. (\ref{equ: cclg-idel}))
\begin{equation*}
  I(\psi\delta_{\ul{\lambda}_{\Sigma_p}})\cong  \cF_{\overline{B}_p}^{G_p}(\overline{L}(-\ul{\lambda}_{\Sigma_p}), \psi).
\end{equation*}
The following proposition follows from \cite[Thm. 4.3]{Br13II}:
\begin{proposition}\label{prop: cclg-pve}Let $V$ be a very strongly admissible representation of $G_p$ over $E$, $\pi$ a finite length smooth representation of $T_p$ over $E$, $J\subseteq \overline{\Sigma}^+(-\ul{\lambda})$ (thus $\ul{\lambda}_J$ is dominant with respect to $B_p$), then there exists a bijection
\begin{equation*}
  \Hom_{G_p}\big(\cF_{\overline{B}_p}^{G_p}(\overline{M}_J(-\ul{\lambda}_{\Sigma_p})^{\vee}, \pi), V(\ul{\lambda}_J)\big)
  \xlongrightarrow{\sim} \Hom_{T_p}\big(\pi\otimes_E \delta_{B_p}, J_{B_p}(V(\ul{\lambda}_J))\big).
\end{equation*}
\end{proposition}
\begin{proof}
 By Breuil's adjunction formula \cite[Thm. 4.3]{Br13II}, one has
  \begin{equation*}
   \Hom_{G_p}\big(\cF_{\overline{B}_p}^{G_p}(\overline{M}(-\ul{\lambda}_{\Sigma_p})^{\vee}, \pi\otimes_E \delta_{B_p}^{-1}), V(\ul{\lambda}_{J})\big) \xlongrightarrow{\sim} \Hom_{T_p}\big(\pi\otimes_E \delta_{\ul{\lambda}_{\Sigma_p}}, J_{B_p}(V(\ul{\lambda}_{J}))\big).
  \end{equation*}
However, since $V(\ul{\lambda}_{J})$ is $\text{U}(\ug_{J})$-finite, any map in the left set factors though the quotient (cf. (\ref{equ: cclg-jad})) \begin{equation*}\cF_{\overline{B}_p}^{G_p}\big(\overline{M}_J(-\ul{\lambda}_{\Sigma_p})^{\vee}, \pi\otimes_E \delta_{B_p}^{-1}\big).\end{equation*}
   Indeed, by \cite[Thm]{OS} and the discussion above Proposition \ref{prop: cclg-cnee}, any irreducible constituent of the kernel  $\cF_{\overline{B}_p}^{G_p}(\overline{M}(-\ul{\lambda}_{\Sigma_p})^{\vee}, \pi\otimes_E \delta_{B_p}^{-1})\twoheadrightarrow \cF_{\overline{B}_p}^{G_p}(\overline{M}_J(-\ul{\lambda}_{\Sigma_p})^{\vee}, \pi\otimes_E \delta_{B_p}^{-1})$ would be an irreducible constituent of a locally analytic representation of the form:
  \begin{equation*}
    \cF_{\overline{B}_p}^{G_p}\big(\overline{L}(-s_{J'}\cdot \ul{\lambda}_{\Sigma_p}),\psi\big)\cong I(\psi\delta_{s_{J'}\cdot \ul{\lambda}_{\Sigma_p}})
  \end{equation*}
  where $J'\subseteq \overline{\Sigma}^+(-\ul{\lambda})$, $J'\cap J\neq \emptyset$, and $\psi$ is a smooth character of $T_p$ appearing as an irreducible constituent in $\pi\otimes_E \delta_{B_p}^{-1}$. While $I(\psi\delta_{s_{J'}\cdot \ul{\lambda}_{\Sigma_p}})$ does not have non-zero $\text{U}(\ug_J)$-finite vectors (e.g. by \cite[\S~2]{Sch10} and (\ref{equ: cclg-gpa})), the proposition follows.
\end{proof}

\end{document}